\documentclass[12pt]{amsart}

\usepackage{amssymb,bm,mathrsfs}
\usepackage{mathabx}
\usepackage{enumerate}
\usepackage{tikz}
\usepackage[centering,width=6in]{geometry}
\usepackage{pgfplots}

\usepackage[colorlinks,citecolor=blue,pdfstartview=FitH]{hyperref}
\theoremstyle{plain}
\newtheorem*{theorem*}{Theorem}
\newtheorem{thm}{Theorem}[section]
\newtheorem{prop}[thm]{Proposition}
\newtheorem{lem}[thm]{Lemma}

\newtheorem{cor}[thm]{Corollary}

\theoremstyle{definition}
\newtheorem{defn}[thm]{Definition}
\newtheorem{exmp}[thm]{Example}

\theoremstyle{remark}
\newtheorem{rem}[thm]{Remark}
\numberwithin{equation}{section}

\newcommand{\N}{\mathbb N}
\newcommand{\Q}{\mathbb Q}
\newcommand{\R}{\mathbb R}
\newcommand{\Z}{\mathbb Z}
\newcommand{\C}{\mathbb C}

\def\pb{\boldsymbol{p}}
\def\mb{\boldsymbol{m}}
\def\hb{\boldsymbol{h}}
\def\Qb{\boldsymbol{Q}}

\def\T{\top}

\usepackage{float}

\title{On the Rajchman property of homogeneous self-affine measures}

\author{De-Jun FENG}
\address{Department of Mathematics\\
The Chinese University of Hong Kong\\
Shatin,  Hong Kong\\}
\curraddr{}
\email{\href{mailto:djfeng@math.cuhk.edu.hk}{djfeng@math.cuhk.edu.hk}}
\thanks{}

\author{Tian-Han YI}
\address{Department of Mathematics\\
The Chinese University of Hong Kong\\
Shatin,  Hong Kong\\}
\curraddr{}
\email{\href{mailto:thyi@math.cuhk.edu.hk}{thyi@math.cuhk.edu.hk}}
\thanks{}
\keywords{Self-affine measures, Fourier transform, Rajchman property, Power Fourier decay, Absolute continuity}

\subjclass[2020]{Primary 28A80, 42A16; Secondary, 37C45}

\begin{document}

\begin{abstract}
We provide a  complete algebraic characterization of homogeneous self-affine measures whose Fourier transform does not tend to~$0$ at infinity. Moreover,  we prove that an integral self-affine measure has  power Fourier decay if and only if it is absolutely continuous. For a given integral affine IFS, we also provide an algorithm for determining those probability vectors for which the associated self-affine measure is absolutely continuous.
\end{abstract}

\maketitle

\section{Introduction}

Let $d\geq 1$ be an integer. Given a Borel probability measure $\nu$ on $\mathbb{R}^d$, the Fourier transform of $\nu$ is defined by
\[
\widehat{\nu}(\xi)=\int e^{2\pi i \langle \xi,\;x\rangle}\,d\nu(x),\qquad \xi\in\mathbb{R}^d,
\]
where $\langle\cdot,\cdot\rangle$ denotes the standard inner product on $\mathbb{R}^d$. The measure $\nu$ is said to be a \textit{Rajchman measure} if $|\widehat{\nu}(\xi)|\to 0$ as $|\xi|\to\infty$. The Riemann–Lebesgue lemma says that $\nu$ is Rajchman whenever it is absolutely continuous with respect to the Lebesgue measure. For singular measures, determining which ones are Rajchman is a subtle question with a long history (see \cite{Lyons1995}). In this paper, we study the Rajchman property in the context of homogeneous self-affine measures on $\R^d$.

Before presenting the background and stating our main results, we first introduce some necessary notation and definitions concerning affine iterated function systems, self-affine sets, and self-affine measures. Let $GL_d(\R)$ denote the set of all invertible real $d\times d$ matrices.
By an {\it affine iterated function system} (affine IFS) on $\mathbb{R}^d$, we mean a finite collection
$
\Phi=\{\phi_\ell\}_{\ell=1}^m
$
of contracting affine maps from $\mathbb{R}^d$ to $\mathbb{R}^d$ (in some norm) of the form
\[
\phi_\ell(x)=A_\ell x+a_\ell,\qquad 1\le \ell \le m,
\]
where $A_\ell\in GL_d(\mathbb{R})$,
and $a_\ell\in \mathbb{R}^d$. It is well known \cite{Hutchinson1981} that  there exists a unique nonempty
compact set $K \subset \mathbb{R}^d$ such that
\[
K = \bigcup_{\ell=1}^m \phi_\ell(K).
\]
We call $K$ the {\it self-affine set} generated by $\Phi$. In particular, if all the maps
in $\Phi$ are contracting similitudes, then $K$ is called a {\it self-similar set}.

It is also well known \cite{Hutchinson1981} that, given an affine IFS
\(\Phi=\{\phi_\ell\}_{\ell=1}^m\) and a probability vector
\(\pb=(p_\ell)_{\ell=1}^m\), there exists a unique probability
measure \(\mu\) on \(\mathbb{R}^d\) such that
\[
\mu=\sum_{\ell=1}^m p_\ell\, \mu \circ \phi_\ell^{-1}.
\]
The measure \(\mu\) is called the \emph{self-affine measure} associated
with \(\Phi\) and \(\pb\), and it is supported on the attractor
\(K\). When all the maps in \(\Phi\) are similitudes, \(\mu\) is called
\emph{self-similar}. Moreover, an affine IFS \(\Phi\) is said to be
\emph{homogeneous} if the linear parts of the affine maps in \(\Phi\)
are identical, that is, $A_\ell=A$ for all $1\leq \ell\leq m$; correspondingly, the associated self-affine measure
\(\mu\) is also called \emph{homogeneous}.

It is a standard fact that, for a given $A\in GL_d(\Bbb R)$, a finite family $\{Ax+a_\ell\}_{\ell=1}^m$ of affine maps on $\R^d$ forms a homogeneous affine IFS  if and only if the spectral radius of $A$ is less than $1$; see e.g.~\cite[Theorem~5.6.9 and Lemma~5.6.10]{HornJohnson1985}.

The Fourier transform of a homogeneous self-affine measure can be expressed as the limit of a certain infinite product. To see this, let $\mu$ be the self-affine measure associated with a homogeneous affine IFS $\{Ax+a_\ell\}_{\ell=1}^m$ on $\R^d$ and a probability vector $(p_\ell)_{\ell=1}^m$. Then
\begin{equation}
\label{mu-fourier}
\widehat{\mu}(\xi)
=\prod_{n=0}^\infty H\left((A^\top)^{n}\xi\right), \qquad \xi\in \R^d,
\end{equation}
where $A^\top$ denotes the transpose of $A$ and
\begin{equation}
\label{e-Hxi}
 H(\xi) = \sum_{\ell=1}^{m} p_\ell \exp\!\bigl(2\pi i \langle \xi, a_\ell \rangle\bigr).
\end{equation}
This follows from the fact that $\mu$ can be expressed as an infinite convolution product
$$
 \mu = \bigast_{n=0}^{\infty} \left( \sum_{\ell=1}^{m} p_\ell\,\delta_{A^n a_\ell} \right).
 $$

Following \cite{Hochman2015}, we say that  an affine IFS $\Phi=\{\phi_\ell\}_{\ell=1}^m$ on $\R^d$ is  {\it affinely irreducible} if  there does not exist a proper affine subspace $V$ of $\R^d$ such that
$\phi_\ell(V) = V$ for all $1 \leq \ell \leq m$. It is easy to see that $\Phi$ is not affinely irreducible if and only if its
attractor $K$ is contained in a proper affine subspace $V$ of $\mathbb{R}^d$.
Hence, if $\Phi$ is not affinely irreducible, then any self-affine measure associated with $\Phi$ is non-Rajchman.

\subsection{Background}

Understanding the Rajchman property of fractal measures is a fundamental problem in harmonic analysis, ergodic theory, fractal geometry, and number theory; see e.g.~ the survey papers \cite{Algom2025, Lyons1995, Sahlsten2025}.

The first major advances were obtained for classical Bernoulli convolutions $\nu_\lambda$, namely, the self-similar measures on the real line  associated with the IFS $\{\lambda x, \lambda x+1\}$  and the probability vector $(1/2, 1/2)$, where  $\lambda\in (0,1)$.  Erd\H{o}s \cite{Erdos1939} proved that if $\lambda^{-1}$ is a Pisot number distinct from $2$, then $\nu_{\lambda}$ is not Rajchman. Recall that a {\it Pisot number} is an algebraic integer greater than one,
whose algebraic conjugates are all less than one in modulus. Note that $\nu_{1/2}$ is absolutely continuous and hence Rajchman. Salem \cite{Salem1943} later proved that if $\lambda^{-1}$ is not a Pisot number, then $\nu_{\lambda}$ is Rajchman, thereby providing a complete characterization of Rajchman Bernoulli convolutions. Erd\H{o}s \cite{Erdos1940} proved that  $\nu_\lambda$ has power Fourier decay for $\mathcal L$-a.e.~$\lambda\in (0,1)$.  Later, Kahane \cite{Kahane1971} pointed out that this actually holds for all $\lambda\in (0,1)$ outside a set of zero Hausdorff dimension.  Recall that a probability measure $\mu$ on $\R^d$ is said to have {\it power Fourier decay} if there exists $\alpha>0$ such that $|\widehat{\mu}(\xi)|=O(|\xi|^{-\alpha})$ as $|\xi|\to\infty$.

Now we turn to the general orientation preserving self-similar measures on the real line.  Based on a new method relying on renewal theory originating in \cite{Li2018}, Li and Sahlsten \cite{LiSahlsten2022} proved that if $\mu$ is the self-similar measure  associated with an IFS  $\{r_\ell x+ a_\ell\}_{\ell=1}^m$ on the line and a strictly positive probability vector $(p_\ell)_{\ell=1}^m$, where $r_\ell>0$ for all $1\leq \ell\leq m$ and there exist $i\neq j$ such that $\log r_i/\log r_j$ is irrational, then $\mu$ is Rajchman. Moreover, they established logarithmic decay of the Fourier transform under an appropriate Diophantine condition.

 Br\'{e}mont \cite{Bremont2021}  further investigated the commensurable case when $\log r_i/\log r_j$ are rational for all $1\leq i,j\leq m$, and obtained an almost complete characterization of non-Rajchman orientation preserving self-similar measures on the line. He proved that, for an IFS
$
\Phi=\{r_\ell x+a_\ell\}_{\ell=1}^m
$
on the line, with $r_\ell>0$ for every $\ell$ and with a nonsingleton attractor, there exists a strictly positive probability vector
$
\pb=(p_\ell)_{\ell=1}^m
$
such that the self-similar measure $\mu$ associated with $\Phi$ and $\pb$ is non-Rajchman if and only if $\Phi$ is  a {\it Pisot IFS} on $\R$, in the sense that there exists a Pisot number $\beta$ such that
\[
r_\ell=\beta^{-n_\ell}
\qquad\text{for all }1\leq \ell\leq m,
\]
where $n_\ell\in\mathbb{N}$, and such that, after conjugating $\Phi$ by a suitable similarity, the translation parameters satisfy
\[
a_\ell\in\mathbb{Q}(\beta)
\qquad\text{for all }1\leq \ell\leq m,
\]
where $\mathbb{Q}(\beta)$ denotes the smallest field containing the field $\Q$ of rational numbers and the element $\beta$.

Br\'{e}mont also showed that for a Pisot IFS $\Phi$ on $\R$, the corresponding self-similar measure $\mu_{\pb}$ is non-Rajchman for every strictly positive probability vector $\pb$ outside a finite
union of real-analytic graphs of dimension at most $m-2$ (which are
points when $m=2$). Moreover, he proved that, in this setting, $\mu_{\pb}$ is absolutely continuous whenever it is Rajchman. As a related result, Varj\'{u} and Yu \cite{VarjuYu2022} gave a complete characterization of self-similar sets on \(\mathbb{R}\) with positive contraction ratios that are sets of uniqueness.

We remark that Solomyak~\cite{Solomyak2021} proved the existence of a set
\(E \subset (0,1)^m\) of Hausdorff dimension zero such that, whenever
\(\Phi=\{r_\ell x+a_\ell\}_{\ell=1}^m\) is an affinely irreducible IFS on
\(\mathbb{R}\) with \((r_\ell)_{\ell=1}^m\notin E\), the self-similar measure
\(\mu\) associated with \(\Phi\) and any strictly positive probability vector
has power Fourier decay. However, this exceptional set is not explicit.
At present, very few concrete examples of self-similar measures are known
to exhibit power Fourier decay; see \cite{DaiFengWang2007, Streck2026}.

Rapaport \cite{Rapaport2022} further extended the non-Rajchman characterization of Li and Sahlsten \cite{LiSahlsten2022} and Br\'{e}mont \cite{Bremont2021} to arbitrary self-similar IFSs on $\R^d$.   Denote the orthogonal group of $\R^d$ by $O(d)$. Rapaport proved  the following result.

 \begin{thm}[\cite{Rapaport2022}]
 \label{thm-Rapaport}
Let
$
\Phi=\{\phi_\ell(x)=r_\ell U_\ell x+a_
\ell\}_{\ell=1}^{m}
$
be an affinely irreducible self-similar IFS on $\mathbb{R}^d$, where
$0<r_\ell<1$, $U_\ell\in O(d)$, and $a_\ell\in\mathbb{R}^d$ for
$1\le  \ell\leq m$.
Then there exists a strictly positive probability vector
$\pb=(p_\ell)_{\ell=1}^{m}$ such that the self-similar measure associated
with $\Phi$ and $\pb$ is non-Rajchman if and only if there exist a linear
subspace $V\subset \mathbb{R}^d$, with
$d':=\dim V>0$, satisfying $U_\ell(V)=V$ for all $1\le \ell\le m$,
and an isometry $S:V\to\mathbb{R}^{d'}$ such that the following
conditions hold.

\begin{itemize}
\item[(1)] For $1\le  \ell\leq m$, let $U_\ell'\in O(d')$ and
$a_\ell'\in\mathbb{R}^{d'}$ be defined by
\[
S\circ P_V\circ\phi_\ell\circ S^{-1}(x)
= r_\ell U_\ell'x+a_\ell',
\]
where $P_V$ denotes the orthogonal projection from $\R^d$ onto $V$.
Let $\boldsymbol{H}\subset GL_{d'}(\mathbb{R})$ be the group generated by
$\{r_\ell U_\ell'\}_{\ell=1}^{m}$, and let
$
\boldsymbol{N}:=\boldsymbol{H}\cap O(d').
$
Then $\boldsymbol{N}$ is finite, $\boldsymbol{N}\triangleleft \boldsymbol{H}$, and $\boldsymbol{H}/\boldsymbol{N}$ is cyclic.

\item[(2)] For every contracting matrix $A\in \boldsymbol{H}$ such that
$
\{A^n\boldsymbol{N}\}_{n\in\mathbb{Z}}=\boldsymbol{H}/\boldsymbol{N},
$
there exist $k\ge1$, $\beta_1,\ldots,\beta_k\in\mathbb{C}$, and
$\zeta_1,\ldots,\zeta_k\in\mathbb{C}^{d'}\setminus\{0\}$ such that:

\begin{itemize}
\item[(a)] $\{\beta_1,\ldots,\beta_k\}$ is a Pisot tuple (see Definition~\ref{defn-1.1});

\item[(b)]
$
A^{-1}\zeta_j=\beta_j\zeta_j$ for all $1\le j\le k$;

\item[(c)] for every $1\le \ell \le m$ and every $B\in \boldsymbol{N}$, there exists a
polynomial $P_{\ell,B}\in\mathbb{Q}[X]$ such that
$
\langle Ba_\ell',\zeta_j\rangle
= P_{\ell,B}(\theta_j)
$ for all $1\le j\le k$,  where $\langle\cdot,\cdot\rangle$ denotes the standard inner product on $\C^d$, that is, \[ \langle z,w\rangle=\sum_{j=1}^d z_j\overline{w_j} \qquad (z,w\in\mathbb{C}^d). \]
\end{itemize}
\end{itemize}
\end{thm}

The above theorem provides a complete algebraic characterization of
self-similar IFS on $\mathbb{R}^d$ for which there exists a strictly
positive probability vector giving rise to a non-Rajchman self-similar
measure. Nevertheless, for such a self-similar IFS, it does not identify the
precise strictly positive probability vectors that give rise to
non-Rajchman self-similar measures. As pointed out by Br\'{e}mont
(see \cite[Remark~2.10]{Bremont2021}), it is important to
determine all strictly positive probability vectors giving rise to
non-Rajchman self-similar measures.

Next we turn to self-affine measures. In \cite{LiSahlsten2020}, Li and Sahlsten established the Rajchman property of non-homgeneous self-affine measures under some natural algebraic hypotheses.  For an affine IFS $\Phi=\{A_\ell x+a_\ell\}_{i=1}^m$ on $\R^d$,   assuming that the multiplicative group $\Gamma$ generated by  $\{A_\ell\}_{\ell=1}^m$ is proximal and totally irreducible,
and that the attractor of $\Phi$ is not a singleton, they  proved that all self-affine measures
associated with $\Phi$ and strictly positive probability vectors are Rajchman. Under additional assumptions on the group $\Gamma$, they also established power Fourier decay.

The result of Li and Sahlsten \cite{LiSahlsten2020} does not apply to homogeneous self-affine measures, since the group generated by the linear part of a homogeneous affine IFS cannot simultaneously satisfy proximality and total irreducibility. By extending the Erd\H{o}s--Kahane argument to higher-dimensional diagonal affine systems, Solomyak \cite{Solomyak2022} proved that for every contracting diagonal matrix $A\in GL_d(\mathbb{R})$ outside an exceptional set of Lebesgue measure zero, every self-affine measure associated with a homogeneous affine IFS $\{Ax+a_\ell\}_{\ell=1}^m$ and a strictly positive probability vector exhibits power Fourier decay, provided that the attractor of the IFS is not contained in any hyperplane of $\mathbb{R}^d$. Nevertheless, this exceptional set is not explicit. Solomyak further observed that, by combining the classical techniques of Pisot \cite{Pisot1938} and Salem \cite{Salem1943} with ideas from \cite[Section~5]{Rapaport2022}, one can show that if a homogeneous self-affine measure is non-Rajchman and the corresponding IFS $\{Ax+a_\ell\}_{\ell=1}^m$ is affinely irreducible, then the spectrum of $A^{-1}$ contains a Pisot tuple.

Very recently,  Fu, Wen and Zhu \cite{FuWenZhu2026} provided a necessary and sufficient condition for a homogeneous self-affine measure to be non-Rajchman.  They proved that the self-affine measure $\mu$ associated with a homogeneous IFS $\{Ax+a_\ell\}_{\ell=1}^m$ and a probability vector $\{p_\ell\}_{\ell=1}^m$ is non-Rajchman if and only if there exists a nonzero vector $\xi \in \mathbb{R}^d$ such that
\[ \sum_{n=0}^{\infty} \bigl(1 - \big|H\left((A^\top)^{-n}\xi\right)\big|\bigr) < \infty, \quad and \;\; \left|H\left((A^\top)^{k}\xi\right)\right| \neq 0 \quad \text{for all } k \in \Z, \] where $H$ is defined as in \eqref{e-Hxi}. Although this condition is difficult to verify directly in practice, they also provided several explicit examples of non-Rajchman homogeneous self-affine measures.

As related results, the Rajchman property and power Fourier decay for self-conformal measures have been verified under certain conditions; see \cite{AlgomChangWuWu2025, AlgomHertzWang2021, AlgomHertzWang2023, AlgomHertzWang2026, BakerBanaji2025, BakerSahlsten2023, JordanSahlsten2016, SahlstenStevens2024} and the references in \cite{Sahlsten2025}. In addition, Algom, Rodriguez Herz, and Wang proved that every Rajchman self-similar measure on $\R$  is pointwise normal in every integer base; see \cite[Theorem~1.4]{AlgomHertzWang2021}.

\subsection{Statement of our main results}
We first introduce some definitions needed to state our main results.
Following \cite{Cantor1962} and \cite[Section 9.2]{Bertin1992},  we make the following definition.

\begin{defn}\label{defn-1.1}
    A finite collection $\{\beta_1,\ldots,\beta_k\}$ of distinct algebraic integers (real or complex) is said to be a {\it Pisot $k$-tuple}, or simply a {\it Pisot tuple},  if the following two conditions are satisfied:
    \begin{itemize}
        \item[(i)] $|\beta_j|>1$ for $1\leq j \leq k$;
        \item[(ii)] there exists a monic integer polynomial $P\in \mathbb{Z}[x]$ so that $P(\beta_j)=0$ for all $1\leq j \leq k$, and every other root $z$ of $P$, if any, satisfies $|z|<1$.
    \end{itemize}
    Furthermore, we call $\{\beta_1,\ldots,\beta_k\}$ a {\it Pisot tuple of type~I} if, in addition, the above polynomial $P$ has at least one root in the open unit disk $\{z\in \C\colon |z|<1\}$; otherwise we call it  a {\it Pisot tuple of type~II}.
 \end{defn}

Moreover, inspired by the work of Br\'{e}mont \cite{Bremont2021} and Rapaport \cite{Rapaport2022}, we give the following definition.

\begin{defn}\label{defn-1.2}
    A homogeneous self-affine IFS
\[
\Phi = \{\phi_\ell(x)=Ax + a_\ell\}_{\ell=1}^m
\]
on $\mathbb{R}^d$ is said to be a {\it generalized Pisot IFS} if there exist
$k \geq 1$, $\beta_1,\ldots,\beta_k \in \mathbb{C}$, and
$\zeta_1,\ldots,\zeta_k \in \mathbb{C}^d \setminus \{0\}$ such that the following three
conditions are satisfied:
\begin{itemize}
    \item[(i)] $\{\beta_1,\ldots,\beta_k\}$ is a Pisot tuple;
    \item[(ii)] $A^\T\zeta_j = \beta_j^{-1} \zeta_j$ for $1 \leq j \leq k$, where $A^\T$ denotes the transpose of $A$;
    \item[(iii)] for every $1 \leq \ell \leq m$, there exists $P_\ell \in \mathbb{Z}[x]$
    such that
    \[
        \langle \zeta_j, a_\ell - a_1 \rangle = P_\ell(\beta_j)
        \quad \text{for } 1 \leq j \leq k.
    \]

\end{itemize}
Furthermore, we say that $\Phi$ is a {\it generalised Pisot IFS of type~I} if,
in addition, $\{\beta_1,\ldots,\beta_k\}$ is a Pisot tuple of type~I.  For completeness, we say that $\Phi$ is a {\it generalized Pisot IFS of type II} if it is a generalized Pisot IFS but not of type I.
\end{defn}
\begin{rem}
It can be checked directly that a homogeneous IFS $$\Phi=\{\beta^{-1}x+a_\ell\}_{\ell=1}^m$$ on $\mathbb{R}$ is a generalized Pisot IFS of type~I if and only if $|\beta|>1$ is a non-integer Pisot number, and $S(a_\ell)\in \mathbb{Q}(|\beta|)$ for some invertible affine map $S(x)=cx+d$ and all $1\le \ell \le m$. Moreover, $\Phi$ is a generalized Pisot IFS on $\R$ of type~II if and only if $|\beta|>1$ is an integer, and $S(a_\ell)\in \mathbb{Z}$ for some invertible affine map $S(x)=cx+d$ and all $1\le \ell \le m$.\end{rem}
For convenience, we introduce two more definitions.
\begin{defn}
A homogeneous affine IFS
\[
\Phi = \{\phi_\ell(x)=Ax + a_\ell\}_{\ell=1}^m
\]
on $\mathbb{R}^d$ is said to be an {\it integral affine IFS} if $A^{-1}$ is an integer matrix and $a_\ell \in \mathbb{Z}^d$ for all $1 \leq \ell \leq m$. Every self-affine measure associated with an integral affine IFS is called an {\it integral self-affine measure}.
\end{defn}

\begin{defn} Let $\Phi = \{\phi_\ell(x)=Ax + a_\ell\}_{\ell=1}^m$ be an homogeneous affine IFS
on $\mathbb{R}^d$.    An IFS $\Psi=\{\psi_\ell\}_{\ell=1}^m$ on $\R^q$ (with $1\leq q\leq m$) is said to be a {\it factor} of $\Phi$, if there exists
a nonzero $A^\top$-invariant linear subspace $V\subset \R^d$ with $\dim V=q$ such that the projected IFS
\begin{equation}
\label{e-phiV}
\Phi_V := \{\phi_{V,\ell}(x)=P_V A x + P_V a_\ell\}_{\ell=1}^m
\end{equation}
on $V$ is conjugate to $\Psi$  via an invertible affine map $T$ from $V$ to $\mathbb{R}^q$, i.e., $$\psi_\ell=T\circ \phi_{V,\ell}\circ T^{-1}\quad\text{for all }1\leq \ell\leq m,$$
where $P_V$ denotes the orthogonal projection from $\R^d$ onto $V$.
 If, in addition, $\Psi$ is an integral IFS, we say that $\Psi$ is an {\it integral factor} of $\Phi$.
\end{defn}

Now we are ready to state the first result of this paper.

\begin{thm}
\label{thm-main1}
Let  $\Phi=\{\phi_\ell(x)=Ax+a_\ell\}_{\ell=1}^m$ be  an affinely irreducible  homogeneous affine IFS on $\R^d$.
Then there exists a strictly positive probability vector $\pb =(p_\ell)_{\ell=1}^m$  such that the self-affine measure associated with $\Phi$ and $\pb$ is non-Rajchman if and only if $\Phi$ is a generailzed Pisot IFS.
\end{thm}

 This result is an analogue of Theorem~\ref{thm-Rapaport} in the context of homogeneous affine IFSs. However, for a given generalized Pisot IFS, it does not specify which probability vectors $\pb$ lead to a non-Rajchman self-affine measure.  This shortcoming is remedied in our second result:

\begin{thm}
\label{thm-main}
Let $\mu$ be the self-affine measure associated with an affinely irreducible homogeneous affine IFS $\Phi=\{\phi_\ell(x)=Ax+a_\ell\}_{\ell=1}^m$ on $\mathbb{R}^d$ and a strictly positive probability vector $\pb=(p_\ell)_{\ell=1}^m$.  Then $\mu$ is non-Rajchman if and only if at least one of the following two conditions holds:
\begin{itemize}
\item[(i)] $\Phi$ is a generalized Pisot IFS of type~I.
\item[(ii)] $\Phi$ has an integral factor $\Psi$ such that the self-affine measure associated with $\Psi$ and $\pb$ is singular.
\end{itemize}
\end{thm}

Theorem~\ref{thm-main} reveals a dichotomy between affinely irreducible generalized Pisot IFSs of types~I and II. For systems of type~I, non-Rajchman behavior occurs automatically for every strictly positive probability vector $\pb$. By contrast, for systems of type~II, non-Rajchman behavior can arise only through a singular integral factor. Consequently, the theorem provides an essentially complete characterization of non-Rajchman homogeneous self-affine measures.
We remark that every type~II system admits at least one integral factor (see Lemma~\ref{lem-5.6'}), whereas a type~I system may also admit an integral factor (see Example~\ref{exmp-8.5}).

Applying Theorem~\ref{thm-main} in the case $d=1$, we obtain the following characterization. Let $\mu$ be the self-similar measure associated with an affinely irreducible IFS
$
\{\beta^{-1}x+a_\ell\}_{\ell=1}^m
$
on $\R$ and a strictly positive probability vector $\pb$. Then $\mu$ is non-Rajchman if and only if one of the following conditions holds:
\begin{enumerate}
\item[(a)] $|\beta|$ is a non-integer Pisot number, and there exists an invertible affine map $S(x)=cx+d$ such that
$
S(a_\ell)\in \Q(|\beta|)
$ for all $1\le \ell\le m.$

\item[(b)] $|\beta|>1$ is an integer, and there exists an invertible affine map $S(x)=cx+d$ such that
$
S(a_\ell)\in \Z
$
for all $1\le \ell\le m,$
and $\mu$ is singular.
\end{enumerate}

As a related earlier result, Lau, Ngai, and Rao~\cite[Theorem~5.1]{LNR2001} proved that if $\beta$ is a non-integer Pisot number and $a_\ell\in\Q$ for all $1\le \ell\le m$, then every self-similar measure associated with the IFS
$
\{\beta^{-1}x+a_\ell\}_{\ell=1}^m
$
on $\R$ is non-Rajchman, and hence singular.

It is known that an integral self-affine measure is absolutely continuous if and only if it is Rajchman; see Proposition~\ref{pro-self-affine} for details. Moreover, there exist effective algorithms for determining the absolute continuity of integral self-affine measures on $\R^d$; see \cite[Theorem~1.3]{LNR2001} and \cite[Theorem~1]{Protasov2000} for the case $d=1$, and \cite[Theorems~7.5, 1.7, and~1.8]{FengLoShen2020} for the general case $d\ge1$.

We remark that there is no straightforward analogue of Theorem~\ref{thm-main} for non-homogeneous self-affine measures. Indeed, Example~\ref{exmp-1} provides a rather unexpected example showing that the self-similar measure associated with the IFS
\[
\{\beta^{-1}x,\,-\beta^{-1}x+1\},
\]
where $\beta=(\sqrt5+1)/2$ is a non-integer Pisot number, together with the probability vector $(\beta^{-1},\beta^{-2})$, is absolutely continuous and hence Rajchman.

Interestingly, the density of this self-similar measure is piecewise constant (though not globally constant) and strictly positive on its support. This answers affirmatively a question posed by Chun-Kit Lai (private communication) concerning the existence of such a self-similar measure. Moreover, it yields the first example of a self-similar measure that admits a Fourier frame but is not a spectral measure; see Definition~\ref{defn-Fourier frame} and the discussion in Example~\ref{exmp-1}. We further remark that the self-similar measure associated with the same IFS and the weight vector $(\beta^{-2},\beta^{-1})$ is also absolutely continuous. However, in this case its density is not strictly positive on its support.

To state our next result, we introduce a generalization of the classical $L^q$-dimension.

\begin{defn}
\label{defn-1.9}
Let $\nu$ be a Borel probability measure on $\R^d$ with compact support. Let $A\in GL_d(\R)$ have spectral radius strictly less than $1$.
For $q>0$, the {\it $L^{(A,q)}$-dimension} of $\nu$ is defined by
$$
\tau_A(\nu,q)=\liminf_{n\to \infty}\frac{d}{n\log |\det(A)|}\log\left(\sum_{Q\in {\mathcal D}}\nu(A^nQ)^q\right),
$$
where ${\mathcal D}$ denotes the standard partition $\{[0,1)^d+z:\; z\in \Z^d\}$ of $\R^d$.
 \end{defn}
It is readily verified that $\tau_A(\nu,q)$ coincides with the classical $L^q$-dimension of $\nu$ whenever all eigenvalues of $A$ have the same modulus.

 Our third result states that for an affinely irreducible integral self-affine measure,  absolute continuity is equivalent to having power Fourier decay. It is also equivalent to the $L^{(A,2)}$-dimension being equal to $d$, where $A$ is the linear part of the homogeneous affine IFS associated with the measure.

\begin{thm}
\label{thm-1.6}
 Let $\mu$ be the self-affine measure  associated with an affinely irreducible integral affine IFS $\Phi=\{Ax+a_\ell\}_{\ell=1}^m$ on $\R^d$ and a strictly positive probability vector $\pb=(p_\ell)_{\ell=1}^m$.   Then the following conditions are equivalent.
\begin{itemize}
    \item[(i)] $\mu$ is absolutely continuous with respect to $\mathcal L^d$;
    \item[(ii)] $\mu$ has power Fourier decay;
    \item[(iii)] $\tau_A(\mu,2)=d$.
\end{itemize}
\end{thm}

The one-dimensional case $d=1$ of Theorem~\ref{thm-1.6} is known; see, for example, \cite[Theorem~1.1]{DaiFengWang2007} for the equivalence between (i) and (ii), and see \cite[Theorem~1.3]{LNR2001} for the equivalence between (i) and (iii).

It is worth emphasizing that Theorem~\ref{thm-1.6} also provides a method for constructing concrete examples of non-integral self-affine measures with power Fourier decay. Indeed, suppose that $\mu$ is an absolutely continuous self-affine measure associated with an integral IFS
\[
\{Ax+a_\ell\}_{\ell=1}^m
\]
on $\R^d$. Then, for any $A^\top$-invariant linear subspace $V\subset\R^d$, the pushforward measure $P_V\mu$ is a self-affine measure associated with the projected IFS $\Phi_V$ defined in \eqref{e-phiV}. Consequently, $P_V\mu$ has power Fourier decay.

By Theorem~\ref{thm-1.6}, every Rajchman integral self-affine measure is absolutely continuous. Consequently, its support has positive Lebesgue measure. However, the converse fails in general. Indeed, an integral affine IFS whose attractor has positive Lebesgue measure may admit no Rajchman self-affine measures; see Example~\ref{exmp-nonRajchman}.

We remark that the $L^{(A,2)}$ dimension of an integral self-affine measure is computable. In Proposition~\ref{prop-8.1}, we provide an algorithm for computing this dimension, adapted from the proof of \cite[Proposition~4.1]{Akiyama2020}. Moreover, for a given integral affine IFS, Proposition~\ref{prop-8.1} also yields an algorithm  for determining those probability vectors for which the associated self-affine measure is absolutely continuous; see Remark~\ref{rem-8.2}.

\subsection{About the proofs} The proof for the equivalence between (i) and (ii) in Theorem~\ref{thm-1.6}  extends an idea from the proof of  \cite[Theorem~1.6]{DaiFengWang2007}, whereas the proof of the equivalence between (i) and (iii) relies on a result of Deng, He, and Lau \cite{DengHeLau2008} concerning vector representations of integral self-affine measures, as well as on the thermodynamic formalism for matrix products developed in \cite{Feng2004,FengLau2002}.

Theorems~\ref{thm-main1} and \ref{thm-main} are consequences of  the following three theorems.
\begin{thm}
\label{thm-1.4}
Let  $\Phi=\{\phi_\ell(x)=Ax+a_\ell\}_{\ell=1}^m$ be  an affinely irreducible  generalized Pisot IFS on $\R^d$.
Then there exists a probability vector $\pb =(p_\ell)_{\ell=1}^m>0$  such that the self-affine measure associated with $\Phi$ and $\pb$ is non-Rajchman.
\end{thm}

\begin{thm}
\label{thm-1.4'}
Let \(\Phi = \{\phi_\ell(x)=Ax+a_\ell\}_{\ell=1}^m\) be an affinely irreducible homogeneous affine IFS on \(\mathbb{R}^d\), and let $K$ be the attractor of $\Phi$.
Suppose there exists a probability vector \(\pb = (p_\ell)_{\ell=1}^m > 0\) such that the self-affine measure \(\mu\) associated with \(\Phi\) and \(\pb\) is non-Rajchman. Then \(\Phi\) is a generalized Pisot IFS. Moreover, at least one of the following two cases will occur:
\begin{itemize}
\item[(i)]  \(\Phi\) is a generalized Pisot IFS of type~I;
\item[(ii)] $\Phi$ has an integral factor $\Psi$ such that the self-affine measure associated with $\Psi$ and $\pb$ is non-Rajchman.

\end{itemize}
\end{thm}

\begin{thm}
\label{thm-1.5}
Let  $\Phi=\{\phi_\ell\}_{\ell=1}^m$ be an affinely irreducible  generalized Pisot IFS of type I.  Then for every strictly positive probability vector  $\pb=(p_\ell)_{\ell=1}^m$, the self-affine measure corresponding to $\Phi$ and $\pb$ is non-Rajchman.
\end{thm}

\begin{proof}[Proof of Theorem~\ref{thm-main1}] It follows immediately from Theorems~\ref{thm-1.4} and \ref{thm-1.4'}.
\end{proof}

\begin{proof}[Proof of Theorem~\ref{thm-main}] It follows immediately from Theorems~\ref{thm-1.4'} and \ref{thm-1.5}, together with the known fact that an integral self-affine measure is non-Rajchman if and only if it is singular (see Proposition~\ref{pro-self-affine}).
\end{proof}

Below, we briefly outline the main ideas underlying the proofs of Theorems~\ref{thm-1.4}--\ref{thm-1.5}.

The proof of Theorem \ref{thm-1.4} extends an idea of Erd\H{o}s \cite{Erdos1939}.  For Theorem~\ref{thm-1.4'},  we first adapt ideas of Salem \cite{Salem1943}, Rapaport \cite{Rapaport2022}, and Solomyak \cite{Solomyak2022} to show that, under the assumptions of that theorem, there exist $\xi \in \mathbb{R}^d \setminus \{0\}$ and $\varepsilon>0$ such that  $|\widehat{\mu}((A^\top)^{-n}\xi)|\geq \varepsilon$ for every integer $n\geq 0$, and
\[
\sum_{n=0}^{\infty}
\bigl\| \langle (A^\top)^{-n}\xi,\, a_\ell-a_1\rangle \bigr\|_{\mathbb{Z}}^2 < \infty
\qquad \text{for } 1 \le \ell \le m.
\]
Using this, we then complete the proof of  Theorem~\ref{thm-1.4'} by analysing the linear-algebraic properties of $(A^\top)^{-n}\xi$ and applying a result of Pisot \cite{Pisot1938} and K\"{o}rnyei \cite{Kornyei1987} (see Theorem~\ref{thm-Pisot}) on the characterization of Pisot tuples. The argument in this final step is rather long and delicate.

For Theorem~\ref{thm-1.5}, we argue by contradiction. Suppose, contrary to the conclusion of the theorem, that $\mu$ is Rajchman. Extending an argument of Salem \cite{Salem1943} to the present setting, we obtain that for every $n\in \mathbb{N}$, there exists $\tau \in \mathbb{Z}$ such that
\begin{equation}
\label{e-ourkey}
\sum_{\ell=1}^m p_\ell
\exp\!\left(2\pi i
\sum_{j=1}^k \beta_j^\tau (1+\beta_j^n)\, P_\ell(\beta_j)
\right)=0,
\end{equation}
where $\beta_1,\ldots, \beta_k$ and $P_\ell$ ($1\leq \ell\leq m$) are the parameters introduced in Definition~\ref{defn-1.2}. More precisely, $\{\beta_1,\ldots, \beta_k\}$ is a Pisot tuple of type~I, and each $P_\ell$ ($1\leq \ell\leq m$) is a monic integer polynomial. Thus,
$(\beta_1,\ldots, \beta_k)$  is a common zero of infinitely many analytic functions. In the generality considered here, however, it appears difficult to describe the structure of the zero sets of these functions.

To derive a contradiction from \eqref{e-ourkey}, we employ several new ideas and techniques.
A key ingredient is Proposition~\ref{prop-4.1}, which asserts that, under certain additional mild assumptions, if
\[
v_n=\left(\exp\left(2\pi i\sum_{j=1}^k\beta_j^{n} P_1(\beta_j)\right),\ldots,
\exp\left(2\pi i\sum_{j=1}^k\beta_j^{n} P_m(\beta_j)\right)\right),\qquad n\in \mathbb{Z},
\]
then, for every subset $\Lambda\subset\mathbb{N}$ of positive upper density, one has
one has
\[
\operatorname{span}_{\mathbb{C}}\{v_n:n\in\Lambda\}
=
\operatorname{span}_{\mathbb{C}}\{v_{-n}:n\in\Lambda\}
=
\mathbb{C}^m.
\] The proof of this proposition is delicate and uses results from Galois theory, ergodic theory, and complex analysis.

Applying Proposition~\ref{prop-4.1}, we show that, under some additional mild assumptions, there exist $\tau \in \mathbb{Z}$ and a subset $\Lambda\subset \mathbb{N}$ of positive upper density such that \eqref{e-ourkey} holds for all $n\in \Lambda$; see Lemma~\ref{lem-4.4} for details. This implies that the span of $\{v_n: n\in \Lambda+\tau\}$ over $\mathbb{C}$ is $\mathbb{C}^m$, which in turn implies that $p_\ell=0$ for all $1\leq \ell\leq m$, leading to a contradiction. Finally, by replacing $\Phi$ with a suitable iterate of $\Phi$, we can remove these additional assumptions.

\subsection{Organization of the paper}
The structure of the paper is as follows. In
Section~\ref{S-2}, we introduce the notation and several preliminary results. In Section~\ref{S-3}, we give some basic properties of affinely irreducible generalized Pisot IFSs. The proofs of Theorems~\ref{thm-1.4}, \ref{thm-1.4'} and \ref{thm-1.5} are given in Sections~\ref{S-4}, \ref{S-5} and \ref{S-6}, respectively. In Section ~\ref{S-7}, we prove Theorem~\ref{thm-1.6}. In Section~\ref{S-8'}, we provide an algorithm to compute the $L^{(A,2)}$ dimension of integral self-affine measures. In Section~\ref{S-8}, we give several examples.
\subsection{Acknowledgements}

 This research was partially supported by the General Research Fund grant (projects CUHK14303524 and CUHK14310625) from the
Hong Kong Research Grant Council, and by a direct grant for research from the Chinese University
of Hong Kong. The authors are grateful to Zhou Feng for providing numerical estimates for the spectral radius of the matrix $M(p)$ defined in the discussion of Example~\ref{exmp-1}, for values of $p$ ranging over   $(0,1)$. Copilot was used to check English grammar and to generate Figure~\ref{Fig1}.

\section{Preliminaries}
\label{S-2}

\subsection{Notation}
For a subset $D$ of $\C^d$ and a subgroup $\Bbb G$ of $\C$,  let $\operatorname{span}_{\mathbb{G}}(D)$ denote the linear span of $D$ over $\mathbb{G}$. That is,
\[
\operatorname{span}_{\mathbb{G}}(D)
= \left\{ \sum_{k=1}^n g_k d_k \;:\; n \ge 1,\; g_k \in \mathbb{G},\; d_k \in D \text{ for all } k=1,\dots,n \right\}.
\]

For $u=(u_1,\ldots, u_d)\in \R^d$ or $\C^d$, let $\|u\|$ be the Euclidean norm of $u$, i.e.,
$$
\|u\|=\left(\sum_{j=1}^d |u_j|^2\right)^{1/2}.
$$
Moreover, let $\|u\|_{\Z^d}$ denote the distance between $u$ and $\Z^d$, i.e.,
$$
\|u\|_{\Z^d}=\min\{\|u-z\|\colon z\in \Z^d\}.
$$
When $d=1$, we simply write $\|u\|_{\Z}$ for  $\|u\|_{\Z^1}$.

\subsection{Several simple facts}

We start with the following elementary lemma.
\begin{lem}
\label{lem-infinity}
Let $\ell, d \in \mathbb{N}$ with $\ell \le d$. Let
$\lambda_1,\ldots,\lambda_\ell \in \mathbb{C}$ be such that
$|\lambda_j|>1$ for at least one $j \in \{1,\ldots,\ell\}$.
Suppose that $v_1,\ldots,v_\ell \in \mathbb{C}^d$ are linearly independent
over $\mathbb{C}$. Then
\[
\lim_{n\to +\infty}
\left\|\sum_{j=1}^\ell \lambda_j^n v_j \right\|
= +\infty.
\]
\end{lem}
\begin{proof}
Choose $v_{\ell+1}, \ldots, v_d$ so that $\{v_1,\ldots, v_d\}$ forms a basis of
$\mathbb{C}^d$. Set $$\lambda_{\ell+1}=\cdots=\lambda_d=0,$$ and let
$A=(v_1,\ldots, v_d)$. Clearly, $A\in GL_d(\mathbb{C})$, and hence
\[
\|Av\|\geq \|A^{-1}\|^{-1}\|v\| \quad \text{for all } v\in \mathbb{C}^d.
\]
Notice that
\[
\sum_{j=1}^\ell \lambda_j^n v_j
= A
\begin{pmatrix}
\lambda_1^n\\
\vdots\\
\lambda_d^n
\end{pmatrix}.
\]
Consequently,
\[
\left\|\sum_{j=1}^\ell \lambda_j^n v_j \right\|
\geq \|A^{-1}\|^{-1}
\left\|
\begin{pmatrix}
\lambda_1^n\\
\vdots\\
\lambda_d^n
\end{pmatrix}
\right\|
= \|A^{-1}\|^{-1}\left(\sum_{j=1}^d |\lambda_j|^{2n}\right)^{1/2}
\longrightarrow +\infty
\]
as $n\to +\infty$.
\end{proof}

The second lemma is standard in algebraic number theory.
\begin{lem}
\label{lem-Galois}
Let $\beta_1,\ldots, \beta_k$ be the roots of a monic integer polynomial $P\in \Z[x]$ of degree $k$. Then the following properties hold:
\begin{itemize}
\item[(i)] Let $g\in \Z[x_1,\ldots, x_k]$ be a symmetric integer polynomial in $x_1,\ldots, x_k$. Then $g(\beta_1,\ldots, \beta_k)\in \Z$.
\item[(ii)]
 Let $c_j\in \Q(\beta_j)$ for $1\leq j\leq k$.
Suppose that whenever $\beta_j$ and $\beta_{j'}$ are conjugate over $\mathbb{Q}$, every isomorphism
\[
\sigma:\mathbb{Q}(\beta_j)\to\mathbb{Q}(\beta_{j'})
\]
satisfying $\sigma(\beta_j)=\beta_{j'}$ also satisfies
$
\sigma(c_j)=c_{j'}.
$
Then
\[
\sum_{j=1}^k c_j \in \mathbb{Q}.
\]

\end{itemize}
\end{lem}
\begin{proof}
Part (i) follows from the fundamental theorem of symmetric polynomials (see, e.g.,
\cite[Theorem~3.10]{Pollard1950}) and Vieta’s formulas. To see (ii), from the hypothesis it follows that $\sigma\left(\sum_{j=1}^k c_j\right)=\sum_{j=1}^k c_j$ for all Galois automorphisms $\sigma$ of $P$ over $\Q$. Therefore the sum is fixed by all Galois automorphisms, so it lies in $\mathbb{Q}$; see, e.g.~\cite[Chap.~VI, Theorem~1.2]{Lang2002}.
\end{proof}

\begin{lem}\label{lem-3.10}
    Let $\Phi = \{Ax + a_\ell\}_{\ell=1}^m$ be an affinely irreducible homogeneous affine IFS in $\mathbb{R}^d$. Then the following statements hold:
\begin{itemize}
\item[(i)] For any nonzero $v \in \mathbb{C}^d$, there exist an integer $n \ge d$ and an index $1 \le \ell \le m$ such that $\langle (A^\top)^n v,\; a_\ell \rangle \neq 0$.
\item[(ii)] For any eigenvector $\eta$ of $A^\top$, there exists $1 \le \ell \le m$ such that $\langle \eta,\; a_\ell \rangle \neq 0$.
\end{itemize}

\end{lem}

\begin{proof}
It suffices to prove (i), since (ii) follows directly from (i). Let $K$ be the attractor of $\Phi$.

To prove (i), suppose, on the contrary, that there exists a nonzero $v \in \mathbb{C}^d$ such that
\[
\langle (A^\top)^n v,\; a_\ell \rangle = 0
\]
for every integer $n \ge d$ and every $1 \le \ell \le m$. Then, for any $1 \le \ell \le m$ and any $n \in \mathbb{N} \cup \{0\}$,
\[
\langle (A^\top)^d v,\; A^{n} a_\ell \rangle
= \langle (A^\top)^{n+d} v,\; a_\ell \rangle
= 0.
\]
This implies that
\[
\langle (A^\top)^d v,\; x \rangle = 0
\quad \text{for all } x \in K,
\]
since
\[
K = \left\{ \sum_{n=0}^\infty A^{n} a_{\ell_n}
\;:\; \ell_n \in \{1, \ldots, m\}
\text{ for all } n \ge 0 \right\}.
\]
Hence, $K$ is contained in the proper subspace
\[
\{x \in \mathbb{R}^d : \langle (A^\top)^d v, x \rangle = 0\},
\]
which is a contradiction.\end{proof}

\section{Properties of affinely irreducible generalized Pisot IFSs}
\label{S-3}

In this section, we give some basic properties of affinely irreducible generalized Pisot IFSs that will be needed in the proofs of our main results.

Throughout this section, let $\Phi=\{\phi_\ell(x)=Ax+a_\ell\}_{\ell=1}^m$ be an affinely irreducible homogeneous affine IFS on $\R^d$. Suppose that $\Phi$ is also a generalized Pisot IFS. That is,  there exist
$k \geq 1$, $\beta_1,\ldots,\beta_k \in \mathbb{C}$, and
$\zeta_1,\ldots,\zeta_k \in \mathbb{C}^d \setminus \{0\}$ such that the following three
conditions are satisfied:
\begin{itemize}
    \item[(i)] $\{\beta_1,\ldots,\beta_k\}$ is a Pisot $k$-tuple;
    \item[(ii)] $A^\T\zeta_j = \beta_j^{-1} \zeta_j$ for $1 \leq j \leq k$, where $A^\T$ denotes the transpose of $A$;
    \item[(iii)] for any $1 \leq \ell \leq m$, there exists $P_\ell \in \mathbb{Z}[x]$
    such that
    \[
        \langle \zeta_j, a_\ell - a_1 \rangle = P_\ell(\beta_j)
        \quad \text{for } 1 \leq j \leq k.
    \]
\end{itemize}

Similar to \cite[Lemma 5.10]{Rapaport2022}, we have the following.

\begin{lem}
\label{lem-conjugate}
If $\beta_i=\overline{\beta_j}$ for some $1\leq i,j\leq k$, then $\zeta_i=\overline{\zeta_j}$.  Consequently, $\zeta_i\in \R^d$ if $\beta_i$ is real.
\end{lem}
\begin{proof}Our argument is adapted from the proof of \cite[Lemma 5.10]{Rapaport2022}.
Let $K$ denote the attractor of $\Phi$ and set $x=\sum_{n=0}^\infty A^n a_1$. Then
$$
K-x=\left\{\sum_{n=0}^\infty A^n v_n\colon u_n\in \{0,a_2-a_1, \ldots, a_m-a_1\} \text{ for all }n\geq 0\right\}.
$$
Since $\Phi$ is affinely irreducible, $K-x$ is not contained in any hyperplane of $\R^d$.  It follows that the linear span  of the set $$F:=\{A^n v\colon n\geq 0, v\in \{a_\ell-a_1:\; 1\leq \ell\leq m\}\}$$ over $\R$ is the entire space $\R^d$.

Now suppose that $\beta_i=\overline{\beta_j}$ for some $1\leq i,j\leq k$. Since ${\rm span}_{\R}(F)=\R^d$, to prove that $\zeta_i=\overline{\zeta_j}$, it suffices to show that
$$\langle \zeta_i-\overline{\zeta_j}, A^n (a_\ell-a_1)\rangle =0
$$ for all $n\geq 0$ and $1\leq \ell\leq m$. Indeed, this holds since
\begin{align*}
\left\langle \zeta_i-\overline{\zeta_j}, A^n (a_\ell-a_1)\right\rangle &=\left\langle (A^\top)^n(\zeta_i-\overline{\zeta_j}), a_\ell-a_1\right\rangle\\
&=\left\langle \beta_i^{-n}\zeta_i-\overline{\beta_j}^{-n}\overline{\zeta_j}, a_\ell-a_1\right\rangle\\
&=\beta_i^{-n}P_\ell(\beta_i)-\overline{\beta_j^{-n}P_\ell(\beta_j)}\\
&=0,
\end{align*}
where the last equality uses the facts that $\beta_i=\overline{\beta_j}$ and $P_\ell\in \Z[x]$.
\end{proof}

\begin{lem}
\label{lem-npisot}
For every positive integer $n$, the set of distinct elements of
$
\{\beta_1^n,\ldots,\beta_k^n\}
$
 is again a Pisot tuple of the same type as $\{\beta_1,\ldots, \beta_k\}$.
\end{lem}
\begin{proof}
This follows directly from the following standard fact about algebraic numbers:
let $\alpha_1,\ldots,\alpha_\ell$ be the roots of a monic irreducible polynomial in
$\mathbb{Z}[x]$ of degree $\ell$. Then, for any positive integer $n$, the distinct numbers among
$\alpha_1^n,\ldots,\alpha_\ell^n$ are precisely the roots of a monic irreducible polynomial with
integer coefficients; see e.g.~\cite[Theorem~5.10(i)]{Pollard1950}.
\end{proof}

\begin{prop}
\label{pro-2.4}
The following two statements hold:
\begin{itemize}
\item[(i)]
For every positive integer \(n\), the \(n\)-th iterate of \(\Phi\), defined by
\[
\Phi^n:=\{\phi_{\ell_1}\circ\cdots\circ\phi_{\ell_n}
: 1\le \ell_1,\ldots,\ell_n\le m\},
\]
is also a generalized Pisot IFS of the same type as $\Phi$.

\item[(ii)]
There exists \(n\in\mathbb{N}\) such that the following holds:
letting \(\{\widetilde{\beta}_1,\ldots,\widetilde{\beta}_\tau\}\) denote
the associated Pisot tuple of the IFS \(\Phi^n\), and letting $\{\widetilde{\beta}_1,\ldots,\widetilde{\beta}_{\tau'}\}$ be the set of all their conjugates (including themselves)  over $\Q$,  the arguments of these
complex numbers satisfy
\[
\arg(\widetilde{\beta}_j)/(2\pi)\in
\operatorname{span}_{\mathbb{Z}}\{1,\theta_1,\ldots,\theta_s\},
\qquad j=1,\ldots,\tau',
\]
for some \(\theta_1,\ldots,\theta_s\in[0,1)\) such that
\(1,\theta_1,\ldots,\theta_s\) are linearly independent over \(\mathbb{Q}\).
\end{itemize}
\end{prop}

To prove the above proposition, we need the following result of K\"{o}rnyei, which is a special version of \cite[Theorem 1]{Kornyei1987}.

\begin{thm} [{\cite[Theorem 1]{Kornyei1987}}]
\label{thm-kornyei}
Let $\alpha_1,\ldots,\alpha_l$ be distinct non-zero algebraic numbers with $|\alpha_j|> 1$.
Let $c_1,\ldots,c_l$ be nonzero complex numbers.
Suppose
\begin{equation}
\label{e-pvlimit}
\lim_{n\to+\infty} \left\|\sum_{j=1}^l c_j \alpha_j^n \right\|_\Z = 0.
\end{equation}
Then the following statements hold:
\begin{itemize}
\item[(i)]
 $\{\alpha_1,\ldots,\alpha_l\}$ is a Pisot $l$-tuple;
\item[(ii)] for $1\le j\le l$, $c_j\in\Q(\alpha_j)$;
\item[(iii)] if $\alpha_{j_1}$ and $\alpha_{j_2}$ are conjugates over $\Q$ and
$\sigma:\Q(\alpha_{j_1})\to\Q(\alpha_{j_2})$ is an isomorphism with $\sigma(\alpha_{j_1})=\alpha_{j_2}$,
then $\sigma(c_{j_1})=c_{j_2}$.
\end{itemize}
\end{thm}

The above theorem has a direct corollary.
\begin{cor}
\label{cor-2.7}
Let $\{\alpha_1,\ldots,\alpha_l\}$ be a Pisot $l$-tuple such that $\alpha_1,\ldots, \alpha_l$ are conjugates over $\Q$. Suppose that \eqref{e-pvlimit} holds for some complex numbers $c_1,\ldots, c_l$; here we don't require them to be nonzero. Then  conclusions
\((ii)\) and \((iii)\) of Theorem~\ref{thm-kornyei} hold.
\end{cor}
\begin{proof}
It suffices to show that either all of \( c_1,\ldots,c_l \) are zero, or all of them are nonzero.
Suppose, on the contrary, that some of them are zero but not all.
Then, by Theorem~\ref{thm-kornyei}, there exists a proper subset of
$
\{\alpha_1,\ldots,\alpha_l\}
$
that forms a Pisot tuple. However, this is impossible, since by assumption
\( \alpha_1,\ldots,\alpha_l \) are conjugates over \( \mathbb{Q} \).
\end{proof}

We also need the following lemma.
\begin{lem}\label{lem-2.1}

    For any $Q\in \mathbb{Z}[x]$, there exist $C>1$ and $0<\delta<1$ such that
    $$
    \left\|\sum_{j=1}^k \beta_j^nQ(\beta_j)\right\|_\Z\leq C \delta^{|n|} \quad \text{ for all } n \in \mathbb{Z}.
    $$
\end{lem}

\begin{proof}
    Let $R \in \mathbb{Z}[x]$ be the monic integer polynomial of smallest degree such that
$R(\beta_j)=0$ for $1 \le j \le k$. Denote by $\beta_{k+1},\ldots,\beta_{k'}$
the remaining roots of $R$, if any. Set
\[
\delta = \max\{|\beta_1|^{-1},\ldots,|\beta_k|^{-1},
|\beta_{k+1}|,\ldots,|\beta_{k'}|\}
\]
and
\[
C=\max\left\{\sum_{j=1}^k |Q(\beta_j)|,\;
\sum_{j=k+1}^{k'} |Q(\beta_j)|\right\}.
\]
Then $0<\delta<1$, since $\{\beta_1,\ldots, \beta_k\}$ is a Pisot tuple.

Now suppose that $n<0$. Then
\[
\left\|\sum_{j=1}^k \beta_j^n Q(\beta_j)\right\|_{\mathbb{Z}}
\le \sum_{j=1}^k |\beta_j|^n\, |Q(\beta_j)|
\le C\,\delta^{-n}.
\]
On the other hand, if $n\ge 0$, since
$\sum_{j=1}^{k'} \beta_j^nQ(\beta_j) \in \mathbb{Z}$ by Lemma~\ref{lem-Galois}(i), we have
\[
\left\|\sum_{j=1}^k \beta_j^n Q(\beta_j)\right\|_{\mathbb{Z}}
=
\left\|\sum_{j=k+1}^{k'} \beta_j^nQ(\beta_j)\right\|_{\mathbb{Z}}
\le \sum_{j=k+1}^{k'} |\beta_j|^n\,|Q(\beta_j)|
\le C\,\delta^n.
\qedhere
\]
\end{proof}

\medskip

\begin{proof}[Proof of Proposition~\ref{pro-2.4}] Suppose that $\Phi$ is a generalized Pisot IFS of type I (resp. type II) on $\mathbb{R}^d$ with parameters $k$,
$\beta_1,\ldots,\beta_k$, $\zeta_1,\ldots,\zeta_k$, and $P_1,\ldots,P_m$
as in Definition~\ref{defn-1.2}. Without loss of generality we may assume that $a_1=0$. By replacing the index set \(\{1,\ldots,k\}\) with a suitable subset if necessary, we may assume that
\(\beta_1,\ldots,\beta_k\) are Galois conjugates over \(\mathbb{Q}\).

To prove part (i), let $n\in \N$. We show below that the $n$-th iterate $\Phi^n$ of $\Phi$ is again a generalized Pisot IFS of type I (resp. type II).

For  each word $\ell_1\ldots \ell_n\in \{1,\ldots,m\}^n$,
define
$$
a_{\ell_1\ldots \ell_n}:=\sum_{i=1}^{n}A^{i-1}a_{\ell_i}.
$$
Clearly, $$\Phi^n=\left\{A^nx+a_{\ell_1\ldots \ell_n}:\; \ell_1\ldots \ell_n\in \{1,\ldots,m\}^n\right\}.$$

By  relabeling the indices $\{1,\ldots, k\}$ if necessary, we may assume that
\( \beta_1^n,\ldots,\beta_\tau^n \) enumerate the distinct elements of the set
\( \{\beta_j^n\}_{j=1}^k \). By Lemma~\ref{lem-npisot},
\( \{\beta_1^n,\ldots,\beta_\tau^n\} \) is a Pisot tuple. For $1\leq j\leq \tau$, write
$$
[j]=\{1\leq j'\leq k\colon \beta_{j'}^n=\beta_j^n\}.
$$

Note that for each $1\leq j\leq k$ and every word $\ell_1\ldots \ell_n\in \{1,\ldots,m\}^n$,
\begin{align*}
\left\langle \beta_j^{n-1} \zeta_j, a_{\ell_1\ldots \ell_n} \right\rangle&=\sum_{i=1}^{n} \beta_j^{n-1} \left\langle \zeta_j, A^{i-1} a_{\ell_i})\right\rangle
=\sum_{i=1}^{n} \beta_j^{n-1}\left \langle (A^\top)^{i-1}\zeta_j,  a_{\ell_i}\right\rangle\\
&=\sum_{i=1}^{n} \beta_j^{n-i}\left \langle \zeta_j,  a_{\ell_i}\right\rangle=\sum_{i=1}^n \beta_j^{n-i}P_{\ell_i}(\beta_j).
\end{align*}
Together with Lemma~\ref{lem-2.1}, this implies that for every $\ell_1\ldots \ell_n\in \{1,\ldots,m\}^n$,
\begin{align*}
\left \|\sum_{j=1}^k   \left\langle \beta_j^{n-1} \zeta_j, a_{\ell_1\ldots \ell_n} \right\rangle \beta_j^{nt} \right\|_\Z&=
\left \|\sum_{j=1}^k  \sum_{i=1}^n \beta_j^{n-i}P_{\ell_i}(\beta_j)\beta_j^{nt} \right\|_\Z\\
&\leq
\sum_{i=1}^n \left \|\sum_{j=1}^k   \beta_j^{n-i}P_{\ell_i}(\beta_j)\beta_j^{nt} \right\|_\Z
\longrightarrow  0 \;\text{ exponentially},
\end{align*}
as $t\in \N$, $t\to +\infty$. Equivalently,
\begin{equation}
\label{e-tend0}
\left \|\sum_{j=1}^\tau   \left\langle \sum_{j'\in [j]}\beta_{j'}^{n-1} \zeta_{j'}, a_{\ell_1\ldots \ell_n} \right\rangle\left(\beta_j^{n}\right)^t \right\|_\Z\longrightarrow 0 \;\text{ exponentially},
\end{equation}
as $t\in \N$, $t\to +\infty$.
Define
$$
\widetilde{\zeta}_j=\sum_{j'\in [j]}\beta_{j'}^{n-1} \zeta_{j'},\quad j=1,\ldots, \tau.
$$
By \eqref{e-tend0} and Corollary~\ref{cor-2.7},  for every $\ell_1\ldots \ell_n\in \{1,\ldots,m\}^n$,
$$
 \left\langle \widetilde{\zeta}_j, a_{\ell_1\ldots \ell_n} \right\rangle\in \Q\left( \beta_j^{n}\right),\quad j=1,\ldots, \tau;
$$
moreover, for any $j_1, j_2\in \{1,\ldots, \tau\}$, if
$\sigma:\Q(\beta^n_{j_1})\to\Q(\beta^n_{j_2})$ is an isomorphism with $\sigma(\beta^n_{j_1})=\beta^n_{j_2}$,
then
$$
\sigma\left(\left\langle \widetilde{\zeta}_{j_1}, a_{\ell_1\ldots \ell_n} \right\rangle\right)= \left\langle \widetilde{\zeta}_{j_2}, a_{\ell_1\ldots \ell_n} \right\rangle.
$$

Since $\left\langle \widetilde{\zeta}_1, a_{\ell_1\ldots \ell_n} \right\rangle\in \Q\left( \beta_1^{n}\right)$ for all $\ell_1\ldots \ell_n\in \{1,\ldots,m\}^n$, it follows that there exist $q\in \N$ and integer polynomials $P_{\ell_1\ldots \ell_n}\in \Z[x]$,  such that
$$
 \left\langle q \widetilde{\zeta}_1, a_{\ell_1\ldots \ell_n} \right\rangle=P_{\ell_1\ldots \ell_n}(\beta_1^n)\quad\text{ for all }\ell_1\ldots \ell_n.
  $$
 For every $j\in \{1,\ldots, \tau\}$, letting $\sigma:\Q(\beta^n_{1})\to\Q(\beta^n_{j})$ be an isomorphism with $\sigma(\beta^n_{1})=\beta^n_{j}$, we have for all $\ell_1\ldots \ell_n\in \{1,\ldots,m\}^n$,
 $$
 \left\langle q \widetilde{\zeta}_j, a_{\ell_1\ldots \ell_n} \right\rangle=\sigma \left(\left\langle q \widetilde{\zeta}_1, a_{\ell_1\ldots \ell_n} \right\rangle\right)=\sigma (P_{\ell_1\ldots \ell_n}(\beta_1^n))= P_{\ell_1\ldots \ell_n}(\beta_j^n).
 $$

  Finally, observe that \( q\widetilde{\zeta}_j\neq 0 \) and
\[
(A^n)^\top(q\widetilde{\zeta}_j)
=(A^\top)^n\!\left(q\sum_{j'\in [j]} \beta_{j'}^{n-1}\zeta_{j'}\right)
=q\sum_{j'\in [j]} \beta_{j'}^{-1}\zeta_{j'}
=\beta_j^{-n}\, q\widetilde{\zeta}_j,
\quad 1\le j\le \tau.
\]
Hence \( \Phi^n \) is a generalized Pisot IFS of type~I (resp.\ type~II),
with parameters
\(
\tau,\,
\beta_1^n,\ldots,\beta_\tau^n,\,
q\widetilde{\zeta}_1,\ldots,q\widetilde{\zeta}_\tau
\),
and
\( \{P_{\ell_1\ldots \ell_n}\}_{\ell_1\ldots \ell_n\in \{1,\ldots,m\}^n} \).  This completes the proof of part (i).

Next we prove part (ii) of the proposition. Let $\beta_1,\ldots, \beta_{k'}$ denote  all Galois conjugates of $\beta_1,\ldots, \beta_k$ over $\Q$, including $\beta_1,\ldots, \beta_k$  themselves.   Since $${\rm span}_\Q\left\{\frac{\arg(\beta_1)}{2\pi},\ldots, \frac{\arg(\beta_{k'})}{2\pi}\right\}$$ has  finite rank, there exist $\theta_1,\ldots,\theta_s \in [0,1)$ with $s\geq 1$ such that
$1,\theta_1,\ldots,\theta_s$ are rationally independent, and  rational numbers
$t_{j,u}$ ($1 \leq j \leq k'$,  $0 \leq u \leq s$) satisfying
$$
\frac{\arg(\beta_j)}{2\pi} =t_{j,0}+ \sum_{u=1}^s t_{j,u}\theta_u,
\qquad j=1,\ldots,k'.
$$
Choose a positive integer $n$ such that  $nt_{j,u}\in \Z$ for all $1 \leq j \leq k'$ and $0 \leq u \leq s$.
Then
$$\frac{\arg(\beta_j^n)}{2\pi}\in {\rm span}_\Z(1,\theta_1,\ldots, \theta_s), \qquad j=1,\ldots,k',$$  which completes the proof of part (ii).
\end{proof}

\begin{lem}\label{lem-2.2}
  Let $\mu$ be the self-affine measure associated with $\Phi$ and a probability vector $\pb=(p_\ell)_{\ell=1}^m$.  Suppose that $\mu$ is Rajchman. Then for any $Q \in \mathbb{Z}[x]$
such that $Q(\beta_j)\neq 0$ for at least one $j \in \{1,\ldots,k\}$,
there exists an integer $n \in \mathbb{Z}$ such that
\[
\sum_{\ell=1}^m p_\ell
\exp\!\left(2\pi i
\sum_{j=1}^k \beta_j^n Q(\beta_j) P_\ell(\beta_j)
\right)=0.
\]
\end{lem}

\begin{proof}
Set $\xi_N=\sum_{j=1}^k \beta_j^N Q(\beta_j) \zeta_j$ for $N\in \mathbb{N}$. Since $\{\beta_1,\ldots,\beta_k\}$ is a Pisot tuple, for each $1\leq j\leq k$, $$\overline{\beta_j}\in \{\beta_1,\ldots,\beta_k\}.$$ Hence, by Lemma~\ref{lem-conjugate}, $\xi_N\in \R$.
Then, by Lemma~\ref{lem-infinity}, we have $|\xi_N| \to \infty$ as $N \to +\infty$.
Consequently,
\[
\widehat{\mu}(\xi_N)\to 0 \quad \text{as } N\to +\infty,
\]
since $\mu$ is Rajchman. Recall that
\[
\widehat{\mu}(\xi)
=\prod_{n=0}^\infty H\left((A^\top)^{n}\xi\right),
\qquad\text{where }
H(\xi)=\sum_{\ell=1}^m p_\ell \exp\left(2\pi i\langle \xi,\; a_\ell \rangle\right).
\]
A direct computation yields
\begin{align*}
H\left((A^\top)^{n}\xi_N\right)
&=\sum_{\ell=1}^m p_\ell
\exp\left(2\pi i\left\langle (A^\top)^{n} \xi_N,\; a_\ell \right\rangle\right)\\
&=\sum_{\ell=1}^m p_\ell
\exp\!\left(2\pi i\left\langle
\sum_{j=1}^k \beta_j^{N-n} Q(\beta_j)\zeta_j,\; a_\ell
\right\rangle\right)\\
&=\sum_{\ell=1}^m p_\ell
\exp\!\left(2\pi i
\sum_{j=1}^k \beta_j^{N-n} Q(\beta_j) P_\ell(\beta_j)
\right).
\end{align*}
It follows that for $N\in \mathbb{N}$,
\[
\widehat{\mu}(\xi_N)
=\prod_{n=-\infty}^N
\left[
\sum_{\ell=1}^m p_\ell
\exp\!\left(2\pi i
\sum_{j=1}^k \beta_j^{n} Q(\beta_j) P_\ell(\beta_j)
\right)
\right].
\]
Hence,
\[
0=\lim_{N\to +\infty}\widehat{\mu}(\xi_N)
=\prod_{n=-\infty}^{+\infty}
\left[
\sum_{\ell=1}^m p_\ell
\exp\!\left(2\pi i
\sum_{j=1}^k \beta_j^{n} Q(\beta_j) P_\ell(\beta_j)
\right)
\right],
\]
where the latter infinite product converges, since the $n$th factor is
$1+O(\delta^{|n|})$ for some $0<\delta<1$, by applying
Lemma~\ref{lem-2.1} to the polynomials $QP_\ell$, $1\le \ell\le m$.
This forces
\[
\sum_{\ell=1}^m p_\ell
\exp\!\left(2\pi i
\sum_{j=1}^k \beta_j^{n} Q(\beta_j) P_\ell(\beta_j)
\right)=0
\]
for some $n\in \mathbb{Z}$.
\end{proof}

\begin{lem}\label{lem-2.3}
    For any $Q \in \mathbb{Z}[x]$, there exists $N \in \mathbb{N}$ such that,
for any probability vector $\pb=(p_\ell)_{\ell=1}^m$ and any
$n \in \mathbb{Z}$ with $|n|>N$, we have
\[
\sum_{\ell=1}^m p_\ell
\exp\!\left(2\pi i
\sum_{j=1}^k \beta_j^n Q(\beta_j) P_\ell(\beta_j)
\right)\neq 0.
\]
\end{lem}

\begin{proof}
    Let $C$ and $\delta$ be as in Lemma~\ref{lem-2.1}. Choose $N \in \mathbb{N}$
sufficiently large so that $2\pi C \delta^N < \tfrac{1}{2}$.
Then, for any probability vector $\pb=(p_\ell)_{\ell=1}^m$ and any $n \in \mathbb{Z}$
with $|n|>N$, we have
\begin{align*}
&\left|
\sum_{\ell=1}^m p_\ell
\exp\!\left(2\pi i
\sum_{j=1}^k \beta_j^{n} Q(\beta_j) P_\ell(\beta_j)
\right)
-1
\right| \\
&\quad \leq
\sum_{\ell=1}^m p_\ell
\left|
\exp\!\left(2\pi i
\sum_{j=1}^k \beta_j^{n} Q(\beta_j) P_\ell(\beta_j)
\right)
-1
\right| \\
&\quad \leq
2\pi \sum_{\ell=1}^m p_\ell
\left\|
\sum_{j=1}^k \beta_j^{n} Q(\beta_j) P_\ell(\beta_j)
\right\|_{\mathbb{Z}} \\
&\quad \leq
2\pi C \delta^{|n|}
< 2\pi C \delta^N
< \tfrac{1}{2},
\end{align*}
which implies the desired result. In the second inequality, we used the
elementary estimate $|\exp(2\pi i x)-1|\leq 2\pi |x|$ for all $x \in \mathbb{R}$.\end{proof}

\section{Proof of Theorem \ref{thm-1.4}}
\label{S-4}

In this section, we prove Theorem \ref{thm-1.4}.

\begin{proof}[Proof of Theorem~\ref{thm-1.4}]
Suppose that $\Phi$ is a generalized Pisot IFS on $\mathbb{R}^d$ with parameters
$\beta_1,\ldots,\beta_k$, $\zeta_1,\ldots,\zeta_k$, and $P_1,\ldots,P_m$, as in Definition~\ref{defn-1.2}.
Let $\Omega$ denote the set of probability vectors in $\mathbb{R}^m$, that is,
\[
\Omega := \Bigl\{\pb = (p_\ell)_{\ell=1}^m :
\sum_{\ell=1}^m p_\ell = 1,\; p_\ell \ge 0 \text{ for all } 1 \le \ell \le m \Bigr\}.
\]By Lemma~\ref{lem-2.3}, there exists $N \in \mathbb{N}$ such that
\[
\sum_{\ell=1}^m p_\ell
\exp\!\left(2\pi i
\sum_{j=1}^k \beta_j^n P_\ell(\beta_j)
\right) \neq 0
\quad \text{for all } \pb \in \Omega \text{ and all } n \in \mathbb{Z} \text{ with } |n| > N.
\]
Define
$g:\Omega \to \mathbb{C}$ by
\[
g(\pb)
=
\prod_{n=-N}^N
\left(
\sum_{\ell=1}^m p_\ell
\exp\!\left(2\pi i
\sum_{j=1}^k \beta_j^n  P_\ell(\beta_j)
\right)
\right).
\]
Note that $P_1(\beta_j)=0$ for all $1\leq j\leq k$. It follows that $g(1,0,\ldots,0) = 1$. Clearly, $g$ is continuous over $\Omega$.  Therefore,
there exists $\varepsilon > 0$ such that if $\pb \in \Omega$ is $\varepsilon$-close to
$(1,0,\ldots,0)$, then
\[
|g(\pb)| > \tfrac{1}{2}.
\]
For such $\pb$, which can be chosen to have positive coordinates, we have
\[
\sum_{\ell=1}^m p_\ell
\exp\!\left(2\pi i
\sum_{j=1}^k \beta_j^n P_\ell(\beta_j)
\right) \neq 0
\quad \text{for all } n \in \mathbb{Z}.
\]
 By Lemma~\ref{lem-2.2}, it follows that the self-affine measure associated with $\Phi$ and $\pb$ is non-Rajchman.
\end{proof}

\section{Proof of Theorem~\ref{thm-1.4'}}
\label{S-5}

In this section, we prove Theorem~\ref{thm-1.4'}, which we now restate.
\begin{theorem*}
Let \(\Phi = \{\phi_\ell(x) = Ax + a_\ell\}_{\ell=1}^m\) be an affinely irreducible  homogeneous affine IFS on \(\mathbb{R}^d\), and let $K$ be the attractor of $\Phi$.
Suppose there exists a probability vector \(\pb = (p_\ell)_{\ell=1}^m > 0\) such that the self-affine measure \(\mu\) generated by \(\Phi\) and \(\pb\) is non-Rajchman. Then \(\Phi\) is a generalized Pisot IFS. More precisely, one of the the following two cases will occur:
\begin{itemize}
\item[(i)]  \(\Phi\) is a generalized Pisot IFS of type~I;
\item[(ii)] $\Phi$ has an integral factor $\Psi$ such that the self-affine measure associated with $\Psi$ and $\pb$ is non-Rajchman.
\end{itemize}
\end{theorem*}

The proof of the above theorem is given in Section~\ref{S-4.1.3'}, which is partially inspired by that of \cite[Theorem~1.5]{Rapaport2022}.  It is based on several auxiliary results given in Sections~\ref{S-4.1.3} and~\ref{S-4.1.2}.
\subsection{Two useful lemmas in linear algebra}
\label{S-4.1.3}

The proof of Theorem \ref{thm-1.4'} relies on two elementary lemmas in linear algebra.

To state these results, we begin by recalling the concept of a generalized eigenspace and its basic properties. The reader is referred to \cite[Chapter~7.1]{Friedberg2002} for further details.

Let $B \in GL_d(\mathbb{C})$, and let $\beta$ be an eigenvalue of $B$. The \emph{generalized eigenspace} of $B$ corresponding to $\beta$, denoted by $E_\beta$, is the subspace of $\mathbb{C}^d$ defined by
\[
E_\beta = \left\{ x \in \mathbb{C}^d : (B - \beta I)^n x = 0 \text{ for some positive integer } n \right\},
\]
where $I$ denotes the $d \times d$ identity matrix.

It is well known that $E_\beta$ is $B$-invariant, that is, $B E_\beta \subseteq E_\beta$. Moreover, if $x$ is a nonzero element of $E_\beta$ and $n$ is the smallest positive integer such that $(B - \beta I)^n x = 0$, then $n \le d$ and $(B - \beta I)^{n-1} x$ is an eigenvector of $B$ corresponding to $\beta$. Furthermore, if $\beta_1, \ldots, \beta_p$ are the distinct eigenvalues of $B$, then $\C^d$ admits the following direct sum decomposition:
\[
\mathbb{C}^d = E_{\beta_1}  \oplus \cdots \oplus E_{\beta_p}.
\]

Our first lemma is the following.
\begin{lem}
\label{lem-Bnxi}
Let $B \in GL_d(\mathbb{C})$ and $\xi \in \mathbb{C}^d \setminus \{0\}$. Let $\beta_1, \ldots, \beta_p$ be the distinct eigenvalues of $B$, and let $E_{\beta_1}, \ldots, E_{\beta_p}$ be the corresponding generalized eigenspaces of $B$. Write
$$
N_j=B-\beta_jI,\quad j=1,\ldots, p,
$$
 and
\[
\xi = \xi_1 + \cdots + \xi_p
\]
with  $\xi_j \in E_{\beta_j}$  for $1 \le j \le p$.
Define
\[
\mathcal{J} = \{\, 1 \le j \le p : \xi_j \ne 0 \,\}.
\]
For each $j \in \mathcal{J}$, let $m_j$ be the smallest positive integer such that
$N_j^{m_j} \xi_j = 0$.
Then the following statements hold:
\begin{itemize}
\item[(i)] For every integer $n\geq 0$,
\[
B^n \xi
= \sum_{j \in \mathcal{J}} \sum_{u=0}^{\min\{n,\, m_j-1\}}
\binom{n}{u} \beta_j^{\,n-u} N_j^u \xi_j.
\]
\item[(ii)] Let $q=\sum_{j\in \mathcal J}m_j$ and write
$$Q(x)=\prod_{j\in \mathcal J}(x-\beta_j)^{m_j}=x^q+c_{q-1}x^{q-1}+\cdots+c_1 x+c_0.$$
Then $\{\xi, B\xi,\ldots, B^{q-1}\xi\}$ is a basis of ${\rm span}_\C\{B^n\xi\colon n\in \N\cup \{0\}\}$, and
$$
B^q\xi=-\sum_{i=0}^{q-1} c_i B^{i}\xi.
$$
\end{itemize}
\end{lem}

\begin{proof}
We first prove (i). Let $n \in \mathbb{N} \cup \{0\}$. It suffices to show that for each $j \in \mathcal{J}$,
\[
B^n \xi_j
= \sum_{u=0}^{\min\{n,\,m_j-1\}}
\binom{n}{u} \beta_j^{\,n-u} N_j^u \xi_j.
\]
To this end, fix $j \in \mathcal{J}$. By the binomial theorem,
\[
B^n \xi_j
= (N_j+\beta_j I)^n \xi_j
= \sum_{u=0}^{n} \binom{n}{u} \beta_j^{\,n-u} N_j^u \xi_j.
\]
Since $N_j^{m_j} \xi_j = 0$, we have $N_j^u \xi_j = 0$ for all $u \ge m_j$. Hence, the above sum reduces to
\[
\sum_{u=0}^{\min\{n,\, m_j-1\}} \binom{n}{u} \beta_j^{\,n-u} N_j^u \xi_j,
\]
as desired.

Next, we prove (ii). First we show that for each $j \in \mathcal{J}$, the vectors
$$\xi_j, N_j \xi_j, \ldots, N_j^{m_j-1} \xi_j$$ are linearly independent.
To see this, suppose to the contrary that there exist scalars
$a_0, a_1, \ldots, a_{m_j-1} \in \mathbb{C}$, not all zero, such that
\[
a_0 \xi_j + a_1 N_j \xi_j + \cdots + a_{m_j-1} N_j^{m_j-1} \xi_j = 0.
\]
Let $k$ be the smallest index such that $a_k \neq 0$. Then
\[
\sum_{u=k}^{m_j-1} a_u N_j^u \xi_j = 0.
\]
Applying the operator $N_j^{\,m_j-1-k}$ to both sides yields
\[
\sum_{u=k}^{m_j-1} a_u N_j^{\,m_j-1+u-k} \xi_j = 0.
\]
Since $N_j^r \xi_j = 0$ for all $r \ge m_j$, all terms vanish except the first, and we obtain
\[
a_k N_j^{m_j-1} \xi_j = 0.
\]
Thus $N_j^{m_j-1} \xi_j = 0$, a contradiction. Hence $\xi_j, N_j \xi_j, \ldots, N_j^{m_j-1} \xi_j$ are linearly independent.

Note also that
\[
\operatorname{span}_{\mathbb{C}} \{ \xi_j, N_j \xi_j, \ldots, N_j^{m_j-1} \xi_j \}
\subseteq E_{\beta_j}, \qquad j \in \mathcal{J}.
\]
Hence, the vectors in the set
\[
\mathcal{B} = \bigcup_{j \in \mathcal{J}} \{ \xi_j, N_j \xi_j, \ldots, N_j^{m_j-1} \xi_j \}
\]
are linearly independent. Let $W = \operatorname{span}_{\mathbb{C}} \mathcal{B}$. Then
\[
\dim W = \sum_{j \in \mathcal{J}} m_j = q.
\]

Observe that for each $j \in \mathcal{J}$,
\begin{align*}
B N_j^u \xi_j &= \beta_j N_j^u \xi_j + N_j^{u+1} \xi_j \quad \text{for } 0 \le u \le m_j - 2, \\
B N_j^{m_j-1} \xi_j &= \beta_j N_j^{m_j-1} \xi_j.
\end{align*}
It follows that $BW \subseteq W$, and the matrix representation of $B_W$ (the restriction of $B$ to $W$) with respect to the basis $\mathcal{B}$ is the block diagonal matrix
\[
\operatorname{diag}(J_j : j \in \mathcal{J}),
\]
where $J_j$ is the $m_j \times m_j$ Jordan block with eigenvalue $\beta_j$. Hence, the characteristic polynomial of $B_W$ is
\[
Q(x) = \prod_{j \in \mathcal{J}} (x - \beta_j)^{m_j}
= x^q + c_{q-1} x^{q-1} + \cdots + c_1 x + c_0.
\]
By the Cayley--Hamilton theorem, $Q(B_W) = 0$, and therefore,
\begin{equation}
\label{e-3.8}
(B_W)^q = -\sum_{u=0}^{q-1} c_u (B_W)^u.
\end{equation}

Finally, set
\[
W' = \operatorname{span}_{\mathbb{C}} \{ B^n \xi : n = 0,1,\ldots \}.
\]
Clearly, $BW' \subseteq W'$. Since $W$ is $B$-invariant and $\xi = \sum_{j \in \mathcal{J}} \xi_j$ with $\xi_j \in W$, it follows that $W' \subseteq W$.

Notice that for each $j \in \mathcal{J}$, the set $\{ B^n \xi_j : n \ge 0 \}$ spans
\[
W_j := \operatorname{span}_{\mathbb{C}} \{ \xi_j, N_j \xi_j, \ldots, N_j^{m_j-1} \xi_j \}.
\]
Hence, the projection $P_j$ of $W'$ onto $E_{\beta_j}$ is $W_j$ for each $j \in \mathcal{J}$, where $P_j$ denotes the projection with range $E_{\beta_j}$ and null space
\[
\bigoplus_{\substack{1\leq j'\leq p\colon  j' \ne j}} E_{\beta_{j'}}.
\]  Therefore,
\[
W' \supseteq \bigoplus_{j \in \mathcal{J}} W_j = W,
\]
and thus $W' = W$.

Since $\xi \in W$ and $W$ is $B$-invariant, we have $B^n \xi = (B_W)^n \xi$. By \eqref{e-3.8},
\[
B^q \xi = -\sum_{u=0}^{q-1} c_u B^u \xi.
\]
It follows that
\[
W = \operatorname{span}_{\mathbb{C}} \{ \xi, B \xi, \ldots, B^{q-1} \xi \}.
\]
Since $\dim W = q$, the set $\{ \xi, B \xi, \ldots, B^{q-1} \xi \}$ is a basis of $W$.
\end{proof}

Next we state our second lemma.
\begin{lem}
\label{lem-3.7}
Let $p \in \mathbb{N} \cup \{0\}$. Then there exists an invertible $(p+1)\times(p+1)$ matrix $R_p$ with entries in $\Bbb Q$ such that the following holds: if $f(n) = \sum_{j=0}^p a_j n^j$ is a polynomial in $n$ with complex coefficients, then $f$ can be written uniquely as
\[
f(n) = \sum_{j=0}^p b_j \binom{n}{j}
\]
for all $n \ge p$, where the coefficients $b_j$ satisfy
\[
(b_0,b_1,\ldots, b_p)^\top=
R_p(a_0,a_1,\ldots, a_p)^\top
\]
\end{lem}

\begin{proof}
For each $j=0,\ldots,p$, the polynomial $\binom{n}{j}$ has degree $j$ and
\[
\binom{n}{j} = \frac{1}{j!} n^j + \text{(lower-degree terms)}.
\]
It follows that
\[
\bigl(\binom{n}{0},\binom{n}{1},\ldots, \binom{n}{p}\bigr)
=
(1,n,\ldots, n^p)\, C,
\]
where $C$ is an upper triangular $(p+1)\times (p+1)$ matrix with entries in $\mathbb{Q}$ and nonzero diagonal entries. In particular, $C$ is invertible and  $C^{-1}$ also has entries in $\mathbb{Q}$.

Thus,
\[
f(n)
=
(1,n,\ldots, n^p)(a_0,a_1,\ldots, a_p)^\top
=
\bigl(\binom{n}{0},\ldots, \binom{n}{p}\bigr)
C^{-1}(a_0,a_1,\ldots, a_p)^\top,
\]
which shows that
\[
(b_0,b_1,\ldots, b_p)^\top
=
C^{-1}(a_0,a_1,\ldots, a_p)^\top.
\]
Setting $R_p = C^{-1}$ proves the result. Uniqueness follows since $\{\binom{n}{0},\ldots,\binom{n}{p}\}$ forms a basis of the space of polynomials of degree at most $p$.
\end{proof}

\subsection{Two more auxiliary results for the proof of Theorem~\ref{thm-1.4'}}
\label{S-4.1.1}

To prove Theorem~\ref{thm-1.4'}, we still need several auxiliary results. One of them is stated below and follows directly from \cite[Chapter III, Theorem III]{Pisot1938} together with \cite[Theorem 1]{Kornyei1987}.
\begin{thm}\label{thm-Pisot}
    Let $k\geq 1, \beta_1,\ldots,\beta_k\in\mathbb{C}$ and let $Q_1,\ldots,Q_k\in \mathbb{C}[x]$ be non-zero polynomials with complex coefficients. Suppose that $|\beta_j|>1$ for $1\leq j \leq k$, $\beta_j\neq \beta_{j'}$ for $j\neq j'$, and $$\sum_{n=0}^\infty \left\|\sum_{j=1}^k Q_j(n)\beta_j^n\right\|_\Z^2<\infty.$$ Writing $Q_j(x)=\sum_{s=0}^{t_j}c_s^{(j)}x^s$ for $1 \leq j \leq k$, then the following statements hold:
    \begin{itemize}
        \item[(i)] The set $\{\beta_1,\ldots,\beta_k\}$ is a Pisot $k$-tuple;
        \item[(ii)] For $1 \leq j \leq k$, $c_s^{(j)}\in \mathbb{Q}(\beta_j)$ for $0\leq s \leq t_j$;
        \item[(iii)]If $\beta_j$ and $\beta_{j'}$ are conjugates over $\mathbb{Q}$ and $\sigma: \mathbb{Q}(\beta_j) \to \mathbb{Q}(\beta_{j'})$ is an isomorphism with $\sigma(\beta_j)=\beta_{j'}$, then $t_j=t_{j'}$ and $\sigma(c_s^{(j)})=c_s^{(j')}$ for $0 \leq s \leq t_j$.
    \end{itemize}
\end{thm}

The other auxiliary result is the following proposition.
\begin{prop}
\label{prop-3.8}
Under the assumptions of Theorem~\ref{thm-1.4'}, there exist
$\xi \in \mathbb{R}^d \setminus \{0\}$ and $\varepsilon>0$ such that  $|\widehat{\mu}((A^\top)^{-n}\xi)|\geq \varepsilon$ for every integer $n\geq 0$, and
\[
\sum_{n=0}^{\infty}
\bigl\| \langle (A^\top)^{-n}\xi,\, a_\ell-a_1\rangle \bigr\|_{\mathbb{Z}}^2 < \infty
\qquad \text{for } 1 \le \ell \le m.
\]
\end{prop}

In the remainder  of this subsection, we prove Proposition~\ref{prop-3.8}, which is adapted from the proofs of \cite[Theorem~II]{Salem1943} and \cite[Propsition~5.2]{Rapaport2022}. The proof of Proposition~\ref{prop-3.8} is based on the following lemma.

\label{S-4.1.2}
\begin{lem}
\label{lem-3.9}
Under the assumptions in Theorem \ref{thm-1.4'}, for any $\varepsilon > 0$ there exists
a constant $C = C(\varepsilon,\pb) > 0$ such that for any
$\xi \in \mathbb{R}^d$ with $|\widehat{\mu}(\xi)| > \varepsilon$, we have
\[
\sum_{n=0}^{\infty}
\bigl\|\langle (A^\top)^n \xi, a_\ell -a_1\rangle\bigr\|_{\mathbb{Z}}^2 < C
\qquad \text{for } 1 \le \ell \le m.
\]
\end{lem}

\begin{proof}

Recall that for $\xi\in\R^d$,
\[
\widehat{\mu}(\xi)
= \prod_{n=0}^{\infty} H((A^\top)^n\xi),
\quad \text{where }H(\xi) := \sum_{\ell=1}^m p_\ell e^{2\pi i \langle \xi, a_\ell \rangle}.
\]

Set $p_{\min} =\min_{1\leq \ell\leq m}p_\ell$.  Applying \cite[Lemma~2.1]{Solomyak2022}, we obtain
\begin{equation}
\label{e-e-hx}
1-|H(\xi)|
\ge
4p_{\min}
\|\langle\xi,a_{\ell}-a_1\rangle\|_{\mathbb{Z}}^2 \qquad\text{ for all }\xi\in \R^d \text{ and }1\leq \ell \leq m.
\end{equation}

For the reader's convenience, we reproduce the proof of \eqref{e-e-hx} from \cite{Solomyak2022} with slight modifications. Since $|H(\xi|\leq 1$,   \eqref{e-e-hx} clearly holds in the case $\ell=1$.  Next assume that $2\leq \ell\leq m$.  Note that for any $\xi \in \mathbb{R}^d$,
\begin{align*}
1-|H(\xi)|
&= 1-\left|
\sum_{j=1}^m p_j e^{2\pi i \langle \xi,\, a_j-a_1\rangle}
\right| \\
&\ge
1-\left|
p_1 + p_{\ell}e^{2\pi i \langle \xi,\, a_{\ell}-a_1\rangle}
\right|-(p_2+\cdots+p_{\ell-1})\\
&=
p_1 + p_{\ell}
-
\left|
p_1 + p_{\ell}e^{2\pi i \langle \xi,\, a_{\ell}-a_1\rangle}
\right|.
\end{align*}
Moreover, if  $p_1 \ge p_{\ell}$ then
\begin{align*}
p_1 + p_{\ell}
-
\left|
p_1 + p_{\ell}e^{2\pi i \langle \xi,\, a_{\ell}-a_1\rangle}
\right|
&=p_1 + p_{\ell}-\left|
p_1-p_\ell+ p_{\ell}(1+e^{2\pi i \langle \xi,\, a_{\ell}-a_1\rangle})
\right|\\
&\geq 2 p_{\ell}-
p_{\ell}\left|1+e^{2\pi i \langle \xi,\, a_{\ell}-a_1\rangle}\right|\\
&
\geq 2p_{\ell}\bigl(1-|\cos(\pi\langle\xi,\, a_{\ell}-a_1\rangle)|\bigr)\\
&=2p_{\ell}\bigl(1-|\cos(\pi\|\langle\xi,\, a_{\ell}-a_1\rangle\|_\Z)|\bigr).
\end{align*}
Similarly, if $p_1 < p_\ell$  then
\[
p_1 + p_{\ell}
-
\left|
p_1 + p_{\ell}e^{2\pi i \langle \xi,\; a_{\ell}-a_1\rangle}
\right|
\ge
2p_1\bigl(1-|\cos(\pi\|\langle\xi,\, a_{\ell}-a_1\rangle\|_\Z)|\bigr).
\]
Hence
\[
1-|H(\xi)|
\ge
2p_{\min}
\bigl(1-\cos(\pi\|\langle\xi,\, a_{\ell}-a_1\rangle\|_{\mathbb{Z}})\bigr)\geq 4p_{\min}
\|\langle\xi,\, a_{\ell}-a_1\rangle\|_{\mathbb{Z}}^2,
\]
where in the second inequality we used the inequality
$1-\cos(\pi t) \ge 2t^2$ for $0 \le t \le 1$.  This proves \eqref{e-e-hx}.

In what follows, fix $\varepsilon>0$ and let $\xi\in \R^d$ be such that  $|\widehat{\mu}(\xi)|>\varepsilon$.

For each $n\geq 0$, writing $t = |H((A^\top)^n\xi)|-1$ and using the inequality
$e^t \ge t+1$, we obtain $$e^{|H((A^\top)^n\xi)|-1}
\ge
 |H((A^\top)^n\xi)|.$$
 Hence
 \[
\prod_{n=0}^{\infty} e^{|H((A^\top)^n\xi)|-1}
\ge
\prod_{n=0}^{\infty} |H((A^\top)^n\xi)|= |\widehat{\mu}(\xi)|>\varepsilon.
\]
Taking logarithms yields
\begin{equation}
\label{e-3.3}
\sum_{n=0}^{\infty}
\bigl(1-|H((A^\top)^n\xi)|\bigr)
< -\log \varepsilon .
\end{equation}
Combining this with \eqref{e-e-hx}, we obtain that for all $1\leq \ell\leq m$,
\[
4p_{\min}
\sum_{n=0}^{\infty}
\bigl\|\langle (A^\top)^{n}\xi, a_{\ell} \rangle\bigr\|_{\mathbb{Z}}^2
\le
\sum_{n=0}^{\infty}
\bigl(1-|H((A^\top)^n\xi)|\bigr)
< -\log \varepsilon .
\]
Taking
\[
C = \frac{-\log \varepsilon}{4p_{\min}}
\]
completes the proof.
\end{proof}

\begin{proof}[Proof of Proposition~\ref{prop-3.8}]
Since $\mu$ is non-Rajchman, there exist $\varepsilon>0$ and a sequence
$(\xi_N) \subset \mathbb{R}^d$ with $\|\xi_N\|>1$ and  $\|\xi_N\|\to\infty$ as $N\to\infty$ such that
\[
|\widehat{\mu}(\xi_N)|>\varepsilon \qquad \text{for all } N\in\mathbb{N}.
\]
By Lemma~\ref{lem-3.9}, there exists a constant $C=C(\varepsilon,\pb)>0$ such that
\begin{equation}
\label{e-e3.4}
\sum_{n=0}^{\infty}
\bigl\|\langle (A^\top)^n\xi_N,a_\ell-a_1\rangle\bigr\|_{\mathbb{Z}}^2
< C
\qquad \text{for all } N\in \N\text{ and }1\le \ell\le m.
\end{equation}

Define
\[
\tau_N = \min\{\tau\in\mathbb{N} : \|(A^\top)^{\tau}\xi_N\|<1\},\qquad N\in \N.
\]
Then
$$
\|(A^\top)^{-1}\|\cdot\|(A^\top)^{\tau_N}\xi_N\|\geq \|(A^\top)^{\tau_N-1}\xi_N\|\geq 1,
$$
and hence $\left|(A^\top)^{\tau_N}\xi_N\right|\in\left[\|(A^\top)^{-1}\|^{-1},1\right)$.
Since $\|\xi_N\|\to\infty$ as $N\to\infty$, we have $\tau_N\to\infty$ as
$N\to\infty$. Passing to a subsequence if necessary, we may assume that
\[
(A^\top)^{\tau_N}\xi_N \to \xi
\]
for some $\xi\in \R^d\backslash \{0\}$ as $N\to \infty$. Then, by the continuity of $\widehat{\mu}$, for every integer $n \ge 0$,
\[
|\widehat{\mu}((A^\top)^{-n}\xi)|
= \lim_{N \to \infty} |\widehat{\mu}((A^\top)^{\tau_N - n}\xi_N)|
\ge \liminf_{N \to \infty} |\widehat{\mu}(\xi_N)|
\ge \varepsilon,
\]
where, in the second inequality, we have used the fact that
\[
|\widehat{\mu}((A^\top)^u \xi_N)| \ge |\widehat{\mu}(\xi_N)|
\quad \text{for every positive integer } u.
\]

Next, let $1\leq \ell\leq m$.
By \eqref{e-e3.4},
\[
\sum_{n=-\tau_N}^{\infty}
\bigl\|\langle (A^\top)^{n}(A^\top)^{\tau_N}\xi_N,
a_\ell-a_1\rangle\bigr\|_{\mathbb{Z}}^2
< C.
\]

Since $\tau_N\to\infty$ as $N\to\infty$, for any $L\in\mathbb{N}$ there exists
$M_L\in\mathbb{N}$ such that $\tau_N>L$ for all $N>M_L$. Hence, for every $L\in \N$ and all
$N>M_L$,
\[
\sum_{n=-L}^{L}
\bigl\|\langle (A^\top)^{n}(A^\top)^{\tau_N}\xi_N,
a_\ell-a_1]\rangle\bigr\|_{\mathbb{Z}}^2
< C.
\]
Letting $N\to\infty$, we obtain for every $L\in\mathbb{N}$,
\[
\sum_{n=-L}^{L}
\bigl\|\langle (A^\top)^{n}\xi, a_\ell-a_1\rangle\bigr\|_{\mathbb{Z}}^2
< C.
\]
Therefore,
\[
\sum_{n=-\infty}^{\infty}
\bigl\|\langle (A^\top)^{n}\xi, a_\ell-a_1\rangle\bigr\|_{\mathbb{Z}}^2
< C.
\]
In particular,
\[
\sum_{n=0}^{\infty}
\bigl\|\langle (A^\top)^{-n}\xi, a_\ell-a_1\rangle\bigr\|_{\mathbb{Z}}^2
< C.
\]
This completes the proof.
\end{proof}

\subsection{Proof of Theorem \ref{thm-1.4'}}
\label{S-4.1.3'}

\begin{proof}[Proof of Theorem \ref{thm-1.4'}] The proof is lengthy and requires several delicate steps.
 We first prove that $\Phi=\{Ax+a_\ell\}_{\ell=1}^m$ is a generalized Pisot IFS on $\R^d$.

Since $\mu$ is non-Rajchman, Proposition~\ref{prop-3.8} provides a nonzero vector $\xi \in \mathbb{R}^d$ and a constant $\varepsilon>0$ such that
\begin{equation}
\label{e-nondecay}
\left|\widehat{\mu}\left((A^\top)^{-n}\xi\right)\right|\geq \varepsilon\quad \text{ for every integer }n\geq 0,
\end{equation}
and
\begin{equation}
\label{e-sumxi}
\sum_{n=0}^\infty
\bigl\|\langle (A^\top)^{-n}\xi,\, a_\ell -a_1\rangle\bigr\|_{\mathbb{Z}}^2
< \infty
\quad \text{for all } 1 \le \ell \le m.
\end{equation}

Set $B=(A^\top)^{-1}$. By Lemma~\ref{lem-Bnxi}, there exist a finite set $\{\beta_j : j \in \mathcal{J}\}$ of distinct eigenvalues of $B$ and nonzero vectors $\xi_j \in E_{\beta_j}$  ($j \in \mathcal{J}$), where $E_{\beta_j}$ denotes the generalized eigenspace corresponding to $\beta_j$, such that
$\xi=\sum_{j\in \mathcal J}\xi_j$ and, for every integer $n \ge 0$,
\begin{equation}
\label{e-bnxij}
B^n \xi_j=\sum_{u=0}^{\min\{n,\, m_j-1\}}
\binom{n}{u} \beta_j^{\,n-u} (B - \beta_j I)^u \xi_j,\qquad j\in \mathcal J,
\end{equation}
and
\begin{equation}
\label{e-bnxi}
B^n \xi
= \sum_{j \in \mathcal{J}} \sum_{u=0}^{\min\{n,\, m_j-1\}}
\binom{n}{u} \beta_j^{\,n-u} (B - \beta_j I)^u \xi_j,
\end{equation}
where $m_j$ is the smallest positive integer such that
\[
(B - \beta_j I)^{m_j} \xi_j = 0.
\]
Moreover, setting
\begin{equation}
\label{e-qdef}
q=\sum_{j\in \mathcal J}m_j,
\end{equation}
then by Lemma~\ref{lem-Bnxi}, $q\leq \sum_{j\in \mathcal J}\dim E_{\beta_j}\leq d$,  the vectors $\xi, B\xi,\ldots, B^{q-1}\xi$ are linearly independent, and
\begin{equation}
\label{e-Bqx}
B^{q}\xi=-\sum_{u=0}^{q-1}t_uB^u\xi,
\end{equation}
where the complex coefficients $t_u$ are given by the expansion
\begin{equation}
\label{e-Qx}
Q(x):=\prod_{j\in \mathcal J}(x-\beta_j)^{m_j}=x^q+\sum_{u=0}^{q-1}t_ux^u.
\end{equation}

Now, for  $j \in \mathcal{J}$ and  $1\leq \ell\leq m$, define
\[
\Qb_j(n)
= \sum_{u=0}^{m_j-1}
\binom{n}{u} \beta_j^{\,-u} (B - \beta_j I)^u \xi_j
\]
and
\[
Q_{j,\ell}(n)
= \langle \Qb_{j}(n),\, a_\ell -a_1\rangle
\]
for every integer $n\geq d$.
Each $Q_{j,\ell}$ is a  polynomial with complex coefficients. Moreover,  for $0\leq u\leq d$, $j\in \mathcal J$ and $1\leq \ell\leq m$, define
\begin{equation}
\label{e-bujl}
b_u^{(j,\ell)}=\left\{
\begin{array}{ll}
 \left\langle\beta_j^{\,-u} (B - \beta_j I)^u \xi_j,\, a_\ell-a_1\right\rangle,&\text{ if } 0 \leq u\leq m_j-1,\\
 0,& \text{ if } m_j\leq u\leq d.
\end{array}
\right.
\end{equation}
Then by \eqref{e-bnxij} and the definitions of $Q_{j,\ell}$ and $b_u^{j,\ell}$,
$$
Q_{j,\ell}(n)=\sum_{u=0}^d b_u^{(j,\ell)} \binom{n}{u} \quad \text{ for }n\geq d.
$$
Clearly the degree of each polynomial  $Q_{j,\ell}$ does not exceed $m_j-1$.
Rewrite these polynomials in the usual monomial basis,
$$
Q_{j,\ell}(n)=\sum_{u=0}^d c_u^{(j,\ell)}n^u.
$$
Then by Lemma \ref{lem-3.7}, there is an invertible $(d+1)\times (d+1)$ matrix $R_d$, with entries in $\Bbb Q$, such that
\begin{equation}
\label{e-b0}
\left(b_0^{(j,\ell)},b_1^{(j,\ell)},\ldots, b_{d}^{(j,\ell)}\right)^\top =R_d\left(c_0^{(j,\ell)},c_1^{(j,\ell)},\ldots,c_{d}^{(j,\ell)}\right)^\top\quad
\end{equation}
for all $j \in \mathcal{J}$ and $1\leq \ell\leq m$.

Note that by \eqref{e-bnxij}-\eqref{e-bnxi} and the definition of $Q_{j,\ell}$,
\begin{equation}
\label{e-xijaell}
\langle B^n\xi_j,\, a_\ell-a_1 \rangle= Q_{j,\ell}(n) \beta_j^n,\quad \text{and } \langle B^n\xi,\, a_\ell -a_1\rangle=\sum_{j\in \mathcal J} Q_{j,\ell}(n) \beta_j^n
\end{equation}
for all $n\geq d$, $j\in \mathcal J$, and $1\leq \ell\leq m$. Define
$$
\mathcal J_\ell=\{j\in \mathcal J\colon Q_{j,\ell}\neq 0\} \quad \text{ for }1\leq \ell\leq m,
$$ and
$$
\Lambda=\{1\leq \ell\leq m\colon \mathcal J_\ell\neq \emptyset\}.
$$
Observe that the IFS $\{Ax+a_\ell-a_1\}_{\ell=1}^m$ is also affinely irreducible. By the first identity in \eqref{e-xijaell} and Lemma~\ref{lem-3.10}(i),  for each $j\in \mathcal J$, there exists an index $\ell$ such that $Q_{j,\ell}\neq 0$. Hence
\begin{equation}
\label{e-mathcalJ}
\mathcal J=\bigcup_{\ell\in \Lambda}\mathcal J_\ell.
\end{equation}

By  \eqref{e-sumxi} and the second identity in \eqref{e-xijaell}, for every $\ell\in \Lambda$,
\begin{equation}
\label{e-fin1}
\sum_{n=0}^\infty \left\|\sum_{j\in \mathcal J_\ell} Q_{j,\ell}(n) \beta_j^n \right\|_\Z^2 = \sum_{n=0}^\infty \left\|\sum_{j\in \mathcal J} Q_{j,\ell}(n) \beta_j^n \right\|_\Z^2<\infty.
\end{equation}
Applying Theorem \ref{thm-Pisot}(i) yields that  $\{\beta_j\colon j\in \mathcal J_\ell\}$ is a Pisot tuple for each $\ell\in \Lambda$. Since a finite union of Pisot tuples is again a Pisot tuple,  \eqref{e-mathcalJ} implies that $\{\beta_j\colon j\in \mathcal J\}$ is a Pisot tuple.

Moreover, by \eqref{e-fin1} and Theorem~\ref{thm-Pisot}(ii)--(iii), for every $\ell \in \Lambda$,
\[
c_u^{(j,\ell)} \in \mathbb{Q}(\beta_j)
\quad \text{for all } j \in \mathcal{J} \text{ and } 0 \le u \le d,
\]
and if $\beta_j$ and $\beta_{j'}$ are conjugate over $\mathbb{Q}$ and $\sigma: \mathbb{Q}(\beta_j) \to \mathbb{Q}(\beta_{j'})$ is an isomorphism with $\sigma(\beta_j) = \beta_{j'}$, then
\[
\sigma(c_u^{(j,\ell)}) = c_u^{(j',\ell)}.
\]
These two properties also hold for $\ell \in \{1,\ldots,m\} \setminus \Lambda$, since in this case $Q_{j,\ell} = 0$ for all $j \in \mathcal{J}$, and hence $c_u^{(j,\ell)} = 0$ for all $j \in \mathcal{J}$ and $0 \le u \le d$.

By \eqref{e-b0} and the fact that $R_d$ is a rational matrix, the same properties also hold for $b_u^{(j,\ell)}$. More precisely,
\[
b_u^{(j,\ell)} \in \mathbb{Q}(\beta_j)
\quad \text{for all } j \in \mathcal{J},\ 0 \le u \le d,\ \text{and } 1 \le \ell \le m,
\]
and the same conjugacy relation holds for $b_u^{(j,\ell)}$.

Now for every $j\in \mathcal J$, define
\[
\eta_j = \beta_j^{\,1-m_j}(B - \beta_j I)^{m_j-1} \xi_j.
\]
Each $\eta_j$ is an eigenvector of $B$ corresponding to $\beta_j$. Moreover, by the definition of $b_u^{(j,\ell)}$ (see \eqref{e-bujl}), we have
\begin{equation}
\label{e-etaaell}
\langle\eta_j, \, a_\ell -a_1\rangle=b_{m_j-1}^{(j,\ell)} \in \mathbb{Q}(\beta_j)\quad \mbox{ for all }j\in \mathcal J\text{ and }1\leq \ell\leq m.
\end{equation}

Since the set $\{\beta_j : j \in \mathcal{J}\}$ is a Pisot tuple, we can partition $\mathcal{J}$ into disjoint nonempty subsets $\mathcal{J}^{(1)}, \ldots, \mathcal{J}^{(r)}$ such that, for each $1 \leq i \leq r$, the elements of $\{\beta_j : j \in \mathcal{J}^{(i)}\}$ are mutually conjugate over $\mathbb{Q}$, and no element of $\{\beta_j : j \in \mathcal{J} \setminus \mathcal{J}^{(i)}\}$ is conjugate to any element of $\{\beta_j : j \in \mathcal{J}^{(i)}\}$. Clearly, each set $\{\beta_j : j \in \mathcal{J}^{(i)}\}$ is a Pisot tuple.

We claim that for each $1\leq i\leq r$,
\begin{equation}
\label{e-mjj'}
m_{j}=m_{j'} \quad \mbox{ for all }j,j'\in \mathcal J^{(i)}.
\end{equation}

To prove the above claim, let $1 \le i \le r$. Choose $j_0 \in \mathcal{J}^{(i)}$ such that
\begin{equation}
\label{e-mj0}
m_{j_0} = \max\{ m_j : j \in \mathcal{J}^{(i)} \}.
\end{equation}
Applying Lemma~\ref{lem-3.10}(ii) to the IFS $\{Ax+a_\ell-a_1\}_{\ell=1}^m$, we see that there exists $\ell_0 \in \{1,\ldots,m\}$ such that
\[
b_{m_{j_0}-1}^{(j_0,\ell_0)} = \langle \eta_{j_0}, a_{\ell_0}-a_1 \rangle \ne 0,
\]
which implies that $Q_{j_0,\ell_0}$ is a polynomial of degree $m_{j_0} - 1$. By the conjugacy relation between the coefficients of $Q_{j_0,\ell_0}$ and $Q_{j,\ell_0}$, it follows that $\deg(Q_{j,\ell_0}) = m_{j_0} - 1$ for every $j \in \mathcal{J}^{(i)}$. On the other hand, by the definition of $Q_{j,\ell}$, we have $\deg(Q_{j,\ell}) \le m_j - 1$ for all $j\in \mathcal J$ and $1\leq \ell\leq m$. Therefore, $m_j\geq m_{j_0}$ for every $j\in \mathcal{J}^{(i)}$. In view of \eqref{e-mj0}, we obtain
\[
m_j = m_{j_0} \quad \text{for all } j \in \mathcal{J}^{(i)},
\]
which completes the proof of the claim.

In what follows we consider the following two cases separately:
\begin{itemize}
\item[(C1)]
$\{\beta_j\colon j\in \mathcal J\}$ is a Pisot tuple of type I;
\item[(C2)]
$\{\beta_j\colon j\in \mathcal J\}$ is a Pisot tuple of type II.
\end{itemize}

First, suppose that (C1) holds. Then at least one of $\{\beta_j : j \in \mathcal{J}^{(1)}\}, \ldots, \{\beta_j : j \in \mathcal{J}^{(r)}\}$ is a Pisot tuple of type~I. Without loss of generality, we may assume that $\{\beta_j : j \in \mathcal{J}^{(1)}\}$ is a Pisot tuple of type~I.

Choose $j_1\in \mathcal J^{(1)}$. By \eqref{e-mjj'},  $m_j=m_{j_1}$ for all $j\in \mathcal J^{(1)}$. For each $1\leq \ell\leq m$, since $\langle\eta_{j_1}, \, a_\ell-a_1 \rangle\in \Bbb Q(\beta_{j_1})$ by \eqref{e-etaaell}, we can choose  a polynomial $Q_\ell\in \Bbb Q[x]$ such that $$\langle\eta_{j_1}, \, a_\ell-a_1 \rangle=Q_\ell(\beta_{j_1}).$$
By \eqref{e-etaaell} and the conjugacy relation for $b_{u}^{(j,\ell)}$,  for every $j\in \mathcal J^{(1)}$,  letting $\sigma: \mathbb{Q}(\beta_{j_1}) \to \mathbb{Q}(\beta_j)$ be an isomorphism with $\sigma(\beta_{j_1})=\beta_{j}$, it follows that
\[
\langle \eta_j, a_\ell-a_1 \rangle
= b_{m_j-1}^{(j,\ell)}
= b_{m_{j_1}-1}^{(j,\ell)}
= \sigma\bigl(b_{m_{j_1}-1}^{(j_1,\ell)}\bigr)
= \sigma\bigl(\langle \eta_{j_1}, a_\ell-a_1 \rangle\bigr)
= \sigma\bigl(Q_\ell(\beta_{j_1})\bigr)
= Q_\ell(\beta_j).
\]

Choose $s \in \mathbb{N}$ such that $sQ_\ell \in \mathbb{Z}[x]$ for all $\ell$, and set $P_\ell = sQ_\ell$. Then
$$
\langle s\eta_j, \, a_\ell -a_1\rangle=P_\ell(\beta_j)\quad \text{ for all }j\in \mathcal J^{(1)}.
$$
Let $k' = \#(\mathcal{J}^{(1)})$. Then $\Phi = \{Ax + a_\ell\}_{\ell=1}^m$ is a generalized Pisot IFS of type I, with parameters $k'$, $\{s\eta_j : j \in \mathcal{J}^{(1)}\}$, and $P_1, \ldots, P_m$.

For the remainder of the proof, assume that condition~(C2) holds. The same construction as above shows that $\Phi = \{Ax + a_\ell\}_{\ell=1}^m$ is a generalized Pisot IFS with parameters $k'$, $\{s\eta_j : j \in \mathcal{J}^{(1)}\}$, and $P_1, \ldots, P_m$, although it is not of type~I.

Since the set $\{\beta_j : j \in \mathcal{J}\}$ is a Pisot tuple of type~II, it follows that each set $\{\beta_j : j \in \mathcal{J}^{(i)}\}$, for $1 \le i \le r$, is also a Pisot tuple of type~II. Therefore, for each $1 \le i \le r$ and every $j \in \mathcal{J}^{(i)}$, the set $\{\beta_{j'} : j' \in \mathcal{J}^{(i)}\}$ consists precisely of all conjugates of $\beta_j$ over $\mathbb{Q}$.

We now prove that
\begin{equation}
\label{e-integer}
\langle B^n \xi,\, a_\ell-a_1 \rangle \in \mathbb{Q}\quad \text{  for all }n \in \mathbb{N} \cup \{0\}\text { and }1 \le \ell \le m.
\end{equation}
 To prove this, fix $n \in \mathbb{N} \cup \{0\}$ and $1 \le \ell \le m$. By \eqref{e-bnxi} and \eqref{e-bujl},
\begin{equation}
\label{e-bnxi1}
\langle B^n \xi,\, a_\ell-a_1 \rangle
= \sum_{j \in \mathcal{J}} \sum_{u=0}^{\min\{n,\, m_j-1\}} b_u^{(j,\ell)} \binom{n}{u}
= \sum_{i=1}^r \sum_{j \in \mathcal{J}^{(i)}} \sum_{u=0}^{\min\{n,\, m_j-1\}} b_u^{(j,\ell)} \binom{n}{u}.
\end{equation}
By the conjugacy relation for $b_u^{(j,\ell)}$ and Lemma \ref{lem-Galois}(ii), we see that
\begin{equation}
\label{e-sumji}
\sum_{j \in \mathcal{J}^{(i)}} b_u^{(j,\ell)} \in \mathbb{Q}
\quad \text{for all } 1 \le i \le r,\; 0 \le u \le d,\; \text{and } 1 \le \ell \le m.
\end{equation}
Since the numbers $m_j$ (for $j \in \mathcal{J}^{(i)}$) are equal for each $i$, it follows from \eqref{e-bnxi1} and \eqref{e-sumji} that
\[
\langle B^n \xi, \, a_\ell-a_1 \rangle \in \mathbb{Q}.
\]
This proves \eqref{e-integer}.

By \eqref{e-integer},   we can choose a suitable positive integer $\kappa$ such that
\begin{equation}
\label{e-hnell}
h_{n,\ell}:=\langle \kappa B^n \xi,\, a_\ell -a_1\rangle \in \mathbb{Z}\quad \text{ for all }0\leq n\leq  q-1 \text{ and }1\leq \ell\leq m,
\end{equation}
where $q$ is defined as in \eqref{e-qdef}.

Define a $d\times q$ real matrix $S$ by
$$S=(\kappa\xi,\kappa B\xi,\ldots, \kappa B^{q-1}\xi),$$
that is, the $n$-th column of $S$ is  $\kappa B^{n-1}\xi$ for every  $1\leq n\leq q$. By Lemma~\ref{lem-Bnxi}(ii), $S$ has rank $q$.

Let $Q$ be the polynomial defined as in \eqref{e-Qx}, that is,
$$
Q(x)=x^q+\sum_{u=0}^{q-1}t_ux^u=\prod_{j\in \mathcal J}(x-\beta_j)^{m_j}=\prod_{i=1}^r\sum_{j\in \mathcal J^{(i)}} (x-\beta_j)^{m_j}.
$$
For each $1\leq i\leq r$, since $\{\beta_j\colon j\in \mathcal J^{(i)}\}$ is a Pisot tuple of type II and the numbers $m_j$ (for $j\in \mathcal J^{(i)}$) are equal,  it follows from Lemma~\ref{lem-Galois}(i) that $\sum_{j\in \mathcal J^{(i)}} (x-\beta_j)^{m_j}\in \Z[x]$. Therefore  $Q\in \Z[x]$, and hence $t_u\in \Z$ for all $0\leq u\leq q-1$.
Let $M$ be the companion matrix of $Q$, that is,
 \[
M =
\begin{pmatrix}
0 & 0 & 0 & \cdots & 0 & -t_{0} \\
1 & 0 & 0 & \cdots & 0 & -t_{1} \\
0 & 1 & 0 & \cdots & 0 & -t_{2} \\
\vdots & \ddots & \ddots & \ddots & \vdots & \vdots \\
0 & 0 & \cdots & 1 & 0 & -t_{q-2} \\
0 & 0 & \cdots & 0 & 1 & -t_{q-1}
\end{pmatrix}.
\]
Thus $M$ is an integer matrix whose eigenvalues are the $\beta_j$ for $j\in \mathcal J$.
By \eqref{e-Bqx} and \eqref{e-Qx},  we have
$$
B(\kappa\xi,\kappa B\xi,\ldots, \kappa B^{q-1}\xi)=(\kappa \xi,\kappa B\xi,\ldots, \kappa B^{q-1}\xi)M,
$$
that is, $BS=SM$. Taking transpose gives $S^\top A^{-1}=M^\top S^\top$, and hence
\begin{equation}
\label{e-3.17}
S^\top A=(M^\top)^{-1} S^\top.
\end{equation}
Moreover, for each $1\leq \ell\leq m$,
\begin{equation}
\label{e-3.20}
\begin{split}
S^\top (a_\ell-a_1)&=(\kappa\xi,\kappa B\xi,\ldots, \kappa B^{q-1}\xi)^\top (a_\ell-a_1)\\
&=(\langle \kappa \xi,\, a_\ell-a_1\rangle, \ldots, \langle \kappa B^{q-1} \xi,\, a_\ell-a_1\rangle)^\top\\
&=\hb_\ell\in \Z^q,
\end{split}
\end{equation}
where $\hb_\ell:=(h_{0,\ell}, \ldots, h_{q-1,\ell})^\top,$ and $h_{n,\ell}$ is defined as in \eqref{e-hnell}.

Finally, we define an affine map $T\colon \R^d\to \R^q$ by
$$Tx=S^\top x+u,\quad \text{where }u:=-\sum_{n=0}^\infty S^\top A^n a_1.$$
Since $S$ has rank $q$, $T$ is surjective. Using \eqref{e-3.17}, we obtain
\begin{equation}
\label{e-idenu}
S^\top a_1+u-(M^\top)^{-1}u=S^\top a_1-\sum_{n=0}^\infty S^\top A^n a_1 +\sum_{n=0}^\infty S^\top A^{n+1} a_1=0.
\end{equation}
We next prove  that \(T\mu\), the pushforward of \(\mu\) under $T$, is non-Rajchman and is the self-affine measure associated with the integral IFS
\[
\Psi = \{\psi_\ell(x) = (M^\top)^{-1}x + \hb_\ell\}_{\ell=1}^m
\]
and the probability vector \(\pb = (p_\ell)_{\ell=1}^m\). Moreover, the affine image $T(K)$ of $K$, where $K$ is the attractor of $\Phi$, is the attractor of $\Psi$.

To prove that \(T \mu\) is non-Rajchman, observe that for \(\zeta \in \mathbb{R}^q\),
\begin{align*}
\widehat{T \mu}(\zeta)
&= \int e^{2\pi i \langle \zeta,\, z\rangle} \, d(T \mu)(z)
= \int e^{2\pi i \langle \zeta, \,S^\top x+u\rangle} \, d\mu(x)\\
&= e^{2\pi i \langle \zeta,\, u\rangle}\int e^{2\pi i \langle S\zeta,\, x\rangle} \, d\mu(x)
=e^{2\pi i \langle \zeta,\, u\rangle} \widehat{\mu}(S\zeta),
\end{align*}
so $|\widehat{T \mu}(\zeta)|=|\widehat{\mu}(S\zeta)|$.
Let \(e_1 = (1,0,\ldots,0)^\top\) denote the unit vector in \(\mathbb{R}^q\) with first coordinate equal to \(1\) and all others equal to \(0\). Then
$$
|\widehat{T \mu}(M^ne_1)|=|\widehat{\mu}(SM^ne_1)|=|\widehat{\mu}(B^nSe_1)|=|\widehat{\mu}(B^n\xi)|\geq \varepsilon, \quad  n\geq 0,
$$
where the second equality follows from $SM=BS$, and the last inequality follows from \eqref{e-nondecay}. Since $|M^ne_1|\to \infty$ as $n\to \infty$,   $T \mu$ is non-Rajchman.

To see that \(T\mu\) is the self-affine measure associated with  \(\Psi\) and \(\pb\), observe that for each \(1 \leq \ell \leq m\) and \(x \in \mathbb{R}^d\),
\begin{align*}
T\circ \phi_\ell(x)&=T(Ax+a_\ell)=S^\top(Ax+a_\ell)+u=S^\top A x+S^\top a_\ell+u\\
&=(M^\top)^{-1}(S^\top x+u)+S^\top (a_\ell-a_1)+S^\top a_1+u-(M^\top)^{-1}u \quad(\text{by \eqref{e-3.17}})\\
&=(M^\top)^{-1}Tx+\hb_\ell \quad(\text{by \eqref{e-3.20} and \eqref{e-idenu}})\\
&=\psi_\ell\circ T(x).
\end{align*}
That is,
\begin{equation}
\label{e-Tcomm}
\psi_\ell \circ T = T \circ \phi_\ell,\qquad \ell=1,\ldots, m.
\end{equation}
 Hence, for any Borel set \(E \subset \mathbb{R}^q\),
\begin{align*}
(T \mu)(E)
&= \mu\big(T^{-1} E\big)\\
&= \sum_{\ell=1}^m p_\ell \, \mu\big(\phi_\ell^{-1}\big(T^{-1} E\big)\big) \\
&= \sum_{\ell=1}^m p_\ell \, \mu\big(T^{-1} (\psi_\ell^{-1} E)\big) \\
&= \sum_{\ell=1}^m p_\ell \, (T \mu)(\psi_\ell^{-1} E).
\end{align*}
That is, $T\mu$ is the self-affine measure associated with \(\Psi\) and \(\pb\).

To see that $T(K)$ is the attractor of $\Psi$, from the identity $K=\bigcup_{\ell=1}^m \phi_\ell(K)$ we obtain
\[
T(K)=\bigcup_{\ell=1}^m T(\phi_\ell(K))
=\bigcup_{\ell=1}^m \psi_\ell(T(K)),
\]
where in the second equality we  used \eqref{e-Tcomm}. Hence, $T(K)$ is the attractor of $\Psi$.

Finally, set $V={\rm span}_\R\{\xi, B\xi,\ldots, B^{q-1}\xi\}$. We show that the following properties hold.
\begin{itemize}
\item[(a)] $V$ is a $B$-invariant subspace of dimension $q$;
\item[(b)] The restriction of $T$ on $V$, denoted by $T|_V$, is invertible. That is, the mapping $T\colon V\to \R^q$ is a homeomorphism;
\item[(c)] Let  $P_V\colon \R^d\to V$ denote the orthogonal projection onto $V$, and let $$\Phi_V:=\{\phi_{V,\ell}(x)=P_VAx+P_V a_\ell\}_{\ell=1}^m$$ be the corresponding projected IFS on $V$.  Then $\Phi_V$ is conjugated to $\Psi$ via the affine map $T|_V$. That is,  $$(T|_V)\circ \phi_{V,\ell}\circ (T|_V)^{-1}=\psi_\ell$$ for every $1\leq \ell\leq m$.
\end{itemize}

Part (a) follows directly from the linear independence of the vectors $\xi, B\xi, \ldots, B^{q-1}\xi$.

To prove (b), note that $S$ is a $d\times q$ matrix of rank $q$. For $x_1,\ldots, x_q\in \R$,
\begin{align*}
(S^\top|_V)(x_1\xi+x_2B\xi+\cdots +x_qB^{q-1}\xi)&=S^\top(x_1\xi+x_2B\xi+\cdots +x_qB^{q-1}\xi).\\
&=(S^\top S)(x_1,\ldots, x_q)^\top.
\end{align*}
Since $\operatorname{rank}(S)=q$, it follows that $S^\top S$ also has rank $q$ (indeed, $S$ and $S^\top S$ has the same kernel). Hence $S^\top|_V$ is invertible. Consequently, $T|_V$ is invertible.

To prove (c),  we first show that $S^\top\circ P_V=S^\top$ on $\R^d$. For $y \in \R^d$, we have
\begin{align*}
(S^\top\circ P_V)(y)
&=k(\langle \xi, P_Vy\rangle, \ldots, \langle B^{q-1}\xi, P_Vy\rangle)\\
&=k(\langle P_V\xi, y\rangle, \ldots, \langle P_VB^{q-1}\xi, y\rangle)\\
&=k(\langle \xi, y\rangle, \ldots, \langle B^{q-1}\xi, y\rangle)\\
&=S^\top(y),
\end{align*}
where the third equality follows from the fact that $B^j\xi\in V$ for $0\leq j\leq q-1$. Hence $S^\top\circ P_V=S^\top$.

Let $x\in \R^q$ and $1\leq \ell\leq m$. Then
\begin{align*}
(T|_V)&\circ \phi_{V,\ell}\circ (T|_V)^{-1}(x)\\
&=(S^\top|_V)\bigl(P_VA (T|_V)^{-1}(x)+P_V a_\ell \bigr)+u\\
&=S^\top \bigl(P_VA (T|_V)^{-1}(x)+P_V a_\ell \bigr)+u\\
&=S^\top P_VA (T|_V)^{-1}(x)+S^\top P_V a_\ell+u\\
&=S^\top A (T|_V)^{-1}(x)+S^\top a_\ell+u \qquad(\text{since } S^\top P_V=S^\top)\\
&=(M^\top)^{-1}S^\top (T|_V)^{-1}(x)+\hb_\ell +S^\top a_1+u
\quad\text{(by \eqref{e-3.17}--\eqref{e-3.20})}\\
&=(M^\top)^{-1}(T|_V) (T|_V)^{-1}(x)+\hb_\ell +S^\top a_1+u-(M^\top)^{-1}u\\
&=(M^\top)^{-1}x+\hb_\ell\qquad \text{(by \eqref{e-idenu})}\\
&=\psi_\ell(x),
\end{align*}
as desired. This proves (c), and hence $\Psi$ is an integral factor of $\Phi$.
\end{proof}

By adapting the arguments in the proof of Theorem~\ref{thm-1.4'}, we obtain the following result, which implies that an affinely irreducible generalized Pisot IFS of type II has at least one integral factor.
\begin{lem}
\label{lem-5.6'}
Let $\Phi=\{\phi_\ell(x)=Ax+a_\ell\}_{\ell=1}^m$ be an affinely irreducible generalized Pisot IFS on $\mathbb{R}^d$, with parameters $\beta_1,\dots,\beta_k$, $\zeta_1,\dots,\zeta_k$, and $P_1,\dots,P_m$ as in Definition~\ref{defn-1.2}. Suppose additionally that $\{\beta_1,\dots,\beta_k\}$ is a Pisot tuple of type II. Then $\Phi$ has an integral factor.
\end{lem}
\begin{proof} We  follow mainly the arguments given near the end of the proof of Theorem~\ref{thm-1.4'} (specifically, the last two pages). To avoid repetition, we  only give a brief sketch.

Write $B=(A^\top)^{-1}$ and set $\xi=\sum_{j=1}^k \zeta_j$. By Lemma~\ref{lem-conjugate}, $\xi\in \mathbb{R}^d\setminus \{0\}$. Since $\{\beta_1,\dots,\beta_k\}$ is a Pisot tuple of type II, for every integer $n\geq 0$ and $1\leq \ell\leq m$,
\begin{equation*}
h_{n,\ell}:=\langle B^n\xi, a_\ell-a_1 \rangle
=\sum_{j=1}^k \beta_j^n \langle \zeta_j, a_\ell-a_1 \rangle
=\sum_{j=1}^k \beta_j^n P_\ell(\beta_j)\in \mathbb{Z}.
\end{equation*}

Let $Q$ be the polynomial defined by
\begin{equation*}
Q(x)=\prod_{j=1}^k(x-\beta_j)
=:x^k+t_{k-1}x^{k-1}+\cdots+t_1x+t_0.
\end{equation*}
Then $Q\in \mathbb{Z}[x]$, since $\{\beta_1,\dots,\beta_k\}$ is a Pisot tuple of type II.
By Lemma~\ref{lem-Bnxi}, the vectors $\xi, B\xi,\dots, B^{k-1}\xi$ are linearly independent over $\mathbb{R}$, and
\begin{equation*}
B^k\xi=-\sum_{j=0}^{k-1}t_j B^{j}\xi.
\end{equation*}
It follows that
\begin{equation*}
B(\xi,B\xi,\ldots, B^{k-1}\xi)=(\xi,B\xi,\ldots, B^{k-1}\xi)M,
\end{equation*}
where $M\in \mathrm{GL}_k(\mathbb{Z})$ is the companion matrix of $Q$. Setting $S=(\xi,B\xi,\ldots, B^{k-1}\xi)$, we have $BS=SM$. Moreover,
\begin{align*}
S^\top (a_\ell-a_1)
&=(\langle \xi, a_\ell-a_1\rangle, \ldots, \langle B^{k-1} \xi, a_\ell-a_1\rangle)^\top \notag\\
&=\hb_\ell\in \mathbb{Z}^k,
\end{align*}
where $\hb_\ell:=(h_{0,\ell}, \ldots, h_{k-1,\ell})^\top$.

Let $\Psi$ be the integral IFS on $\mathbb{R}^k$ defined by
\begin{equation}\label{eq:Psi}
\Psi=\{\psi_\ell(x)=(M^\top)^{-1}x+\hb_\ell\}_{\ell=1}^m.
\end{equation}
Define an affine map $T:\mathbb{R}^d\to \mathbb{R}^k$ by
\begin{equation}\label{eq:T}
Tx=S^\top x+u,\quad \text{where } u:=-\sum_{n=0}^\infty S^\top A^n a_1.
\end{equation}
Since $S$ has rank $k$, $T$ is surjective.

Let $V=\operatorname{span}_{\mathbb{R}}\{\xi, B\xi,\ldots, B^{k-1}\xi\}$. Then $\dim V=k$. As in the argument for part (b) in the proof of Theorem~\ref{thm-1.4'},  the restriction of $T$ to $V$, denoted by $T|_V$, is invertible. Let
\begin{equation}\label{eq:PhiV}
\Phi_V=\{\phi_{V,\ell}(x)=P_VAx+P_Va_\ell\}_{\ell=1}^m
\end{equation}
be the projected IFS of $\Phi$ onto $V$.  Arguing as at the end of the proof of Theorem~\ref{thm-1.4'}, one obtains that
\begin{equation}\label{eq:conjugacy}
(T|_V)\circ \phi_{V,\ell}\circ (T|_V)^{-1}=\psi_\ell
\end{equation}
for every $1\leq \ell\leq m$, which implies that $\Psi$ is a factor of $\Phi$.
\end{proof}

\section{Proof of Theorem \ref{thm-1.5}}
\label{S-6}

In this section, we prove Theorem \ref{thm-1.5}.  We first present two useful lemmas in Section~\ref{S-5.1} and a key proposition in Section~\ref{S-5.2}. The proof of Theorem~\ref{thm-1.5} is given in Section~\ref{S-5.3}.

\subsection{Two useful lemmas}
\label{S-5.1}
Let ${\Bbb T}^d=\Bbb R^d/\Bbb Z^d$ denote the $d$-dimensional torus. For $\theta=(\theta_1,\ldots, \theta_d)\in \Bbb T^d$, let $T_{\theta}\colon \Bbb T^d\to \Bbb T^d$ be the translation map defined by
$$
T_{\theta}(x_1,\ldots, x_d)=(x_1+\theta_1,\ldots, x_d+\theta_d)\quad (\mbox{mod}\;\; 1).
$$

For $\Lambda\subset \N$, the {\it upper density} of $\Lambda$ in $\N$ is defined as
$$
\overline{d}(\Lambda)=\limsup_{n\to \infty}\frac{\#(\Lambda\cap [1,n])}{n},
$$
where $\#$ stands for cardinality.

\begin{lem}
\label{lem-td}
Let ${\theta}=(\theta_1,\ldots, \theta_d)\in \Bbb T^d$ such that the numbers $\theta_1,\ldots, \theta_d$ and $1$ are rational independent; that is, $\sum_{i=1}^d k_i \theta_i$ is not an integer for any collection of integers $k_1,\ldots, k_d$,  except when $k_1=k_2=\ldots=k_d=0$. Let $\Lambda\subset \N$, and let $E\subset \Bbb T^d$ be the closure of  the set $\{T_{\theta}^n({0,\ldots,0})\colon n\in \Lambda\}$.  Then
$$
\mathcal L^d (E)\geq \overline{d}(\Lambda),
$$
where $\mathcal L^d$ denotes the Lebesgue measure on $\Bbb T^d$. In particular, $\mathcal L^d (\widetilde{E})\geq \overline{d}(\Lambda)$, where
$\widetilde{E}$ denotes the set of accumulation points of $\{T_{\theta}^n({0,\ldots, 0})\colon n\in \Lambda\}$.
\end{lem}
\begin{proof}
The result is well known in the case  $d=1$ (see, e.g., \cite[Lemma 7.2]{Wu2019}). It is most likely known in the general case as well, but we have not been able to find a reference, so  a proof is provided for the reader's convenience.

Endow $\Bbb T^d$ with the flat metric $\rho$, defined by
$$
\rho(x, y)=\|x-y\|_{\Z^d}.
$$
For any $\epsilon>0$, let
$$
V_\epsilon(E):=\{y \in \Bbb T^d\colon \rho(y, x)<\epsilon \mbox{ for some } x\in E\}.
$$
be the $\epsilon$-neighborhood of $E$. Let $\chi_{V_\epsilon(E)}$ denote the characteristic function of the set $V_\epsilon(E)$.  Since $T_{\theta}$ is ergodic with respect to $\mathcal L^d$ (see, e.g., \cite[Proposition 4.2.2]{KH1995}), by the Birkhoff ergodic theorem,
\begin{equation}
\label{e-2.1}
\lim_{n\to \infty}\frac{1}{n}\sum_{i=0}^{n-1} \chi_{V_\epsilon(E)}(T^i_{\theta}(x))=\mathcal L^d(V_\epsilon(E))\quad \mbox{ for $\mathcal L^d$-a.e.~$x\in \Bbb T^d$}.
\end{equation}
Moreover, for every $x\in \Bbb T^d$ such that $\rho(x, (0,\ldots, 0))<\epsilon$, we have
$$
\rho(T^i_{\theta}(x), T^i_{\theta}(0,\ldots,0))=\rho(x, (0,\ldots,0))<\epsilon \quad \mbox{ for  all }i\geq 0;
$$
which implies that $T^i_{\theta}(x)\in V_\epsilon(E)$ for every $i\in \Lambda$, and therefore
$$
\limsup_{n\to \infty}\frac{1}{n}\sum_{i=0}^{n-1} \chi_{V_\epsilon(E)}(T^i_{\theta}(x))\geq \limsup_{n\to \infty}\frac{\#(\Lambda\cap [1,n])}{n}=\overline{d}(\Lambda).
$$
This, together with  \eqref{e-2.1}, yields that $\mathcal L^d(V_\epsilon(E))\geq \overline{d}(\Lambda)$. Letting $\epsilon\to 0$ yields  $$\mathcal L^d(E)\geq \overline{d}(\Lambda).$$
 Since $\widetilde{E}$ differs from $E$ by at most a countable set, we conclude that
 \[\mathcal L^d(\widetilde{E})= \mathcal L^d(E)\geq \overline{d}(\Lambda).\qedhere\]
\end{proof}

The following result is well known (see, e.g., \cite[Proposition~6]{Krantz2018}).
\begin{lem}
\label{lem-real analytic}
Let $f\colon \Bbb T^d\to \R$ be real analytic. Suppose that $f$ vanishes on a subset of $\Bbb T^d$ of positive Lebesgue measure. Then $f$ is identically zero.
\end{lem}

\subsection{A technical proposition}
\label{S-5.2}
In this subsection, we provide  a technical proposition which plays an essential role in the proof of Theorem \ref{thm-1.5}.

Let $\beta_1,\ldots, \beta_k$ be a Pisot tuple of type I such that  \(\beta_1,\ldots,\beta_k\) are Galois conjugates over \(\mathbb{Q}\), and
let \(\beta_{k+1},\ldots,\beta_{k'}\) be the remaining Galois conjugates. For \(1 \leq j \leq k'\), write
\[
\beta_j = r_j e^{2\pi i \alpha_j},
\]
where \(r_j > 0\) and \(\alpha_j \in [0,1)\).

The main result in this subsection is the following.

\begin{prop}\label{prop-4.1}
 Under the above setting, assume in addition that there exist $\theta_1,\ldots,\theta_s \in [0,1)$ such that
$1,\theta_1,\ldots,\theta_s$ are rationally independent, and that there exist integers
$t_{j,u}$ ($1 \leq j \leq k'$,  $1 \leq u \leq s$) satisfying
\begin{equation}
\label{e-e5.2}
\alpha_j \equiv \sum_{u=1}^s t_{j,u}\theta_u \pmod{1},
\qquad j=1,\ldots,k'.
\end{equation}
Let $P_\ell\in \Z[x]$, $1\leq \ell \leq m$, be such that $P_\ell(\beta_j)\neq P_{\ell'}(\beta_j)$ for all $1\leq j\leq k'$ and $\ell\neq \ell'$.
 For $n\in \mathbb{Z}$, define $v_n\in \C^m$ by
$$
v_n=\left(\exp\left(2\pi i\sum_{j=1}^k\beta_j^{n} P_1(\beta_j)\right),\ldots, \exp\left(2\pi i\sum_{j=1}^k\beta_j^{n} P_m(\beta_j)\right)\right).
$$
Let $\Lambda$ be an infinite subset of $\N$ with positive upper density. Then
$$
{\rm span}_\C(\{v_n:\; n\in \Lambda\})=\C^m \quad \text{and} \quad {\rm span}_\C(\{v_{-n}:\; n\in \Lambda\})=\C^m.
$$
\end{prop}

Before proving the above proposition, we first give a lemma.

\begin{lem}
\label{lem-i}
Under the conditions of Proposition \ref{prop-4.1},  there exists $i\in \{1,\ldots, k\}$ such that  for every $n\in \N$,
$$
\beta_{j_1}\cdots \beta_{j_q}=\beta_i^n \mbox{ for some }j_1,\ldots, j_q\in \{1,\ldots k\}\Longleftrightarrow q=n \text{ and } j_1=\cdots=j_n=i.
$$
\begin{proof}
Choose \( i \in \{1,\ldots,k\} \) such that the following two conditions are satisfied:
\begin{itemize}
\item[(1)] \( |\beta_i| = \min_{1\le j\le k} |\beta_j| \);
\item[(2)] for any other \( j \in \{1,\ldots,k\} \) with \( |\beta_j| = |\beta_i| \), if any, we have
\[
(t_{i,1}, \ldots, t_{i,s}) > (t_{j,1}, \ldots, t_{j,s}),
\]
where \( > \) denotes the lexicographic (dictionary) order on \( \mathbb{Z}^s \); that is,
\[
(n_1,\ldots,n_s)>(n_1',\ldots,n_s')
\]
if there exists \(q\in\{1,\ldots,s\}\) such that
\[
n_q>n_q',
\quad \text{and } n_j=n_j' \text{ for all }j<q.
\]\end{itemize}

Clearly, such an index \( i \) exists and is unique. Now suppose that
\[
\beta_{j_1}\cdots \beta_{j_q}=\beta_i^n
\quad \text{for some } n\in \mathbb{N}
\text{ and } j_1,\ldots, j_q\in \{1,\ldots,k\}.
\]
By condition~(1), we have \( q\le n \).
We claim that in fact \( q=n \).
Suppose, on the contrary, that \( q<n \).

Since \( \beta_1,\ldots,\beta_{k'} \) are the roots of an  irreducible integer polynomial of degree \( k' \),
there exists a Galois permutation \( \sigma \) of
\( \{\beta_1,\ldots,\beta_{k'}\} \) such that
\( \sigma(\beta_i)=\beta_{j_0} \), where \( j_0\in \{1,\ldots,k\} \) satisfies
\[
|\beta_{j_0}|=\max_{1\le j\le k} |\beta_j|.
\]
Applying \( \sigma \) to the identity $\beta_{j_1}\cdots \beta_{j_q}=\beta_i^n$ yields
\[
\sigma(\beta_{j_1})\cdots \sigma(\beta_{j_q})=\sigma(\beta_{i})^n=\beta_{j_0}^n.
\]
However,
\[
|\sigma(\beta_{j_1})\cdots \sigma(\beta_{j_q})|
\le |\beta_{j_0}|^{\,q}
< |\beta_{j_0}|^{\,n}
=|\beta_{j_0}^n|,
\]
which is a contradiction. Hence \( q=n \).

Since \( \beta_{j_1}\cdots \beta_{j_n}=\beta_i^n \), taking absolute values and using
condition~(1), we obtain
\begin{equation}
\label{e-modulus}
|\beta_{j_1}|=\cdots=|\beta_{j_n}|=|\beta_i|.
\end{equation}
On the other hand, comparing the arguments of
\( \beta_{j_1}\cdots \beta_{j_n} \) and \( \beta_i^n \) and using \eqref{e-e5.2}, we obtain
\[
\sum_{u=1}^s\sum_{q=1}^n t_{j_q,u}\theta_u
\equiv
n\sum_{u=1}^s t_{i,u}\theta_u \pmod 1,
\]
which implies that
\[
\left(\sum_{q=1}^n t_{j_q,1},\ldots, \sum_{q=1}^n t_{j_q,s}\right)
=
(nt_{i,1}, \ldots, nt_{i,s}),
\]
since $1, \theta_1,\ldots, \theta_s$ are rationally independent. Together with condition~(2), this implies that
\[
t_{j_q,u}=t_{i,u}
\quad \text{for all } 1\le q\le n \text{ and } 1\le u\le s.
\]
Hence, the arguments of \( \beta_{j_q} \), \( 1\le q\le n \), are all equal to that of
\( \beta_i \). Combined with \eqref{e-modulus}, this yields
\[
\beta_{j_1}=\cdots=\beta_{j_n}=\beta_i,
\]
that is, \( j_1=\cdots=j_n=i \).
\end{proof}

\end{lem}
\begin{proof}[Proof of Proposition \ref{prop-4.1}]
We first prove that
\begin{equation}
\label{e-vneativen}
{\rm span}_\C(\{v_{-n}:\; n\in \Lambda\})=\C^m.
\end{equation}

Suppose on the contrary that \eqref{e-vneativen} does not hold.
Then there exists a nonzero vector $$w=(w_1,\ldots,w_m)\in \mathbb{C}^m$$ such that
\begin{equation}\label{e-4.1}
    \sum_{\ell=1}^m w_\ell \exp\left(2\pi i \sum_{j=1}^{k}\beta_j^{-n} P_\ell(\beta_j)\right)=0 \quad \text{for  all } n \in \Lambda.
\end{equation}

Let $\Sigma_*=\bigcup_{n=0}^\infty \{1,\ldots, k\}^n$. For any word $J=j_1\ldots j_n\in \Sigma_*$,
write $$r_J=\prod_{q=1}^nr_{j_q},\quad \alpha_J=\sum_{q=1}^n \alpha_{j_q},\quad \text{ and } c_{\ell, J}=\prod_{q=1}^n P_\ell(\beta_{j_q}),$$
where $r_j=|\beta_j|$ and $\alpha_j=\arg(\beta_j)/(2\pi)$ for $1\leq j\leq k$.
By convention, we define $r_\emptyset=1$, $\alpha_{\emptyset}=0$ and $c_{\ell, \emptyset}=1$, where $\emptyset$ denotes the empty word.

 Set
$$
\Gamma=\{r_J\colon  J\in \Sigma_*\}.
$$
Since $r_j>1$ for $1\leq j\leq k$,   the set $\{J\in  \Sigma_*\colon r_J<N\}$ is finite for each positive integer $N$.  Hence $\Gamma$ is discrete.  We enumerate $\Gamma$ in increasing order  as
$$
\Gamma=\{1=\gamma_0<\gamma_1<\gamma_2<\cdots\}.
$$
For any $\gamma\in \Gamma$, define
$$
\Sigma_{*,\gamma}=\{J\in \Sigma_*: r_J=\gamma\}.
$$
Then $\Sigma_*=\bigcup_{\gamma\in \Gamma} \Sigma_{*,\gamma}$.

By \eqref{e-4.1}, expanding the exponential function into its Taylor series, we obtain that for $n\in \Lambda$,
\begin{equation}
\label{e-long}
\begin{split}
0&=\sum_{\ell=1}^m w_\ell \exp\left(2\pi i \sum_{j=1}^{k}\beta_j^{-n} P_\ell(\beta_j)\right)\\
&=\sum_{\ell=1}^m w_\ell \sum_{J\in \Sigma_*}\frac{(2\pi i)^{|J|}}{|J|!} \cdot r_{J}^{-n}\exp(-2\pi i n \alpha_J) c_{\ell, J}\\
&= \sum_{J\in \Sigma_*}\frac{(2\pi i)^{|J|}}{|J|!} \cdot r_{J}^{-n}\exp(-2\pi i n \alpha_J) \left(\sum_{\ell=1}^m w_\ell c_{\ell, J}\right)\\
&=\sum_{p=0}^\infty\sum_{J\in \Sigma_{*,\gamma_p}} \frac{(2\pi i)^{|J|}}{|J|!} \cdot\gamma_p^{-n}\exp(-2\pi i n \alpha_J) \left(\sum_{\ell=1}^m w_\ell c_{\ell, J}\right).
\end{split}
\end{equation}
Note that for $J=j_1\ldots j_{|J|}\in \Sigma_*$,
$$
\alpha_J\equiv\sum_{q=1}^{|J|}\sum_{u=1}^s t_{{j_q},u} \theta_u\equiv\sum_{u=1}^s\left( \sum_{q=1}^{|J|} t_{{j_q},u} \right) \theta_u\equiv \sum_{u=1}^s t_{J,u}  \theta_u \pmod 1,
$$
where
$$
t_{J,u}:=\sum_{q=1}^{|J|} t_{{j_q},u},\qquad u=1,\ldots, s.
$$
Substituting this into \eqref{e-long} we obtain
\begin{equation}
\label{e-5.5}
\sum_{p=0}^\infty\sum_{J\in \Sigma_{*,\gamma_p}} \frac{(2\pi i)^{|J|}}{|J|!} \cdot\gamma_p^{-n}\exp\left(-2\pi i n \sum_{u=1}^s t_{J,u}\theta_u\right) \left(\sum_{\ell=1}^m w_\ell c_{\ell, J}\right)=0 \quad \text{for } n \in \Lambda.
\end{equation}

For any integer $p\geq 0$, define $g_p\colon \mathbb{T}^s\to \mathbb{C}$ by
\begin{equation}
\label{e-gp''}
g_p(x_1,\ldots,x_s):= \sum_{J\in \Sigma_{*,\gamma_p}} \frac{(2\pi i)^{|J|}}{|J|!} \cdot \exp\left(-2\pi i  \sum_{u=1}^s t_{J,u}x_u\right) \left(\sum_{\ell=1}^m w_\ell c_{\ell, J}\right).
\end{equation}
Then \eqref{e-5.5} can be rewritten as
\begin{equation}
\label{e-5.5'}
\sum_{p=0}^\infty \gamma_p^{-n} g_p(n\theta_1,\ldots, n\theta_s)=0 \quad \text{for } n \in \Lambda.
\end{equation}
It is worth pointing out that the above series converges absolutely.  Indeed, we have
\begin{equation}
\label{e-long1}
\begin{split}
\sum_{p=0}^\infty  \sup_{(x_1,\ldots, x_s)\in \Bbb T^s}|g_p(x_1,\ldots,x_s)| &\leq \sum_{p=0}^\infty \sum_{J\in \Sigma_{*,\gamma_p}} \frac{(2\pi )^{|J|}}{|J|!} \cdot \left(\sum_{\ell=1}^m |w_\ell|\cdot|c_{\ell, J}|\right)\\
&\leq \sum_{n=0}^\infty \sum_{J\in \{1,\ldots, k\}^n} \frac{(2\pi)^n}{n!} \left(\sum_{\ell=1}^m |w_\ell|\cdot|c_{\ell, J}|\right)\\
&\leq \sum_{n=0}^\infty \frac{(2\pi)^n}{n!} \left(\max_{1\leq \ell \leq m} |w_\ell |\right) \left(\sum_{\ell=1}^m\sum_{j=1}^k |P_\ell(\beta_j)|\right)^n\\
&<\infty.
\end{split}
\end{equation}

We now prove by induction on $p$ that
\begin{equation}
\label{e-gp}
g_p \equiv 0 \quad \text{ for  all } p\geq 0.
\end{equation}

First, by \eqref{e-5.5'},
\[
g_0(n\theta_1,\ldots,n\theta_s)
=
-\sum_{p=1}^\infty
\gamma_p^{-n} g_p(n\theta_1,\ldots,n\theta_s)
\quad \text{for } n\in \Lambda.
\]
Since \( \gamma_p>1 \) for all \( p\ge 1 \), it follows from \eqref{e-long1} and the dominated convergence theorem
that the right-hand side of the above identity tends to \(0\) as \( n\to\infty \) with
\( n\in\Lambda \).
Therefore, \( g_0 \) vanishes on the set $\widetilde{E}$ of all accumulation points of
\[
\{(n\theta_1,\ldots,n\theta_s)\in \Bbb T^s: n\in\Lambda\}.
\]
Since $\Lambda$ has positive upper density in $\N$, by Lemma~\ref{lem-td}, $\mathcal L^s(\widetilde{E})>0$.
Since \( g_0 \) is analytic on \( \mathbb{T}^s \),
 Lemma~\ref{lem-real analytic} implies that \( g_0\equiv 0 \). By \eqref{e-gp''},  this means that
\begin{equation}
\label{e-w-ell}
\sum_{\ell=1}^m w_\ell=0.
\end{equation}

Assume now that \( g_0=\cdots=g_p\equiv 0 \) for some \( p\ge 0 \).
Then, by \eqref{e-5.5'},
\[
\gamma_{p+1}^{-n} g_{p+1}(n\theta_1,\ldots,n\theta_s)
=
-\sum_{j=p+2}^\infty
\gamma_j^{-n} g_j(n\theta_1,\ldots,n\theta_s),
\quad n\in\Lambda.
\]
Thus,
\[
g_{p+1}(n\theta_1,\ldots,n\theta_s)
=
-\sum_{j=p+2}^\infty
\left(\frac{\gamma_{p+1}}{\gamma_j}\right)^n
g_j(n\theta_1,\ldots,n\theta_s),
\quad n\in\Lambda.
\]
Since \( \gamma_{p+1}/\gamma_j \leq \gamma_{p+1}/\gamma_{p+2}<1 \) for all \( j\ge p+2 \),  letting \( n \in \Lambda \) tend to infinity and applying Lemma~\ref{lem-td}, we conclude that \( g_{p+1} \) vanishes on a set of positive Lebesgue measure. Since \( g_{p+1} \) is analytic on \( \mathbb{T}^s \),
 Lemma~\ref{lem-real analytic} implies that \( g_{p+1}\equiv 0 \).
By induction, this completes the proof of \eqref{e-gp}.

Since $g_p\equiv 0$ for all $p\geq 0$ and  the following set of exponential functions
$$
\left\{ \exp\left(2\pi i\sum_{u=1}^s n_ux_u\right):\, (n_1,\ldots n_s)\in \Z^s\right\}
$$
on $\Bbb T^s$ forms an orthogonal basis of $L^2(\Bbb T^s)$,
it follows from  \eqref{e-gp''} that,   for every $p\geq 0$ and every $J\in \Sigma_{*,\gamma_p}$,
\[
\sum_{\substack{J' \in \Sigma_{*,\gamma_p} \\
t_{J',u} = t_{J,u} \text{ for all } 1 \le u \le s}}
\frac{(2\pi i)^{|J'|}}{|J'|!}
\left( \sum_{\ell=1}^m w_\ell c_{\ell, J'} \right)
= 0.
\]
Equivalently, for any $J \in \Sigma_*$,
\begin{equation}
\label{e-J}
\sum_{\substack{J' \in \Sigma_* \\ \beta_{J'} = \beta_J}}
\frac{(2\pi i)^{|J'|}}{|J'|!}
\left( \sum_{\ell=1}^m w_\ell c_{\ell, J'} \right) = 0.
\end{equation}

Now, for each \( n \in \mathbb{N} \), applying \eqref{e-J} to
\[
J_n := i^n = \underbrace{i \ldots i}_{n \text{ times}},
\]
where \( i \in \{1,\ldots,k\} \) is chosen as in Lemma~\ref{lem-i}, and using Lemma~\ref{lem-i}, we obtain
\[
\sum_{\ell=1}^m w_\ell c_{\ell,J_n}=0.
\]
That is,
\begin{equation}
\label{e-eeu}
\sum_{\ell=1}^m w_\ell \bigl( P_\ell(\beta_i) \bigr)^n = 0
\quad \text{for all } n \in \mathbb{N}.
\end{equation}

Since $P_1(\beta_i), \ldots, P_m(\beta_i)$ are all distinct by assumption,
the Vandermonde matrix
\[
\begin{pmatrix}
1 & \cdots & 1 \\
P_1(\beta_i) & \cdots & P_m(\beta_i) \\
\vdots & \ddots & \vdots \\
P_1(\beta_i)^{m-1} & \cdots & P_m(\beta_i)^{m-1}
\end{pmatrix}
\]
is invertible. By \eqref{e-w-ell} and \eqref{e-eeu}, this implies that
$w_\ell = 0$ for all $1 \le \ell \le m$, which is a contradiction.
This completes the proof of \eqref{e-vneativen}.

Finally, we prove that
\begin{equation}
\label{e-positiven}
\operatorname{span}_{\mathbb{C}}\bigl(\{v_n : n\in \Lambda\}\bigr)=\mathbb{C}^m.
\end{equation}
To see this, note that by Lemma~\ref{lem-Galois}(i),  for every $n\in \N$,
\[
\sum_{j=1}^{k'} \beta_j^{n} P_\ell(\beta_j)\in \mathbb{Z}
\quad \text{for all } 1\le \ell \le m .
\]
It follows that
\begin{align*}
v_n
&=\Bigl(
\exp\!\bigl(2\pi i\sum_{j=1}^k \beta_j^{n} P_1(\beta_j)\bigr),
\ldots,
\exp\!\bigl(2\pi i\sum_{j=1}^k \beta_j^{n} P_m(\beta_j)\bigr)
\Bigr)\\
&=\Bigl(
\exp\!\bigl(-2\pi i\sum_{j=k+1}^{k'} \beta_j^{n} P_1(\beta_j)\bigr),
\ldots,
\exp\!\bigl(-2\pi i\sum_{j=k+1}^{k'} \beta_j^{n} P_m(\beta_j)\bigr)
\Bigr).
\end{align*}
Owing to the second expression for $v_n$, together with the fact that
$|\beta_j| < 1$ for all $k+1 \le j \le k'$, the proof of
\eqref{e-positiven} is essentially the same as that of
\eqref{e-vneativen}, with only slight modifications, and we therefore
omit the repetition.
\end{proof}

\subsection{Proof of Theorem \ref{thm-1.5}}
\label{S-5.3}
We first prove a lemma.

\begin{lem}\label{lem-4.4}Let
$
\Phi=\{\phi_\ell\}_{\ell=1}^m
$
be a generalized Pisot IFS of type~I on $\R^d$, with parameters \(k\), \(\beta_1,\ldots,\beta_k\),
and \(P_1,\ldots,P_m\) as in Definition~\ref{defn-1.2}.  Assume that
\(\beta_1,\ldots,\beta_k\) are Galois conjugates over \(\mathbb{Q}\), and let
 \(\beta_{k+1},\ldots,\beta_{k'}\) denotes the remaining Galois conjugates. Moreover, assume in addition that there exist $\theta_1,\ldots,\theta_s \in [0,1)$ such that
$1,\theta_1,\ldots,\theta_s$ are rationally independent, and that
\begin{equation*}
\label{e-e5.2*}
\frac{\arg(\beta_j)}{2\pi}\in {\rm span}_\Z\{1,\theta_1,\ldots,\theta_s\},
\qquad j=1,\ldots,k'.
\end{equation*}
Let $\mu$ be the self-affine measure associated with $\Phi$ and a strictly positive probability vector $\pb=(p_\ell)_{\ell=1}^m$.
Suppose that $\mu$ is Rajchman. Then there exist $\tau \in \mathbb{Z}$ and a subset
$\Lambda \subset \mathbb{N}$ with positive upper density such that
\[
\sum_{\ell=1}^m p_\ell
\exp\!\left(2\pi i
\sum_{j=1}^k \beta_j^\tau (1+\beta_j^n)\, P_\ell(\beta_j)
\right)=0
\quad \text{for all } n \in \Lambda.
\]
\end{lem}

\begin{proof}
   For each \( n \in \mathbb{N} \), applying Lemma~\ref{lem-2.2} to the polynomial
\( Q_n(x)=1+x^n \), there exists \( \tau_n \in \mathbb{Z} \) such that
\begin{equation}
\label{e-sumellA}
\sum_{\ell=1}^m p_\ell
\exp\!\left(
2\pi i \sum_{j=1}^k \beta_j^{\tau_n}(1+\beta_j^n)\, P_\ell(\beta_j)
\right)=0.
\end{equation}

By Lemma~\ref{lem-2.1}, there exists \( N \in \mathbb{N} \) such that for all
\( n \in \mathbb{Z} \) with \( |n|\geq N \),
\[
\left\|\sum_{j=1}^k \beta_j^n P_\ell(\beta_j)\right\|_{\mathbb{Z}}
< \frac{1}{8\pi},
\qquad 1\leq \ell \leq m.
\]
Therefore, if there exists \( n \in \mathbb{N} \) such that
\( |\tau_n| \geq N \) and \( |n+\tau_n| \geq N \), then
$$
\left\|\sum_{j=1}^k \beta_j^{\tau_n}(1+\beta_j^n) P_\ell(\beta_j)\right\|_{\mathbb{Z}}\leq \left\|\sum_{j=1}^k \beta_j^{\tau_n} P_\ell(\beta_j)\right\|_{\mathbb{Z}}+\left\|\sum_{j=1}^k \beta_j^{\tau_n+n} P_\ell(\beta_j)\right\|_{\mathbb{Z}}
< \frac{1}{4\pi},
$$
and hence
\begin{align*}
&\left|
\sum_{\ell=1}^m p_\ell
\exp\left(
2\pi i \sum_{j=1}^k \beta_j^{\tau_n}(1+\beta_j^n) P_\ell(\beta_j)
\right)-1
\right|\\
&\quad\leq
\sum_{\ell=1}^m p_\ell
\left|
\exp\left(
2\pi i \sum_{j=1}^k \beta_j^{\tau_n}(1+\beta_j^n) P_\ell(\beta_j)
\right)-1
\right|\\
&\quad\leq
\sum_{\ell=1}^m p_\ell \cdot 2\pi \cdot \frac{1}{4\pi}
=\frac{1}{2},
\end{align*}
where in the second inequality we used the fact that $|\exp(2\pi i x)-1|\leq 2\pi \|x\|_\Z$ for $x\in \R$.  This  contradicts \eqref{e-sumellA}.   Therefore, for every \( n\in\mathbb{N} \),
either \( |\tau_n|<N \) or \( |n+\tau_n|<N \).

Define
\[
\Lambda_1=\{n\in\mathbb{N} : |\tau_n|<N\}
\quad \text{and} \quad
\Lambda_2=\{n\in\mathbb{N} : |n+\tau_n|<N\}.
\]
Then $\Lambda_1\cup \Lambda_2=\N$, so at least one of \( \Lambda_1 \) and \( \Lambda_2 \) has positive upper
density.

We claim that \( \Lambda_2 \) has zero upper density.  To see this, suppose on the contrary that  \( \Lambda_2 \) has positive upper density.  Then  by the pigeonhole
principle there exist an integer \( u \in(-N,N) \) and a subset
\( \Lambda_3\subset\Lambda_2 \) of positive upper density such that
\( n+\tau_n=u \) for all \( n\in\Lambda_3 \). Consequently,
\[
\{\tau_n : n\in\Lambda_3\}=-\Lambda_3+u
\]
has positive upper density in $-\N$. Observe that
\begin{equation}
\label{e-sumell}
\sum_{\ell=1}^m p_\ell
\exp\left(2\pi i
\sum_{j=1}^k \beta_j^{\tau_n}(1+\beta_j^n)
 P_\ell(\beta_j)\right)
=0\qquad \text{ for every }n\in \Lambda_3.
\end{equation}

By  relabelling the indices $\{1,\ldots, m\}$ if necessary, we may assume that
$P_1(\beta_1)$, $\ldots$, $P_q(\beta_1)$ enumerate the distinct elements of the set
\( \{P_\ell(\beta_1)\}_{\ell=1}^m \). For $\ell=1,\ldots, q$, define
$$
[\ell]=\{1\leq \ell' \leq m\colon P_{\ell'}(\beta_1)=P_\ell(\beta_1)\}.
$$
Since $P_\ell\in \Z[x]$ for $1\leq \ell \leq m$ and $\beta_1,\ldots, \beta_{k'}$ are conjugates over $\Q$, we see that for $1\leq \ell\leq q$,
\begin{equation}
\label{e-Pell1}
P_{\ell'}(\beta_j)=P_{\ell}(\beta_j)\quad \mbox{ for all }\ell'\in [\ell] \text{ and all }1\leq j\leq k',
\end{equation}
and
\begin{equation}
\label{e-Pell'1}
P_{\ell'}(\beta_j)\neq P_{\ell}(\beta_j)\quad \mbox{ for all distinct }\ell', \ell\in \{1,\ldots, q\} \text{ and all }1\leq j\leq k'.
\end{equation}

Combining  \eqref{e-sumell} and \eqref{e-Pell1} yields that
\begin{equation}
\label{e-sumell'}
\sum_{\ell=1}^q \widetilde{p}_\ell
\exp\left(2\pi i
\sum_{j=1}^k \beta_j^{\tau_n}(1+\beta_j^n)
 P_\ell(\beta_j)\right)
=0\qquad \text{ for every }n\in \Lambda_3,
\end{equation}
where
$$
\widetilde{p}_\ell:=\sum_{\ell'\in [\ell]}p_{\ell'},\qquad \ell=1,\ldots, q.
$$
Since $\tau_n+n=u$ for all $n\in \Lambda_3$,  \eqref{e-sumell'} can be rewritten as
\begin{equation}
\label{e-sumell''}
\sum_{\ell=1}^q \widetilde{p}_\ell^*
\exp\left(2\pi i
\sum_{j=1}^k \beta_j^{\tau_n}
 P_\ell(\beta_j)\right)
=0\qquad \text{ for every }n\in \Lambda_3,
\end{equation}
where
$$
\widetilde{p}_\ell^*:=\widetilde{p}_\ell \exp\left(2\pi i\sum_{j=1}^k P_\ell(\beta_j)\beta_j^{u}\right)\neq 0 \quad\text{for }1\leq \ell\leq q.
$$

Since $\{\tau_n : n\in\Lambda_3\}$ has positive upper density in $-\N$, by \eqref{e-Pell'1} and Proposition~\ref{prop-4.1},
\[
\operatorname{span}_\C\left\{
\left(
\exp\left(2\pi i\sum_{j=1}^k  \beta_j^{\tau_n}P_\ell(\beta_j)\right)
\right)_{\ell=1}^q
:\, n\in\Lambda_3
\right\}
=\mathbb{C}^q.
\]
Combining it with \eqref{e-sumell''} yields that \( \widetilde{p}_\ell^*=0 \) for \( 1\leq \ell \leq q \), which is a
contradiction. Hence $\Lambda_2$ has zero upper density, which implies that $\Lambda_1$ has positive upper density.
By the pigeonhole
principle there exist \( \tau\in(-N,N) \) and a subset
\( \Lambda\subset\Lambda_1 \) of positive upper density such that
\( \tau_n=\tau \) for all \( n\in\Lambda \). This completes the proof.\end{proof}

Now we are ready to prove Theorem \ref{thm-1.5}.
\begin{proof}[Proof of Theorem \ref{thm-1.5}]
Let
$
\Phi=\{\phi_\ell(x)=Ax+a_\ell\}_{\ell=1}^m
$
be a generalized Pisot IFS of type~I on $\R^d$, with parameters \(k\), \(\beta_1,\ldots,\beta_k\), $\zeta_1,\ldots, \zeta_k$,
and \(P_1,\ldots,P_m\) as in Definition~\ref{defn-1.2}. Let $\mu$ be the self-affine measure associated with $\Phi$ and a strictly positive probability vector $\pb=(p_\ell)_{\ell=1}^m$.  We show below that $\mu$ is non-Rajchman.

By replacing the index set \(\{1,\ldots,k\}\) with a suitable subset if necessary, we may assume that
\(\beta_1,\ldots,\beta_k\) are Galois conjugates over \(\mathbb{Q}\).
We denote by \(\beta_{k+1},\ldots,\beta_{k'}\) the remaining Galois conjugates.
 Since \(\{\beta_1,\ldots,\beta_k\}\) is a Pisot tuple of type~I, it follows that \(k' > k\). For \(1 \leq j \leq k'\), write
\[
\beta_j = r_j e^{2\pi i \alpha_j},
\]
where \(r_j > 0\) and \(\alpha_j \in [0,1)\).

By Proposition~\ref{pro-2.4}, replacing $\Phi$ and $\pb$ by $\Phi^n$ and $\pb^n:=(p_{\ell_1\ldots \ell_n})_{\ell_1\ldots\ell_n\in \{1,\ldots, m\}^n}$ for a suitable $n\in \N$ if necessary, we may assume that there exist $\theta_1,\ldots,\theta_s \in [0,1)$ with $s\geq 1$ such that
$1,\theta_1,\ldots,\theta_s$ are rationally independent, and integers
$t_{j,u}$ ($1 \leq j \leq k'$,  $0 \leq u \leq s$) satisfying
$$
\alpha_j =t_{j,0}+ \sum_{u=1}^s t_{j,u}\theta_u,
\qquad j=1,\ldots,k'.
$$

By  relabelling the indices $\{1,\ldots, m\}$ if necessary, we may assume that
$P_1(\beta_1)$, $\ldots$, $P_q(\beta_1)$ enumerate the distinct elements of the set
\( \{P_\ell(\beta_1)\}_{\ell=1}^m \). For $\ell=1,\ldots, q$, define
$$
[\ell]=\{1\leq \ell' \leq m\colon P_{\ell'}(\beta_1)=P_\ell(\beta_1)\}.
$$
Since $P_\ell\in \Z[x]$ for $1\leq \ell \leq m$ and $\beta_1,\ldots, \beta_{k'}$ are conjugates over $\Q$, we see that for $1\leq \ell\leq q$,
\begin{equation}
\label{e-Pell}
P_{\ell'}(\beta_j)=P_{\ell}(\beta_j)\quad \mbox{ for all }\ell'\in [\ell] \text{ and all }1\leq j\leq k',
\end{equation}
and
\begin{equation}
\label{e-Pell'}
P_{\ell'}(\beta_j)\neq P_{\ell}(\beta_j)\quad \mbox{ for all distinct }\ell', \ell\in \{1,\ldots, q\} \text{ and all }1\leq j\leq k'.
\end{equation}

Now suppose on the contrary that $\mu$ is Rajchman. By Lemma~\ref{lem-4.4}, there exist $\tau \in \mathbb{Z}$ and a subset $\Lambda \subset \mathbb{N}$ with positive upper density in $\mathbb{N}$ such that
\[
\sum_{\ell=1}^m p_\ell
\exp\!\left(2\pi i
\sum_{j=1}^k \beta_j^{\tau}(1+\beta_j^n)
P_\ell(\beta_j)\right)=0
\quad \text{for all } n\in \Lambda.
\]
By \eqref{e-Pell}, this implies that
\[
\sum_{\ell=1}^q \widetilde{p}_\ell
\exp\!\left(2\pi i
\sum_{j=1}^k \beta_j^{\tau}(1+\beta_j^n)
P_\ell(\beta_j)\right)=0
\quad \text{for all } n\in \Lambda,
\]
where
$$
\widetilde{p}_\ell:=\sum_{\ell'\in [\ell]}p_{\ell'},\qquad \ell=1,\ldots, q.
$$
Equivalently,
\begin{equation}
\label{e-pell*}
\sum_{\ell=1}^q \widetilde{p}_\ell^*
\exp\!\left(2\pi i
\sum_{j=1}^k \beta_j^n
P_\ell(\beta_j)\right)=0
\quad \text{for all } n\in \Lambda+\tau,
\end{equation}
where
\[
\widetilde{p}_\ell^*
= \widetilde{p}_\ell \exp\!\left(2\pi i\sum_{j=1}^k \beta_j^{\tau} P_\ell(\beta_j)\right),\qquad \ell=1,\ldots, q.
\]
Clearly, $\widetilde{p}_\ell^*\neq 0$ for all $1\leq \ell\leq q$.
Since $\Lambda+\tau$ has positive upper density in $\mathbb{N}$, by \eqref{e-Pell'} and Proposition~\ref{prop-4.1},
\[
\operatorname{span}_{\mathbb{C}}\!\left\{
\left(
\exp\!\left(2\pi i\sum_{j=1}^k \beta_j^{n} P_\ell(\beta_j)\right)
\right)_{\ell=1}^q
:\, n \in \Lambda+\tau
\right\}
= \mathbb{C}^q.
\]
Together with \eqref{e-pell*}, this implies that  $\widetilde{p}_\ell^*=0$ for $1 \leq \ell \leq q$, which yields a contradiction.  This completes the proof of Theorem \ref{thm-1.5}.
\end{proof}

\section{Proof of Theorem \ref{thm-1.6}}
\label{S-7}

In this section, we prove Theorem \ref{thm-1.6}.

We first begin with the following lemma.

\begin{lem}
\label{lem-6.1'}
Let $\Phi=\{\phi_\ell(x)=Ax+a_\ell\}_{\ell=1}^m$ be an affinely irreducible integral affine IFS
on $\mathbb{R}^d$. Then there exists an invertible affine map
$S\colon \mathbb{R}^d \to \mathbb{R}^d$ such that the conjugated  IFS
\[
\Psi=\{\psi_\ell(x)=S^{-1}\circ \phi_\ell \circ S(x)=Bx+b_\ell\}_{\ell=1}^m
\]
is also an affinely irreducible integral affine IFS on $\mathbb{R}^d$, with $b_1=0$ and
\[
\operatorname{span}_{\mathbb{Z}}(Y)=\mathbb{Z}^d,
\]
where
\begin{equation}
\label{e-Y}
Y:=\left\{ \sum_{i=0}^{n} B^{-i} v_i \colon n\geq 0,\,
v_i \in \{b_1,\ldots,b_m\} \text{ for all }   0 \leq i\leq n\right\}.
\end{equation}
\end{lem}

\begin{proof}
Write $M=A^{-1}$ and $a'_\ell=a_\ell-a_1$ for $1\le \ell\le m$. Then $M$ is an
expanding $d\times d$ integer matrix and $a'_\ell\in \mathbb{Z}^d$ for all
$\ell$. Set $\mathcal A=\{a'_\ell\}_{\ell=1}^m$. Define
\[
X=\left\{\sum_{i=0}^n M^i v_i \colon n\ge 0,\; v_i\in \mathcal A
\text{ for all } 0\le i\le n \right\}.
\]

Clearly, $X\subset \Z^d$, and therefore $\operatorname{span}_{\mathbb{Z}}(X)$ is a subgroup of $\mathbb{Z}^d$.
Hence there exists an integer $d\times d$ matrix $G$ such that
\[
\operatorname{span}_{\mathbb{Z}}(X)=G\mathbb{Z}^d
\]
(see, for example, \cite[Proposition~4.52]{Knapp2006}). We next show that $G$ is
nonsingular.

Let $V$ be the linear subspace of $\mathbb{R}^d$ of smallest dimension containing $X$,
that is, $V=\operatorname{span}_{\mathbb{R}}(X)$. From the definition of $X$, we see
that $a'_\ell\in X$ for all $1\le \ell\le m$, and in particular,
$0=a'_1\in X$. Moreover, $MX\subset X$, which implies that $V=MV$, and hence
$V=AV$. Since $a'_\ell\in X\subset V$ for each $\ell$, we also have
$$V=AV+a'_\ell=AV+a_\ell-a_1,\quad 1\leq \ell\leq m.$$
Setting $$x_0=(I-A)^{-1}a_1,$$ we obtain $Ax_0-x_0=-a_1$, and thus  $$V+x_0=A(V+x_0)+a_\ell.$$
Hence $V+x_0=\phi_\ell(V+x_0)$. As $\Phi$ is affinely irreducible,
it follows that $V=\mathbb{R}^d$. Consequently, $X$ is not contained in any hyperplane of $\mathbb{R}^d$. Therefore,
$G$ is nonsingular; otherwise, $\operatorname{span}_{\mathbb{Z}}(X)=G\mathbb{Z}^d$
would lie in a hyperplane of $\mathbb{R}^d$.

Since $MX\subset X$, it follows that
\[
M\big(\operatorname{span}_{\mathbb{Z}}(X)\big)
=\operatorname{span}_{\mathbb{Z}}(MX)
\subset \operatorname{span}_{\mathbb{Z}}(X).
\]
That is, $MG\mathbb{Z}^d\subset G\mathbb{Z}^d$, and hence
\[
G^{-1}MG\mathbb{Z}^d\subset \mathbb{Z}^d.
\]
This implies that $T:=G^{-1}MG$ is an integer matrix. Meanwhile, since
$a'_\ell\in X\subset G\mathbb{Z}^d$ for all $\ell$, we have
\[
b_\ell:=G^{-1}(a_\ell-a_1)=G^{-1}a'_\ell\in \mathbb{Z}^d,
\qquad 1\le \ell\le m.
\]
Writing $B=T^{-1}=G^{-1}AG$, we see that $\Psi
=\{Bx+b_\ell\}_{\ell=1}^m$ is an integral IFS.

Define $S\colon \mathbb{R}^d\to \mathbb{R}^d$ by $S(x)=Gx+x_0$. Then a straightforward
computation shows that
\[
S^{-1}\circ \phi_\ell\circ S(x)
=G^{-1}\big(A(Gx+x_0)+a_\ell-x_0\big)
=Bx+b_\ell
\]
for each $1\leq \ell\leq m$.

Finally, letting $Y$ be defined as in \eqref{e-Y}, we have
\begin{align*}
\operatorname{span}_{\mathbb{Z}}(Y)
&=\operatorname{span}_{\mathbb{Z}}\!\left(
\left\{\sum_{i=0}^n G^{-1}M^iG v_i \colon n\ge 0,\;
v_i\in \{b_1,\ldots,b_m\} \right\}\right)\\
&=\operatorname{span}_{\mathbb{Z}}\!\left(
\left\{\sum_{i=0}^n G^{-1}M^i u_i \colon n\ge 0,\;
u_i\in \mathcal A \right\}\right)\\
&=\operatorname{span}_{\mathbb{Z}}(G^{-1}X)
=G^{-1}\big(\operatorname{span}_{\mathbb{Z}}(X)\big)
=\mathbb{Z}^d,
\end{align*}
as desired.
\end{proof}

Recall that if $\mu$ is the self-affine measure associated with a homogeneous affine IFS $\Phi=\{Ax+a_\ell\}_{\ell=1}^m$ on $\R^d$ and a probability vector $\pb=(p_\ell)_{\ell=1}^m$, then the Fourier transform of $\mu$ satisfies that
 \begin{equation}
 \label{e-H1}
 \widehat{\mu}(\xi)=H(\xi)\widehat{\mu}((A^\top)\xi),
 \end{equation}
  where $(A^\top)=A^{\T}$, and
\begin{equation}
\label{e-mask}
H(\xi):=\sum_{\ell=1}^m p_\ell e^{2\pi i\langle \xi,\; a_\ell \rangle}.
\end{equation}
Iterating \eqref{e-H1}, we obtain
$$
\widehat{\mu}(\xi)=\prod_{n=0}^\infty H((A^\top)^{n}\xi).
$$

The following result is known to the experts in the areas of self-affine tilings and wavelet theory.
\begin{prop}
\label{pro-self-affine} Let  $\mu$ be the self-affine measure associated with an integral affine IFS $\Phi=\{Ax+a_\ell\}_{\ell=1}^m$ on $\R^d$ and a probability vector $\pb=(p_\ell)_{\ell=1}^m$. Then
 the following statements are equivalent:
\begin{itemize}
\item[(i)] $\mu$ is absolutely continuous with respect to ${\mathcal L}^d$.
\item[(ii)] The Fourier transform of $\mu$ tends to $0$ at infinity.
\item[(iii)] $\widehat{\mu}(\mb)=0$ for any $\mb\in \Z^d\setminus \{{\bf 0}\}$.
\item[(iv)] For any $\mb\in \Z^d\setminus \{{\bf 0}\}$, there exists $n\in \N$ such that $H((A^\top)^{n}\mb)=0$.
\item[(v)] $\overline{\mu}$ is the Haar measure on $\R^d/\Z^d$, where $\overline{\mu}$ stands for  the push forward of $\mu$ under the canonical projection $\pi: \;\R^d\to \R^d/\Z^d$, i.e., $\overline{\mu}=\mu\circ \pi^{-1}$.
\end{itemize}
\end{prop}
 \begin{proof}
 It follows from the proof of \cite[Theorem 2.1]{LagariasWang1996} with minor modifications.
 \end{proof}

Next we prove the equivalences $(i) \Longleftrightarrow (ii)$ in Theorem~\ref{thm-1.6},  which can be restated as follows.

\begin{thm}
\label{thm-6.3}
Under the condition of Proposition~\ref{pro-self-affine}, the following statements are equivalent:
\begin{itemize}
\item[(i)] $\mu$ is absolutely continuous with respect to ${\mathcal L}^d$.
\item[(ii)]  $\mu$ has power Fourier decay;
\end{itemize}
\end{thm}

\begin{proof}  We may additionally assume that $\Phi$ is affinely irreducible, since otherwise neither (i) nor (ii) holds.    By Proposition~\ref{pro-self-affine}, we only need to prove (i)$\Longrightarrow$(ii).

To this end, assume that $\mu$ is absolutely continuous with respect to ${\mathcal L}^d$. In what follows we prove that $\mu$ has power Fourier decay  by extending an idea used in the proof of \cite[Theorem~1.6]{DaiFengWang2007}.

Set $M= A^{-1}$ and $\mathcal A=\{a_1,\ldots, a_m\}$. Notice that the absolute continuity or the power Fourier decay property of a measure is preserved under pushforward by an invertible affine map. Therefore, by Lemma \ref{lem-6.1'}, after conjugating by an invertible affine map, we may assume that $a_1=0$ and ${\rm span}_\Z(X)=\Z^d$, where
$$
X=\left\{\sum_{i=0}^n M^iv_i\colon n\geq 0,\; v_i\in \mathcal A \text{ for all }0\leq i\leq n\right\}.
$$

Let $\{e_1,\ldots, e_d\}$ be the standard basis of $\R^d$; that is,  $$e_i=(0,0,\ldots,1, 0,\ldots, 0)^{\T},$$ where the entry $1$ appears in  the $i$-th position. Since ${\rm span}_{\Bbb Z}(X)=\Z^d$ and $a_1=0$, there exist $q\in \N$ and integers $t_{i,j,\ell}$, with $1\leq i\leq d$, $1\leq j\leq q$ and $1\leq \ell\leq m$, such that
\begin{equation}
\label{e-eei}
e_i=\sum_{j=1}^q \sum_{\ell=1}^m t_{i,j,\ell}M^{j-1} a_\ell,\quad i=1,\ldots,d.
\end{equation}

Set
$$
L:=\max_{1\leq i\leq d}\sum_{j=1}^q\sum_{\ell=1}^m |t_{i,j,\ell}|.
$$
Below we proceed our arguments by stating and proving a series of claims.
\medskip

{\bf Claim 1}: For $u\in \R^d$ and $\varepsilon>0$, if
$$
\|\langle u,\; M^{j-1}a_\ell \rangle\|_\Z<\varepsilon \quad \text{for all }1\leq j\leq q \text{ and }1\leq \ell\leq m,
$$
then $\|\langle  u,\; e_i \rangle\|_\Z< L\varepsilon$ for each $1\leq i\leq d$.
\begin{proof}
By \eqref{e-eei} and the triangle inequality,
\begin{align*}
\|\langle  u,\; e_i \rangle\|_\Z&=\left\| \sum_{j=1}^q \sum_{\ell=1}^m t_{i,j,\ell} \langle u,\; M^{j-1} a_\ell \rangle\right\|_\Z\\
&\leq \sum_{j=1}^q \sum_{\ell=1}^m |t_{i,j,\ell}|\;\|\langle u,\;  M^{j-1} a_\ell \rangle\|_\Z\\
&<L\varepsilon. \qedhere
\end{align*}
\end{proof}
\medskip

{\bf Claim 2}: For $\xi\in \R^d$ and $\varepsilon>0$, if
$$
\|\langle  (A^\top)^{j}\xi,\; a_\ell \rangle\|_\Z<\varepsilon \quad \text{for all }0\leq j\leq q-1 \text{ and }1\leq \ell\leq m,
$$
then
$$
\|(A^\top)^{q-1}\xi \|_{\Z^d}< \sqrt{d}L\varepsilon.
$$
\begin{proof}
Set $u=(A^\top)^{q-1}\xi$. Then the assumption in Claim 2 is equivalent to
$$
\|\langle u,\; M^{j-1} a_\ell \rangle\|_\Z<\varepsilon \quad \text{for all }1\leq j\leq q \text{ and }1\leq \ell\leq m.
$$
By Claim 1,  $\|\langle  u,\; e_i \rangle\|_\Z<L\varepsilon$ for $1\leq i\leq d$, which implies that
$\|u\|_{\Z^d}<\sqrt{d}L\varepsilon$.
\end{proof}
\medskip

{\bf Claim 3}: Let $H\colon \R^d\to \C$ be defined as in \eqref{e-mask}. Then   $$\|\nabla H(x)\|\leq 2\pi \sum_{\ell=1}^m \|a_\ell\|
$$
for all $x\in \R^d$.
\begin{proof}
A direct check shows that  for $x=(x_1,\ldots, x_d)^\T$,
$$
\partial_{x_k} H(x)=\sum_{\ell=1}^m p_\ell e^{2\pi i\langle x,\; a_\ell \rangle} 2\pi i a_{\ell,k}, \quad k=1,\ldots, d,
$$
where $a_{\ell,k}$ denotes the $k$-th entry of $a_\ell$. The desired inequality now follows immediately.
\end{proof}

{\bf Claim 4}: Set $\varepsilon=\displaystyle\frac{1}{2\pi \sqrt{d}L(\sum_{\ell=1}^m \|a_\ell\|)}$. Let $\xi\in \R^d$ with $\|\xi\|\geq \|M^{q-1}\|$.  Assume that
$$
\|\langle   (A^\top)^{j}\xi,\; a_\ell \rangle\|_\Z<\varepsilon \quad \text{for all }0\leq j\leq q-1 \text{ and }1\leq \ell\leq m.
$$
Then there exists $r\in \N$ such that
$$
|H((A^\top)^{r+q-1}\xi)|\leq \|A\|^r.
$$
\begin{proof}
By Claim 2,
$$\|(A^\top)^{q-1}\xi \|_{\Z^d}< \sqrt{d}L\varepsilon< 1.$$
Choose $\mb\in \Z^d$ such that $$\|(A^\top)^{q-1}\xi-\mb\|=\|(A^\top)^{q-1}\xi \|_{\Z^d}.$$
Since $\|\xi\|\geq \|M^{q-1}\|$, it follows that
$$\|(A^\top)^{q-1}\xi \|\geq \frac{\|\xi\|}{\|(A^\top)^{-(q-1)}\|}=\frac{\|\xi\|}{\|M^{q-1}\|}\geq 1>\|(A^\top)^{q-1}\xi \|_{\Z^d}.$$
Hence $\mb\neq 0$.  Since $\mu$ is absolutely continuous with respect to ${\mathcal L}^d$, by Proposition \ref{pro-self-affine}, there exists $r\in \N$ such that $H((A^\top)^{r}\mb)=0$. By the mean value theorem,
\begin{align*}
|H((A^\top)^{r+q-1}\xi)|&=|H((A^\top)^{r+q-1}\xi)-H((A^\top)^{r}\mb)|\\
&\leq \left(\sup_{x\in \R^d}\|\nabla H(x)\|\right)\| (A^\top)^{r+q-1}\xi-(A^\top)^{r}\mb\|\\
&\leq \left(2\pi \sum_{\ell=1}^m \|a_\ell\|\right)\cdot \| (A^\top)\|^r\cdot\| (A^\top)^{q-1}\xi-\mb\|\\
&\leq \left(2\pi \sum_{\ell=1}^m \|a_\ell\|\right) \|(A^\top)\|^r\sqrt{d}L\varepsilon\\
&=\|(A^\top)\|^r=\|A\|^r.
\end{align*}
This completes the proof of the claim.
 \end{proof}
\medskip

{\bf Claim 5}: Let $\varepsilon$ be defined as in Claim~4. Set
\[
\delta
=\sup\left\{
\left|p_1+\sum_{\ell=2}^m p_\ell e^{2\pi i u_\ell}\right|
\;\colon\;
u_\ell\in\mathbb{R}\ \text{for all }2\le \ell\le m
\ \text{and}\
\|u_\ell\|_{\mathbb{Z}}\ge \varepsilon
\ \text{for at least one } \ell
\right\}.
\]
Then $0<\delta<1$. Furthermore, for each $\xi\in\mathbb{R}^d$ with
$\|\xi\|\ge \|M^{q-1}\|$, either there exists $0\le j\le q-1$ such that
\[
\bigl|H((A^\top)^{\,j}\xi)\bigr|\le \delta,
\]
or there exists $r\in\mathbb{N}$ such that
\[
\bigl|H((A^\top)^{\,r+q-1}\xi)\bigr|\le \|A\|^{r}.
\]

\begin{proof}
Let $\xi\in\mathbb{R}^d$ with $\|\xi\|\ge \|M^{q-1}\|$. Suppose that
$\bigl|H((A^\top)^{\,j}\xi)\bigr|>\delta$ for all $0\le j\le q-1$.
Equivalently,
\[
\left|
p_1+\sum_{\ell=2}^m p_\ell
e^{2\pi i\langle (A^\top)^{\,j}\xi,\, a_\ell\rangle}
\right|
>\delta
\qquad
\text{for all } 0\le j\le q-1.
\]
By the definition of $\delta$, this implies
\[
\bigl\|\langle (A^\top)^{\,j}\xi,\, a_\ell\rangle\bigr\|_{\mathbb{Z}}
<\varepsilon
\quad
\text{for all } 0\le j\le q-1 \text{ and } 1\le \ell\le m.
\]
By Claim~4, there exists $r\in\mathbb{N}$ such that
\[
\bigl|H((A^\top)^{\,r+q-1}\xi)\bigr|
\le \|A\|^{r}. \qedhere
\]
\end{proof}\medskip

{\bf Claim 6}. Set $\gamma=\max\{\delta^{1/q},\; \|A\|^{1/(q+1)}\}$, where $\delta$ is defined as in Claim 5.  Then $0<\gamma<1$. Moreover,
$$
\widehat{\mu}(\xi)=O\left(\|\xi\|^{\log \gamma/\log \|M\|}\right).
$$
\begin{proof}
By Claim 5, for each $\xi\in \R^d$ with $\|\xi\|\geq \|M^{q-1}\|$, either
$$
|\widehat{\mu}(\xi)|=\left(\prod_{j=0}^{q-1} |H((A^\top)^{j}\xi)|\right)\cdot|\widehat{\mu}((A^\top)^{q}\xi)|\leq \delta |\widehat{\mu}((A^\top)^{q}\xi)|\leq \gamma^q |\widehat{\mu}((A^\top)^{q}\xi)|,
$$
or there exists $r\in \N$ such that
\begin{align*}
|\widehat{\mu}(\xi)|&=\left(\prod_{j=0}^{r+q-1} |H((A^\top)^{j}\xi)|\right)\cdot|\widehat{\mu}((A^\top)^{r+q}\xi)|\\
&\leq \|A\|^r \cdot|\widehat{\mu}((A^\top)^{r+q}\xi)|\\
&\leq \left(\|A\|^{\frac{r}{r+q}}\right)^{r+q} \cdot|\widehat{\mu}((A^\top)^{r+q}\xi)|\\
&\leq \left(\|A\|^{\frac{1}{1+q}}\right)^{r+q} \cdot|\widehat{\mu}((A^\top)^{r+q}\xi)|\\
&\leq \gamma^{r+q} \cdot|\widehat{\mu}((A^\top)^{r+q}\xi)|.
\end{align*}
It follows that for every $\xi$ with  $\|\xi\|\geq \|M^{q-1}\|$, there always exists $j\in \N$ such that $$|\widehat{\mu}(\xi)|\leq \gamma^j |\widehat{\mu}((A^\top)^{j}\xi)|.$$
Consequently, for every  $\xi\in \R^d$ with  $\|\xi\|\geq \|M^{q-1}\|$, there exist positive integers $j_1,\ldots, j_k$ such that
$$\|(A^\top)^{j_1+\cdots+j_{k-1}}\xi\|\geq \|M^{q-1}\|>\|(A^\top)^{j_1+\cdots+j_{k}}\xi\|,$$
and
$$|\widehat{\mu}(\xi)|\leq \gamma^{j_1} |\widehat{\mu}((A^\top)^{j_1}\xi)|\leq \cdots\leq \gamma^{j_1+\cdots+j_{k}}|\widehat{\mu}((A^\top)^{j_1+\cdots+j_{k}}\xi)|\leq \gamma^{j_1+\cdots+j_{k}}.
$$

Notice that
$$
\|\xi\|\leq \|(A^\top)^{-(j_1+\cdots +j_k)}\|\cdot \|(A^\top)^{j_1+\cdots +j_k}\xi\|\leq \|M^{j_1+\cdots+ j_k}\|\cdot \|M^{q-1}\|\leq  \|M\|^{j_1+\cdots+ j_k}\cdot \|M\|^{q-1}.
$$
Since $\frac{\log \gamma}{\log \|M\|}<0$, it follows that $$\|\xi\|^{\log \gamma/\log \|M\|}\geq \gamma^{j_1+\cdots+j_k} \cdot \gamma^{q-1},$$
and hence
$$
|\widehat{\mu}(\xi)|\leq \gamma^{j_1+\cdots+j_{k}}\leq C \|\xi\|^{\log \gamma/\log \|M\|},
$$
where $C:=\gamma^{1-q}$.
\end{proof}
By Claim 6, $\mu$ has power Fourier decay at infinity. This completes the proof of  Theorem~\ref{thm-6.3}.
\end{proof}

In the remainder of this section, we prove the equivalence between $(i) \Longleftrightarrow (iii)$ in Theorem~\ref{thm-1.6},  which can be restated as follows.

\begin{thm}
\label{thm-6.4}
Let  $\mu$ be the self-affine measure associated with an integral affine IFS $\Phi=\{Ax+a_\ell\}_{\ell=1}^m$ on $\R^d$ and a probability vector $\pb=(p_\ell)_{\ell=1}^m$. Then the following statements are equivalent:
\begin{itemize}
\item[(i)] $\mu$ is absolutely continuous with respect to ${\mathcal L}^d$.
\item[(ii)] $\tau_A(\mu, 2)=d$.
\end{itemize}
\end{thm}

The proof of Theorem~\ref{thm-6.4} relies on a result of Deng, He, and Lau \cite{DengHeLau2008} concerning vector representations of integral self-affine measures, as well as on the thermodynamic formalism for matrix products developed in \cite{FengKaenmaki2011,FengLau2002}.

 Let $\mu$ be as in Theorem~\ref{thm-6.4}.  In \cite{DengHeLau2008}, the authors constructed a $\Z^d$-tile $T\subset \R^d$, which is the attractor of an integral affine iterated function system $\Psi=\{\psi_i(x)=A^{n_0}(x+c_i)\}_{i=1}^k$ satisfying the open set condition (OSC), where $n_0\in \N$, $k=|\det(A)|^{-n_0}$ and $c_i\in \Z^d$, such that
\begin{equation}
\label{e-mu-bdd}
\mu(\partial T+{e})=0\quad  \mbox{ for all }{e}\in \Z^d,
\end{equation}
where $\partial T$ stands for the boundary of $T$.

Set $${\mathcal V}=\{{v}_1,\ldots, {v}_N\}=\{{v}\in \Z^d:\; K\cap (\mbox{int}(T)+{v})\neq \emptyset\}$$ and define a vector-valued measure $\pmb{\mu}$ on $T$  by
$$
\pmb{\mu}(E)=[\mu((E\cap T)+{v}_1), \ldots, \mu((E\cap T)+{v}_N)]^\T.
$$

For $J=j_1\cdots j_{n_0}\in \{1,\ldots, m\}^{n_0}$, set $p_J=p_{j_1}\cdots p_{j_{n_0}}$ and $d_J=\sum_{k=1}^{n_0} A^{-n_0+k-1}a_{j_k}$. Define a tuple ${\bf M}=(M_1,\ldots, M_k)$ of $N\times N$ non-negative matrices by
$$
(M_r)_{i,j}=\left\{\begin{array}{cc} p_J & \mbox{ if }\; c_r+A^{-n_0} v_i-v_j=d_J \;\mbox{ for some }J\in \{1,\ldots, m\}^{n_0},\\
0 & \mbox{ otherwise},
\end{array}
\right.
$$
where $1\leq r\leq k$, and $1\leq i,j\leq N$.   The following theorem is our starting point.

\begin{thm}[{\cite[Theorems 1.1-1.2]{DengHeLau2008}}]
\label{thm-DHL}
 \begin{itemize}
\item[(i)]The tuple ${\bf M}$ is positively irreducible, that is, there exists an integer $j\geq 1$ such that all the entries of the matrix $\sum_{n=0}^j (\sum_{i=1}^k M_i)^j$ is strictly positive.
\item[(ii)] $\sum_{i=1}^k M_i$ is Markov, i.e., all its column sums are equal to $1$.
\item[(iii)] For any $I=i_1\ldots i_n\in \{1,\ldots, k\}^n$,
$$
\pmb{\mu}(\psi_{I}(T))=M_I\pmb{\mu}(T),
$$
where $\psi_I:=\psi_{i_1}\circ \cdots \circ \psi_{i_n}$ and $M_I:=M_{i_1}\cdots M_{i_n}$.
\end{itemize}
\end{thm}

We still need a result on the thermodynamic formalism of matrix products. For $q>0$, define
\begin{equation}
P({\bf M},q)=\lim_{n\to \infty} \frac{1}{n}\log\left(\sum_{I\in \{1,\ldots, k\}^n}\|M_I\|^q\right).
\end{equation}
The existence of the limit follows from a subadditivity argument.  The function $P({\bf M},\cdot)$ is called the {\it pressure function of $\bf M$}.

Let $(\Sigma, \sigma)$ denote the one-sided full shift over the alphabet $\{1,\ldots, k\}$, that is, $\Sigma=\{1,\ldots,k\}^\N$ and $\sigma$ is the left shift on $\Sigma$. Let $\mathcal M(\Sigma,\sigma)$ denote the set of all $\sigma$-invariant Borel probability measures on $\Sigma$.  For $\nu\in \mathcal M(\Sigma,\sigma)$, let $h_\nu(\sigma)$ denote the measure-theoretic entropy of $\nu$, and define
$$
\lambda({\bf M}, \nu)=\lim_{n\to \infty}\frac{1}{n} \sum_{I\in \{1,\ldots, k\}^n} \mu([I])\log (\|M_I\|),
$$
where $[I]:=\{x=(x_n)_{n=1}^\infty\in \Sigma\colon x_1\ldots x_n=I\}$ is the cylinder set of length $n$ corresponding to $I$. The existence of the limit again follows from subadditivity. We call $\lambda({\bf M}, \nu)$ the {\it Lyapunov exponent of ${\bf M}$ with respect to $\nu$}. The reader is referred to \cite{Walters1982} for further details on these notions.

For two families of positive real numbers $\{a_i\}_{i\in\mathcal I}$ and $\{b_i\}_{i\in\mathcal I}$, we write \[ a_i \approx b_i \] if there exists a constant $c\ge 1$ such that $c^{-1}b_i \le a_i \le cb_i$ for all $i\in\mathcal I$.

Since ${\bf M}$ is positive irreducible, the following result follows from \cite{Feng2004,FengLau2002}.
\begin{thm}
\label{thm-pressure}
Let $q>0$. Then
$$
P({\bf M},q)=\sup\{h_\nu(\sigma)+q\lambda({\bf M}, \nu)\colon \nu\in \mathcal M(\Sigma,\sigma)\}.
$$
Moreover, the supremum is attained at a unique measure $\nu_q\in \mathcal M(\Sigma,\sigma)$,  called the equilibrium state for $({\bf M},q)$.  Furthemore,
$$
\nu_q([i_1\ldots i_n])\approx \exp(-nP({\bf M},q))\|M_{i_1\ldots i_n}\|^q\quad\text{for all }n\in \N \text{ and }i_1\ldots i_n\in \{1,\ldots, k\}^n.
$$
\end{thm}

A variant of Theorem~\ref{thm-pressure} for general real or complex matrix products under a different irreducibility assumption can be found in \cite[Proposition~1.2]{FengKaenmaki2011}.  Now we are ready to prove Theorem~\ref{thm-6.4}.

\begin{proof}[Proof of Theorem~\ref{thm-6.4}]
Let $\pi\colon \R^d\to \R^d/\Z^d$ be the canonical projection, and let
$\overline{\mu}=\mu\circ\pi^{-1}$.
Since $\mu$ has compact support, it follows readily from
Definition~\ref{defn-1.9} that
\[
\tau_A(\overline{\mu},2)=\tau_A(\mu,2).
\]
We omit the straightforward verification.

We first prove the implication $(i)\Longrightarrow(ii)$.
Assume that $\mu$ is absolutely continuous.
Then, by Proposition~\ref{pro-self-affine},
$\overline{\mu}$ coincides with the Haar measure on
$\R^d/\Z^d$.
Consequently,
$\tau_A(\overline{\mu},2)=d$, and hence
$\tau_A(\mu,2)=d$.

Next we prove the implication $(ii)\Longrightarrow(i)$.
Assume that
$\tau_A(\mu,2)=d$.
Then
$\tau_A(\overline{\mu},2)=d$.

Since $\Psi$ satisfies the OSC and has attractor $T$, it follows directly from the definition that
\begin{align}
\label{e-7.6}
\tau_A(\overline{\mu},2)
&=
\liminf_{n\to\infty}
\frac{d}{nn_0\log|\det(A)|}
\log\Big(
\sum_{I\in\{1,\ldots,k\}^n}
\overline{\mu}(\psi_I(T))^2
\Big)
\nonumber\\
&=
-\limsup_{n\to\infty}
\frac{d}{n\log k}
\log\Big(
\sum_{I\in\{1,\ldots,k\}^n}
\overline{\mu}(\psi_I(T))^2
\Big).
\end{align}

Since $T$ is a self-affine $\Z^d$-tile of $\R^d$, there exists a Borel set
$T'\subset T$ such that
$\operatorname{int}(T')=\operatorname{int}(T)$ and
\[
\R^d=\bigcup_{v\in\Z^d}(T'+v),
\]
where the union is disjoint.
In other words, $T'$ is a fundamental domain of
$\R^d/\Z^d$.

By \eqref{e-mu-bdd} and the self-affinity relation
\[
\mu=\sum_{j=1}^m p_j\,\mu\circ\phi_j^{-1},
\]
one obtains
$\mu(\psi_I(\partial T)+v)=0$ for every
$v\in\Z^d$ and every finite word $I$ over
$\{1,\ldots,k\}$.
Consequently,
\[
\mu(\psi_I(T)+v)=\mu(\psi_I(T')+v).
\]
Combining this with Theorem~\ref{thm-DHL}(ii) yields
\begin{equation}
\label{e-MI}
\overline{\mu}(\psi_I(T'))
=
\overline{\mu}(\psi_I(T))
=
(1,\ldots,1)M_I\pmb{\mu}(T)
\approx
\|M_I\|.
\end{equation}

Using \eqref{e-7.6}, \eqref{e-MI}, and the assumption
$\tau_A(\overline{\mu},2)=d$, we obtain
\[
\limsup_{n\to\infty}
\frac1n
\log
\Big(
\sum_{I\in\{1,\ldots,k\}^n}
\|M_I\|^2
\Big)
=
-\log k.
\]
Equivalently,
\begin{equation}
\label{e-M2}
P({\bf M},2)=-\log k.
\end{equation}

By Theorem~\ref{thm-DHL}(ii),
$\sum_{i=1}^k M_i$ is a Markov matrix and therefore has spectral radius equal to $1$.
Hence
\begin{align*}
\sum_{I\in\{1,\ldots,k\}^n}\|M_I\|
&\approx
\sum_{I\in\{1,\ldots,k\}^n}
(1,\ldots,1)M_I(1,\ldots,1)^\top\\
&=
(1,\ldots,1)
\Big(\sum_{i=1}^kM_i\Big)^n
(1,\ldots,1)^\top\\
&\approx
\Big\|
\Big(\sum_{i=1}^kM_i\Big)^n
\Big\|.
\end{align*}
Therefore,
\[
P({\bf M},1)
=
\lim_{n\to\infty}
\frac1n
\log
\Big\|
\Big(\sum_{i=1}^kM_i\Big)^n
\Big\|
=
\log
\rho\Big(\sum_{i=1}^kM_i\Big)
=
0,
\]
where $\rho(\cdot)$ denotes the spectral radius.

By Theorem~\ref{thm-pressure},
$({\bf M},1)$ admits a unique equilibrium state
$\nu_1\in\mathcal M(\Sigma,\sigma)$ satisfying
\begin{equation}
\label{e-hnu1}
h_{\nu_1}(\sigma)
+
\lambda({\bf M},\nu_1)
=
P({\bf M},1)
=
0,
\end{equation}
and
\begin{equation}
\label{e-equlinu1}
\nu_1([i_1\cdots i_n])
\approx
\|M_{i_1\cdots i_n}\|
\quad
\text{for all }
n\in\N
\text{ and }
i_1\cdots i_n\in\{1,\ldots,k\}^n.
\end{equation}

Combining \eqref{e-M2}, Theorem~\ref{thm-pressure}, and
\eqref{e-hnu1}, we obtain
\[
-\log k
=
P({\bf M},2)
\ge
h_{\nu_1}(\sigma)
+
2\lambda({\bf M},\nu_1)
=
-h_{\nu_1}(\sigma).
\]
Hence
$
h_{\nu_1}(\sigma)\ge \log k.
$
Since the topological entropy of the full shift equals $\log k$, it follows that
$h_{\nu_1}(\sigma)=\log k$.
Therefore, $\nu_1$ is the unique measure of maximal entropy on $\Sigma$.
By \eqref{e-equlinu1},
\[
\|M_{i_1\cdots i_n}\|
\approx
k^{-n}
\qquad
\text{for all }
n\in\N
\text{ and }
i_1\cdots i_n\in\{1,\ldots,k\}^n.
\]

Combining this with \eqref{e-MI}, we obtain
\[
\overline{\mu}(\psi_I(T'))
\approx
k^{-|I|}
=
{\mathcal L}^d(\psi_I(T')).
\]
It follows that $\overline{\mu}$ is absolutely continuous with respect to the Haar measure on $\R^d/\Z^d$.
Consequently, $\mu$ is absolutely continuous with respect to ${\mathcal L}^d$.
\end{proof}

\section{An algorithm for determining  the $L^{(A,2)}$ dimension of integral self-affine measures}
\label{S-8'}

In this section, we provide an algorithm for computing the $L^{(A,2)}$ dimension of integral self-affine measures, analogous to the algorithms developed in \cite{Akiyama2020,Lau1993,LNR2001} for the computation of the $L^2$ dimension of self-similar measures satisfying the finite type condition.

Throughout this section, let $\Phi=\{\phi_\ell(x)=Ax+a_\ell\}_{\ell=1}^m$ be an integral affine IFS on $\R^d$. That is, $A$ has spectral radius  less than $1$, $A^{-1}\in GL_d(\Z)$, and $a_\ell\in \Z^d$ for all $1\leq \ell\leq m$. Let $K$ denote the attractor of $\Phi$, and let $\mu$ be the self-affine measure associated with $\Phi$ and a strictly positive probability vector $\pb=(p_\ell)_{\ell=1}^m$.

By replacing \(\Phi\) with a suitable iterate if necessary, we may assume without loss of generality that \(\|A\|<1\).

Fix \(R>0\) sufficiently large so that such that
\begin{equation}
\label{e-B0R}
\phi_\ell(B(0,R))\subset B(0,R)\quad \text{for all }1\leq \ell\leq m,
\end{equation}
where $B(0,R)$ denotes the closed ball in $\mathbb{R}^d$ of radius $R$ centered at the origin.

Set
$$
\mathcal A:=\{a_i-a_j\colon 1\leq i,j\leq m\}
$$
and define
\begin{equation}
\label{e-Lambdan'}
\Lambda=\left\{\sum_{k=1}^n A^{-k}\varepsilon_k\colon n\in \N\text { and } \varepsilon_k\in \mathcal A \text{ for }1\leq k\leq n\right\}.
\end{equation}
Since $A^{-1}\in GL_d(\Z)$ and $\mathcal A\subset \Z^d$, we have $\Lambda\subset \Z^d$.

Define
\[
\Omega=\Omega_R:=B(0,R)\cap \Lambda.
\]
Let $N=\#(\Omega)$, and enumerate the elements of $\Omega$ as
\[
\Omega=\{v_1,v_2,\ldots,v_N\},
\]
 with $v_1=0$.

For $1\leq r,r'\leq N$, define
\begin{equation}
\label{e-Lambdaij}
\mathcal E_{r,r'}=\left\{(i, j)\in \{1,\ldots, m\}^2\colon  v_{r'}=A^{-1}(v_r+a_{i}-a_{j})\right\}.
\end{equation}
Finally, define an $N\times N$ matrix ${\bf M}=({\bf M}_{r,r'})_{1\leq r,r'\leq N}$ by
\begin{equation}
\label{e-gij}
{\bf M}_{r,r'}:=
\left\{
\begin{array}{ll}
\sum_{(i, j)\in \mathcal E_{r,r'}} p_i p_j,& \text{ if } \mathcal E_{r,r'}\neq \emptyset,\\
0, & \text{ otherwise}.
\end{array}
\right.
\end{equation}

Let $\tau_A(\mu,2)$ denote the $L^{(A,2)}$ dimension of $\mu$ (see Definition~\ref{defn-1.9}), and let $\rho(\cdot)$ denote spectral radius.  The following result provides  an algorithm to compute $\tau_A(\mu,2)$.

\begin{prop}
\label{prop-8.1}
Let ${\bf M}$ be defined as above. Then
\begin{equation}
\label{e-algorithm}
\tau_A(\mu,2)=\displaystyle \frac{d \log \rho({\bf M})}{\log |\det(A)|}.
\end{equation}
Consequently, $\mu$ is absolutely continuous if and only if $\rho({\bf M})=|\det(A)|$.
\end{prop}

\begin{proof}
The second part of the proposition follows directly from \eqref{e-algorithm} and Theorem~\ref{thm-6.4}. It therefore remains to prove \eqref{e-algorithm}.

Write
\begin{equation}
\label{e-thetan1}
\theta_n:=\sum_{Q\in \mathcal D}\mu(A^nQ)^2.
\end{equation}
By the definition of $\tau_A(\mu,2)$, proving \eqref{e-algorithm} is equivalent to showing that
\begin{equation}
\label{e-thetan}
\limsup_{n\to \infty} \frac{1}{n}\log \theta_n=\log \rho({\bf M}).
\end{equation}

 Our proof of \eqref{e-thetan} is adapted from the proof of \cite[Proposition~4.1]{Akiyama2020}. For completeness, we provide a detailed proof.

For each $n\in \N$, introduce an equivalence relation $\sim_n$ on $\{1,\ldots,m\}^n$ by
$$
x_1\ldots x_n\sim_n y_1\ldots y_n \quad \text{if}\quad  \phi_{x_1\ldots x_n}=\phi_{y_1\ldots y_n}.
$$
Note that  $\phi_{x_1\ldots x_n}=\phi_{y_1\ldots y_n}$ if and only if $\sum_{k=1}^n A^{n+1-k}(a_{x_k}-b_{y_k})=0$.
For convenience, let $\Gamma_n:=\{1,\ldots,m\}^n/\sim_n$ denote the corresponding quotient set.

Define a tuple $(M_1,\ldots, M_m)$ of $N\times N$ non-negative matrices by
$$
(M_i)_{r,r'}=\left\{\begin{array}{ll} p_j, & \mbox{ if }\;  v_{r'}=A^{-1}(v_r+a_i-a_j) \;\text{ for some }j\in \{1,\ldots, m\},\\
0, & \mbox{ otherwise},
\end{array}
\right.
$$
where $1\leq i\leq m$, and $1\leq r,r'\leq N$. 

From the definitions of ${\bf M}$ and the matrices $M_i$, it follows that
\begin{equation}
\label{e-BM}
{\bf M}=\sum_{i=1}^m p_iM_i.
\end{equation}
Moreover, one readily checks that for every word $x_1\ldots x_n\in \{1,\ldots,m\}^n$,
\begin{equation}
\label{e-productM}
(M_{x_1}\cdots M_{x_n})_{i,j}=\sum_{\substack{\ell_1\ldots \ell_n\in \{1,\ldots,m\}^n\\
v_j=A^{-n} v_i+\sum_{k=1}^nA^{-(n+1-k)}(a_{x_k}-a_{\ell_k}) }} p_{\ell_1}\ldots p_{\ell_n}
\end{equation}
for all $1\leq i,j\leq N$. In  verifying this,  we use the fact that if 
$$
v_j=A^{-n} v_i+\sum_{k=1}^nA^{-(n+1-k)}(a_{x_k}-a_{\ell_k}),
$$
then $$A^{-\ell} v_i+\sum_{k=1}^{\ell-1} A^{-(\ell+1-k)}(a_{x_k}-a_{\ell_k})\in \Omega=\{v_1,\ldots, v_N\}$$ for every integer $1<\ell<n$. 
The fact follows directly from \eqref{e-B0R}. As a special case of \eqref{e-productM}, taking $i=j=1$, we obtain
\begin{equation}
\label{e-8.7}
(M_{x_1}\cdots M_{x_n})_{1,1}=\sum_{\substack{y_1\ldots y_n\in \{1,\ldots,m\}^n\\
y_1\ldots y_n\sim_n x_1\ldots x_n}} p_{y_1}\ldots p_{y_n}.
\end{equation}

From \eqref{e-productM}, we see that
\begin{equation}
\label{e-8.8}
M_{x_1}\cdots M_{x_n}=M_{y_1}\cdots M_{y_n}\quad\text{ if }\quad x_1\ldots y_n\sim_n y_1\ldots y_n.
\end{equation}

Now, by the self-affinity relation of $\mu$, and applying \eqref{e-8.7}-\eqref{e-8.8}, we obtain
\begin{equation*}
\label{e-8.9}
\mu=\sum_{x_1\ldots x_n\in \{1,\ldots,m\}^n}p_{x_1\ldots x_n}\mu\circ \phi_{x_1\ldots x_n}^{-1}=\sum_{[x_1\ldots x_n]\in \Gamma_n}(M_{x_1}\cdots M_{x_n})_{1,1}\;\mu\circ  \phi_{x_1\ldots x_n}^{-1}.
\end{equation*}
It follows that
\begin{align}
\label{e-8.11}
\theta_n=\sum_{Q\in \mathcal D}\mu(A^nQ)^2&=\sum_{Q\in \mathcal D}\Big(\sum_{[x_1\ldots x_n]\in \Gamma_n}(M_{x_1}\cdots M_{x_n})_{1,1}\;\mu\big(\phi_{x_1\ldots x_n}^{-1}(A^nQ)\big)\Big)^2\nonumber\\
&=\sum_{Q\in \mathcal D}\left(\sum_{\substack{[x_1\ldots x_n]\in \Gamma_n\\ \phi_{x_1\ldots x_n}(K)\cap A^nQ\neq \emptyset}}(M_{x_1}\cdots M_{x_n})_{1,1}\;\mu\big(\phi_{x_1\ldots x_n}^{-1}(A^nQ)\big)\right)^2.
\end{align}

Since $\Phi$ is an integral IFS, it is readily verified that there exists $L\in \N$,  independent of $n$, such that for each $Q\in \mathcal D$, there are at most $L$  elements $[x_1\ldots x_n]\in \Gamma_n$ satisfying $A^n Q\cap \phi_{x_1\ldots x_n}(K)\neq \emptyset$;  conversely, each $\phi_{x_1\ldots x_n}(K)$ intersects  $A^nQ$ for at most $L$  cubes $Q\in \mathcal D$. By applying this property, together with \eqref{e-8.11}  and the Cauchy-Schwarz inequality, one obtains
\begin{equation}
\label{e-thetacpmpare}
L^{-2} \sum_{[x_1\ldots x_n]\in \Gamma_n}\big((M_{x_1}\cdots M_{x_n})_{1,1}\big)^2\leq \theta_n\leq L^2 \sum_{[x_1\ldots x_n]\in \Gamma_n}\big((M_{x_1}\cdots M_{x_n})_{1,1}\big)^2.
\end{equation}

For an $N\times N$ non-negative matrix $G$, define its $1$-norm by $$\|G\|_1:=\sum_{1\leq i,j\leq N}G_{i,j}.$$
Set
$$
\beta_n:=\sum_{x_1\ldots x_n\in \{1,\ldots,m\}^n}p_{x_1\ldots x_n} \|M_{x_1}\ldots M_{x_n}\|_1.
$$
Then
\begin{equation}
\label{e-limitbetan}
\beta_n=\left\|\left(\sum_{i=1}^m p_iM_i\right)^n\right\|_1=\left\|{\bf M}^n\right\|_1,
\end{equation}
where the second equality follows from \eqref{e-BM}.  It follows that $$\lim_{n\to \infty} \frac{1}{n}\log \beta_n=\log \rho({\bf M}).$$

To compare $\beta_n$ and $\theta_n$, we define
$$
w_n:=\sum_{[x_1\ldots x_n]\in \Gamma_n} \big(\|M_{x_1}\cdots M_{x_n}\|_1\big)^2 \quad \text{and}\quad \widetilde{w}_n:=\sum_{[x_1\ldots x_n]\in \Gamma_n} \big((M_{x_1}\cdots M_{x_n})_{1,1}\big)^2.
$$
Clearly, $w_n\geq \widetilde{w}_n$.  Moreover, by \eqref{e-BM}, \eqref{e-8.7} and \eqref{e-8.8},
\begin{equation}
\label{e-widetildewn}
\widetilde{w}_n=\sum_{x_1\ldots x_n\in \{1,\ldots, m\}^n} p_{x_1\ldots x_n} (M_{x_1}\cdots M_{x_n})_{1,1}=({\bf M}^n)_{1,1}.
\end{equation}

For two elements $[x_1\ldots x_n], [y_1\ldots y_n]\in \Gamma_n$, we say that $[y_1\ldots y_n]$ is a neighbor of $[x_1\ldots x_n]$ if there exist $i,j\in \{1,\ldots N\}$ such that
$$
v_j=A^{-n}v_i+\sum_{k=1}^n A^{-(n+1-k)}(a_{x_k}-a_{y_k}).
$$
Assume that $n$ is sufficiently large so that $A^{-n}v_{i_1}-v_{j_1}\neq A^{-n}v_{i_2}-v_{j_2}$ whenever $(i_1,j_1)\neq (i_2,j_2)$. Then by \eqref{e-productM},
$$
\|M_{x_1}\cdots M_{x_n}\|_1=\sum(M_{y_1}\cdots M_{y_n})_{1,1},
$$
where the sum is over all neighbors of $[x_1\ldots x_n]$.  As is evident from the definition of a neighbor,  any element of $\Gamma_n$ has at most $N^2$ neighbors, and conversely, each element of $\Gamma_n$ is the neighbor of at most $N^2$ elements in $\Gamma_n$.  Applying the Cauchy-Schwartz inequality,  we obtain
 \begin{align*}
    w_n&=\sum_{[x_1\ldots x_n]\in \Gamma_n} (\|M_{x_1}\cdots
    M_{x_n}\|_1)^2\\ &=\sum_{[x_1\ldots x_n]\in \Gamma_n}
    \left(\sum_{\substack{[y_1\ldots y_n]\in \Gamma_n\\ [y_1\ldots
          y_n] \text{ is a neighbor of }[x_1\ldots x_n]}}
    (M_{x_1}\cdots M_{x_n})_{1,1}\right)^2\\
    &\leq N^2
    \sum_{[x_1\ldots x_n]\in \Gamma_n}
    \left(\sum_{\substack{[y_1\ldots y_n]\in \Gamma_n\\ [y_1\ldots
          y_n] \text{ is a neighbor of }[x_1\cdots x_n]}}
    \left((M_{y_1}\cdots M_{y_n})_{1,1}\right)^2\right)\\ &\leq
 N^4 \sum_{[y_1\ldots y_n]\in \Omega_n }
   ((M_{y_1}\cdots M_{y_n})_{1,1})^2=N^4 \widetilde{w}_n.
  \end{align*}
  Hence we obtain $\widetilde{w}_n\leq w_n\leq N^4 \widetilde{w}_n$.

  From the definition of $\beta_n$ and \eqref{e-8.7}-\eqref{e-8.8}, we see that
 \begin{align*}
 \widetilde{w}_n= \sum_{[x_1\ldots x_n]\in \Gamma_n}\big((M_{x_1}\cdots M_{x_n})_{1,1}\big)^2&\leq \sum_{[x_1\ldots x_n]\in \Gamma_n}(M_{x_1}\cdots M_{x_n})_{1,1} \|M_{x_1}\cdots M_{x_n}\|_1\\
 &=\beta_n \leq w_n.
  \end{align*}
  Combining it with the inequalities $\widetilde{w}_n\leq w_n\leq N^4 \widetilde{w}_n$ and \eqref{e-thetacpmpare}, we obtain
  $$
 L^{-2}N^{-4}\beta_n\leq L^{-2}N^{-4} w_n\leq  L^{-2}\widetilde{w}_n\leq  \theta_n\leq L^2\widetilde{w}_n \leq L^2 \beta_n.
  $$
  This, together with \eqref{e-limitbetan}, yields \eqref{e-thetan}.
  \end{proof}

\begin{rem}
\label{rem-8.2}
 \begin{itemize}
\item[(i)]Notice that every entry of the constructed matrix ${\bf M}$ is either zero or an quadratic polynomial in $\Z[\pb]:=\Z[p_1,\ldots, p_\ell]$ with non-negative coeeficients. By   Proposition~\ref{prop-8.1}, $\mu=\mu_{\pb}$ is absolutely continuous if and only if $\rho({\bf M})=|\det(A)|$. Therefore, for a given integral IFS $\Phi$, Proposition~\ref{prop-8.1}  provides an algorithm for determining those probability vectors $\pb$ for which the corresponding self-affine measure $\mu_{\pb}$ is absolutely continuous.
\item[(ii)]  Although the matrix ${\bf M}$ depends on the choice of $R$,  its spectral radius is independent of this choice. Moreover, by \eqref{e-widetildewn} and the argument in the proof of Proposition~\ref{prop-8.1}, we have
$$\rho({\bf M})=\rho(\widetilde{{\bf M}}),$$
where $\widetilde{\bf M}$ is the irreducible component  of
${\bf M}$ containing vertex $1$. We remark that, alternatively, one may adapt the approaches of  \cite{Lau1993}  and \cite{LNR2001} to prove that
$$\tau_A(\mu,2)=\frac{d\log \rho(\widetilde{{\bf M}})}{\log |\det(A)|}.$$
\item[(iii)] Proposition~\ref{prop-8.1} extends to all non-integral homogeneous affine IFSs satisfying the finite type condition.   Here we say that an IFS
$\Phi=\{\phi_\ell(x)=Ax+a_\ell\}_{\ell=1}^m$ satisfies the {\it finite type condition}, if there exists $R>0$ such that \eqref{e-B0R} holds and $\#(\Lambda\cap B(0,R))<\infty$,  where $\Lambda$ is defined as in \eqref{e-Lambdan'}.
\end{itemize}

\end{rem}

\section{Examples}
\label{S-8}

We first give an unexpected example of inhomogeneous Rajchman self-similar measures on $\R$.

\begin{exmp}
\label{exmp-1}
Let $\Phi=\{\phi_1(x)=\beta^{-1}x,\, \phi_2(x)=1-\beta^{-1}x\}$ be an IFS on \(\mathbb{R}\), where \(\beta=\frac{\sqrt{5}+1}{2}\). For \(p\in [0,1]\), let \(\mu_p\) be the self-similar measure associated with \(\Phi\) and the probability vector \((p,1-p)\). Then the following statements hold:
\begin{itemize}
\item[(i)] \(\mu_p\) is absolutely continuous if and only if \(p=\frac{\sqrt{5}-1}{2}\) or \(p=\frac{3-\sqrt{5}}{2}\);
\item[(ii)] In these two cases, letting \(f_p=\frac{d\mu_p}{dx}\) denote the density of \(\mu_p\), we have
\[
f_p(x)=\left\{
\begin{array}{ll}
\dfrac{\beta^2}{2\beta-1}\chi_{[0,\beta^{-1}]}(x)+\dfrac{\beta}{2\beta-1}\chi_{[\beta^{-1},1]}(x),
& \text{if } p=\frac{\sqrt{5}-1}{2},\\[8pt]
2\beta x\,\chi_{[0,\beta^{-1}]}(x)+2\beta^2(1-x)\,\chi_{[\beta^{-1},1]}(x),
& \text{if } p=\frac{3-\sqrt{5}}{2},
\end{array}
\right.
\]
where, for a subset $A$ of $\R$, $\chi_A$ denotes its characteristic function.
See Figure~\ref{Fig1} for the graphs of the densities $f_p$.
\end{itemize}

\end{exmp}

\begin{figure}[H]
 \begin{center}
\includegraphics[width=3.5in]{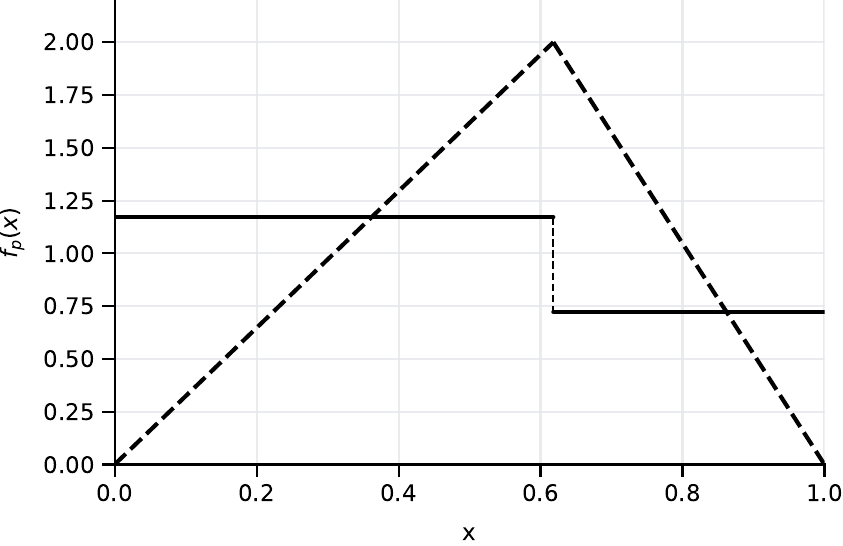}
\par\vspace{0pt}
\end{center}
	\label{tab:graph}
\caption{The densities $f_p$ for $p=\frac{\sqrt{5}-1}{2}$ (solid) and $p=\frac{3-\sqrt{5}}{2}$ (dashed).}
\label{Fig1}
\end{figure}

\begin{proof}[Justification]  Here we give a sketch of the proof of the above statements. For each
\( p \in \{(\sqrt{5}-1)/2,\,(3-\sqrt{5})/2\} \), let \( f_p \) be the density function defined in (ii).
One checks directly that \( \int_{\mathbb{R}} f_p(x)\,dx = 1 \) and that
\[
f_p(x) = p\beta f_p(\beta x) + (1-p)\beta f_p(\beta - \beta x)
\quad \text{for all } x \in [0,1] \setminus \{\beta^{-2}, \beta^{-1}\}.
\]

Let \( \eta_p \) be the probability measure on \([0,1]\) with density \( f_p \). Then by the above scaling property of $f_p$, \( \eta_p \) satisfies the self-similarity relation
\[
\eta_p = p\,\eta_p \circ \phi_1^{-1} + (1-p)\,\eta_p \circ \phi_2^{-1}.
\]
Hence \( \eta_p = \mu_p \), which proves (ii).

To prove (i), we apply a specialized version of a result of Lau, Ngai, and Rao \cite[Theorem~4.1 and Corollary~4.1]{LNR2001}: for any self-similar measure $\mu$ associated with an IFS
\[
\Psi = \{(-1)^{\epsilon_\ell} \rho x + a_\ell\}_{\ell=1}^m
\]
on $\mathbb{R}$, where $0 < \rho < 1$ and $\epsilon_\ell \in \{0,1\}$, satisfying WSC* (see Definition~4.2 in \cite{LNR2001}), $\mu$ is absolutely continuous if and only if its $L^2$-dimension equals $1$; equivalently, if the spectral radius of a certain transition matrix equals $\rho$.

Since $\beta = (\sqrt{5}+1)/2$ is a Pisot number, the IFS $\Phi$ in our example satisfies WSC*. Moreover, a direct computation using the algorithm of \cite{LNR2001}  shows that, for each $p \in [0,1]$, the transition matrix $M(p)$ corresponding to $\mu_p$ is
\[
M(p)=\begin{pmatrix}
p^2+(1-p)^2 & 2p(1-p) & 0 & 0 & 0\\
0 & 0 & (1-p)^2 & 0 & 0 \\
0 & 0 & 0 & p^2 & 2p(1-p)\\
2p(1-p) & p^2+(1-p)^2 & 0 & 0 & 0\\
0 & 0 & p(1-p) & 0 & 0
\end{pmatrix}.
\]
Hence, by the criterion of Lau, Ngai, and Rao, $\mu_p$ is absolutely continuous if and only if the spectral radius of $M(p)$ equals $\beta^{-1}$.

Since $M(p)$ is nonnegative, if its spectral radius is $\beta^{-1}$, then $\beta^{-1}$ is an eigenvalue of $M(p)$; see e.g.~\cite[Theorem~8.3.1]{HornJohnson1985}. The characteristic polynomial of $M(p)$ is
\begin{align*}
Q_p(\lambda)=\lambda [&(-2\lambda-4)p^6+(6\lambda+12)p^5+(2\lambda^2-7\lambda-13)p^4+(-4\lambda^2+4\lambda+6)p^3\\
&+(2\lambda^3+2\lambda^2-\lambda-1)p^2-2\lambda^3 p-\lambda^4+\lambda^3].
\end{align*}
Using $\beta^{-2}+\beta^{-1}-1=0$, we obtain
\[
Q_p(\beta^{-1})=\beta^{-1}(-2\beta^{-1}-4)(p-\beta^{-1})^2(p-\beta^{-2})^2\Bigl(p^2-p-\tfrac{\beta^{-1}}{2}\Bigr).
\]
Since $p \in [0,1]$, we have $Q_p(\beta^{-1})=0$ if and only if $p=\beta^{-1}$ or $p=\beta^{-2}$. In these two cases, one checks that $\beta^{-1}$ is indeed the spectral radius of $M(p)$. This completes the proof of (i).
\end{proof}

\begin{rem}
It is easily checked that the Fourier transforms of the densities $f_p$ for $p=(\sqrt{5}+1)/2$ and $p=(3-\sqrt{5})/2$ in Example~\ref{exmp-1}  have no zeros on $\R$.
\end{rem}

 Example~~\ref{exmp-1} has an interesting application in harmonic analysis.

\begin{defn}
\label{defn-Fourier frame}
A Borel probability measure $\nu$ on $\R$ is said to admit a {\it Fourier frame} if there exist a countable set $\Lambda\subset \R$ and two positive constants $A$ and $C$ such that
$$
A\|f\|^2_{L^(\nu)}\leq \sum_{\lambda\in \Lambda} |\langle f, e_\lambda\rangle_{L^2(\nu)}|^2\leq C \|f\|^2_{L^2(\nu)},\qquad f\in L^2(\mu),
$$
where $e_\lambda(x)=e^{2\pi i\lambda x}$. If, in addition, $A=C=1$, then $\nu$ is called a {\it spectral measure}.
\end{defn}

\begin{rem}
Notice that the density of the self-similar measure \(\mu_p\) constructed in Example~\ref{exmp-1}, with \(p=(\sqrt{5}-1)/2\), is piecewise constant but not constant. This provides an affirmative answer to a question of Chun-Kit Lai (via private communication), namely, whether there exists an absolutely continuous self-similar measure \(\nu\) on \(\mathbb{R}\) whose density is non-constant on its support but takes values in an interval \((c_1,c_2)\) with \(c_1,c_2>0\) for \(\nu\)-almost every point. As an application, by applying a result of Chun-Kit Lai \cite[Theorem~1.1]{Lai2011}, we see that
the $L^2$ space of the measure $\mu_{(\sqrt{5}-1)/2}$ admits a Fourier frame. However, by \cite[Corollary~1.4]{DutkayLai2014},  the measure $\mu_{(\sqrt{5}-1)/2}$ is not a spectral measure since its density on its support is not constant almost everywhere. To the best of our knowledge, this is the first example of a self-similar measure that admits a Fourier frame but is not a spectral measure.

\end{rem}

Next we provide an example of a homogeneous integral IFS on $\R$ whose attractor is an interval, but for which all the associated self-similar measures are non-Rajchman.

\begin{exmp}
\label{exmp-nonRajchman}
Let $\Phi= \{ \tfrac{1}{3}x,\ \tfrac{1}{3}x + 1,\ \tfrac{1}{3}x + 3,\ \tfrac{1}{3}x + 4 \}$ be an IFS on $\R$ and let $K$ be its attractor. Then $K=[0,6]$, and all self-similar measures associated with $\Phi$ are non-Rajchman and hence singular.
\end{exmp}

\begin{proof}[Justification]
One checks directly that
\[
[0,6] = [0,2] \cup [1,3] \cup [3,5] \cup [4,6]=\bigcup_{\ell=1}^4 \phi_\ell([0,6]),
\]
which implies  that $K = [0,6]$.

Next, let  $\pb = (p_1,p_2,p_3,p_4)$ be a probability vector, and let $\mu$ be the self-similar measure associated with $\Phi$ and $\pb$. We show below that $\mu$ is non-Rajchman.

Suppose, on the contrary, that $\mu$ is Rajchman. Recall that for $\xi \in \mathbb{R}$,
\[
\widehat{\mu}(\xi) = \prod_{n=0}^{\infty} H(3^{-n}\xi),
\]
where
\[
H(\xi) := p_1 + p_2 e^{2\pi i \xi} + p_3 e^{2\pi i (3\xi)} + p_4 e^{2\pi i (4\xi)}.
\]
By Lemma~\ref{lem-2.2}, for any $m \in \mathbb{N}$, there exists $\tau_m \in \mathbb{Z}$ such that
\[
H(m 3^{-\tau_m}) = 0.
\]
In particular, there exists an integer $\tau$ such that $H(3^{-\tau})=0$.

Note that $H(3^j)=1$ for all nonnegative integers $j$, so $\tau \geq 1$.
Moreover,
\[
H(3^{-1}) = (p_1 + p_3) + (p_2 + p_4)e^{2\pi i/3} \neq 0,
\]
so $\tau \neq 1$, and hence $\tau \geq 2$. It follows that the imaginary parts of
\[
e^{2\pi i 3^{-\tau}}, \quad e^{2\pi i \cdot 3\cdot 3^{-\tau}}, \quad \text{and} \quad e^{2\pi i \cdot 4\cdot 3^{-\tau}}
\]
are strictly positive. Since all $p_\ell$ are nonnegative, from
\[
0=H(3^{-\tau})=p_1 + p_2 e^{2\pi i 3^{-\tau}} + p_3 e^{2\pi i \cdot 3\cdot 3^{-\tau}} + p_4 e^{2\pi i \cdot 4\cdot 3^{-\tau}} ,
\]
it follows that $p_2 = p_3 = p_4 = 0$, and hence $p_1 = 1$, which is impossible, since the above sum would then be equal to $1$.
\end{proof}

Our last example shows that a generailized IFS of type I may also admit an integral factor.

\begin{exmp}
\label{exmp-8.5}
Let
$$\Phi=\left\{\begin{pmatrix} \beta^{-1} & 0\\
0& 1/2 \end{pmatrix}\begin{pmatrix} x\\
y \end{pmatrix}+a_\ell    \right\}_{1\leq \ell\leq 3}  $$
be an IFS on $\R^2$
with $\beta=\displaystyle\frac{\sqrt{5}+1}{2}$, $a_1=\begin{pmatrix}0\\ 0\end{pmatrix}$, $a_2=\begin{pmatrix}0\\ 1\end{pmatrix}$ and $a_3=\begin{pmatrix}1\\ 0\end{pmatrix}$.
Clearly, $\Phi$ is a generailized IFS of type I but admits an integral factor.
\end{exmp}

\end{document}